\documentclass[11pt]{article}
\usepackage{amsmath,latexsym,amsxtra,enumerate, color, cancel}
\usepackage[english]{babel}
\usepackage[usenames,dvipsnames,svgnames,table]{xcolor}
\usepackage{times, amsfonts, amssymb, amsthm, amscd, graphicx,stmaryrd, bbm, setspace,tocloft}
\usepackage[bottom=1in,top=1in, right=1in, left=1in]{geometry}
\usepackage{amsfonts, amssymb, amsxtra,amscd, enumerate,color, cancel}
\usepackage{bbm, dsfont,latexsym} 
\usepackage{setspace}
\RequirePackage[colorlinks,citecolor=blue,urlcolor=blue]{hyperref}
\usepackage{hyperref}
\allowdisplaybreaks

\theoremstyle{plain}
\newtheorem{proclaim}{Proposition}[section]
\newtheorem{lemma}[proclaim]{Lemma}
\newtheorem{theorem}[proclaim]{Theorem}
\newtheorem{corollary}[proclaim]{Corollary}
\theoremstyle{definition}
\newtheorem{definition}[proclaim]{Definition}

\theoremstyle{remark}
\newtheorem{remark}[proclaim]{Remark}
\numberwithin{equation}{section}
\theoremstyle{assumption}
\newtheorem{assumption}[proclaim]{Assumption}

\newcommand{\IN}[1]{\I_{\set{#1}}}
\newcommand{\Var}{\textrm{Var}}

\newcommand{\p}{\mathbb{P}}
\newcommand{\abs}[1]{\left|{#1} \right|}

\newcommand{\R}{\mathbb{R}}

\newcommand{\N}{\mathbb{N}}

\newcommand{\E}{\mathbb{E}}
\newcommand{\F}{\mathcal{F}}
\newcommand{\Z}{\mathbb{Z}}
\newcommand{\dist}{\textrm{dist}}
\newcommand{\set}[1]{\left \{#1 \right \}}
\newcommand{\floor}[1]{\lfloor #1 \rfloor}
\newcommand{\B}{\mathcal{B}}
\newcommand{\pr}[1]{\left( #1 \right)}
\newcommand{\pl}[1]{\left[ #1 \right]}
\newcommand{\I}{\mathds{1}}
\newcommand{\tn}{\theta_n}
\newcommand{\nt}{n_{\theta}}
\newcommand{\jtn}{J_{n}}
\newcommand{\C}{\mathbf{C}}
\newcommand{\tj}{\tau_J}

\newcommand{\tl}{\tau_{L}}
\newcommand{\pp}[1]{\p \pr{ #1 }}

\newcommand{\var}{\textrm{Var}}

\newcommand{\dox}{d_{o}(x)}
\newcommand{\HH}{\mathcal{H}}
\newcommand{\tnt}{\tau_{\nt}}

\newcommand{\TT}{\theta}
\newcommand{\ceiling}[1]{\lceil #1 \rceil}

\newcommand{\Hl}{\hat{H}}
\newcommand{\EB}{\mathcal{B}_{\text{exp}}}
\newcommand{\Calpha}{C^{0,\alpha}_{\text{loc}}}
\newcommand{\yx}{y{-}x_0}
\newcommand{\gbv}{GBV_{\text{exp}}}

\newcommand{\cae}{C^{0,\alpha}_{\text{exp}}}

\makeatletter
\def\timenow{\@tempcnta\time
\@tempcntb\@tempcnta
\divide\@tempcntb60
\ifnum10>\@tempcntb0\fi\number\@tempcntb
:\multiply\@tempcntb60
\advance\@tempcnta-\@tempcntb
\ifnum10>\@tempcnta0\fi\number\@tempcnta}
\makeatother

\title{Time-dependent weak rate of convergence for functions of generalized bounded variation}

\author{ Antti Luoto}
\date{}
\begin{document}
\setcounter{tocdepth}{1}

\maketitle

\begin{abstract}
\noindent Let $W$ denote the Brownian motion. For any exponentially bounded Borel function $g$ the function $u$ defined by
$u(t,x)=  \E[g(x{+}\sigma W_{T-t})]$ is the stochastic solution of the backward heat equation with terminal condition $g$.
Let $u^n(t,x)$ denote the corresponding approximation generated by a simple symmetric random walk with time steps $2T/n$ and space steps $\pm \sigma \sqrt{T/n}$ where $\sigma > 0$. For quite irregular terminal conditions $g$ (bounded variation on compact intervals, locally H\"older continuous) the rate of convergence of $u^n(t,x)$ to $u(t,x)$ is considered, and also the behavior of the error $u^n(t,x)-u(t,x)$ as $t$ tends to $T$.\\

\end{abstract}

{\let\thefootnote\relax\footnote{{University of Jyvaskyla, Department of Mathematics and Statistics, P.~O.~Box 35, FI-40014 University of Jyvaskyla \\ \hspace*{1.5em}  {\tt antti.k.luoto{\rm@}jyu.fi}}}}
\section{Introduction} 

The objective of this paper is to study the rate of convergence of a finite-difference approximation scheme for the backward heat equation. The error analysis is carried out for a large class of exponentially bounded terminal condition functions which have bounded variation on compact intervals or which are locally H\"older continuous.

During the past decades, convergence rates of finite-difference schemes for parabolic boundary value problems have been studied with varying assumptions on the regularity of the initial/terminal condition, the domain of the solution, properties of the possible boundary data etc.~(see e.g.~\cite{Dong}, \cite{Juncosa}, \cite{Lindhagen}, \cite{Reynolds}, and \cite{Zhu}). In order to study the convergence, several techniques have been applied. Our approach is probabilistic: The solution of the PDE is represented in terms of Brownian motion, and the approximation scheme is realized using an appropriately scaled sequence of simple symmetric random walks in the same probability space, in the spirit of Donsker's theorem. The potential discontinuities of the terminal function produce error bounds which are not uniform over the time-nets under consideration, and hence the time-dependence of the error is of particular interest here.

To explain our setting in more detail, fix a finite time horizon $T > 0$, a constant $\sigma > 0$, and consider the backward heat equation
\begin{align}\label{heat-eq}
\frac{\partial}{\partial t}u  + \dfrac{\sigma^2}{2} \frac{\partial^2}{\partial x^2} u & = 0, \quad (t,x) \in [0,T) \times \R, \quad \quad u(T,x) = \, g(x), \quad x \in \R.
\end{align}
The terminal condition $g : \R \to \R$ is assumed to belong to the class $\gbv$ consisting of exponentially bounded functions that have bounded variation on compact intervals (see Definition \ref{DefinitionGBV} for the precise description of $\gbv$). The stochastic solution of the problem (\ref{heat-eq}) is given by
\begin{align}\label{heat-equation-solution}
u(t,x) := \mathbb{E}[g(\sigma W_T) | \sigma W_t = x] = \E[g(x{+}\sigma W_{T-t})], \quad (t,x) \in [0,T] \times \R,
\end{align}
where $(W_t)_{t \geq 0}$ denotes the standard Brownian motion. To approximate the solution (\ref{heat-equation-solution}), we proceed as follows. Given an even integer $n \in 2\N$, a level $z_0 \in \R$, and time and space step sizes $\delta$ and $h$, define
\begin{align*}
\mathcal{T}^n := \set{t^n_k := 2k \delta \, \big| \, 0 \leq k \leq \tfrac{n}{2}, k \in \Z}, \quad \mathcal{S}^{n}_{z_0} := \set{z_0 + 2mh \, \big| \, m \in \Z}.
\end{align*}
The finite-difference scheme we will consider is given by the following system of equations defined on grids $\mathcal{G}^{n}_{z_0} := \mathcal{T}^n \times \mathcal{S}^{n}_{z_0} \subset [0,T] \times \R$,
\begin{align}\label{fd_eq1}
\left\{ \begin{array}{l}
\dfrac{v^n(t^n_k, x) {-} v^n(t^n_{k-1},x)}{t^n_k {-} t^n_{k-1}} + \dfrac{\sigma^2}{2} \dfrac{v^n(t^n_k,x{+}2h) {-} 2v^n(t^n_k, x) + v^n(t^n_k,x{-}2h)}{(2h)^2} = 0, \\
v^n(T, \, \cdot \,) = g.
\end{array}
\right.
\end{align}
Letting $\delta := \frac{T}{n}$ and $h := \sigma \sqrt{\smash[b]{\tfrac{T}{n}}}$, the system (\ref{fd_eq1}) can be rewritten in an equivalent form as 
\begin{align}\label{simpler-finite-difference}
\left\{ \begin{array}{ll} 
v^n(t^n_{k-1}, x) = \frac{1}{4}\big[v^n(t^n_k,x{+}2h) + 2v^n(t^n_k,x) + v^n(t^n_k, x{-}2h)\big],\\
v^n(T,\, \cdot \,) = g.
\end{array}
\right.
\end{align}
This scheme is explicit: Given the set of terminal values $\set{g(x) \, \big| \, x \in \mathcal{S}^{n}_{z_0}}$, the solution $u^n$ of (\ref{simpler-finite-difference}) is uniquely determined by a backward recursion. We extend the function $v^n$ in continuous time by letting
\begin{align}\label{cont_time_extension}
v^n(t,x) := v^n(t^n_k, x) \quad \text{ for } \quad t \in [t^n_k, t^n_{k+1}), \, 0 \leq k < \tfrac{n}{2},
\end{align}
and consider the error $\varepsilon_n(t,x)$ on $(t,x) \in [0,T) \times \mathcal{S}^{n}_{z_0}$, which is given by
\begin{align}\label{weakerrorforg}
\varepsilon_n(t,x) := v^n(t,x) - u(t,x).
\end{align}
The main result of this paper, Theorem \ref{THE-THEOREM} (A) states that for some constant $C > 0$ depending only on $g$,
\begin{align}\label{the-main-statement}
\abs{\varepsilon_n(t,x)} \leq \dfrac{C \psi(x)}{\sqrt{n(T-t)}} \IN{t \neq t^n_k} + \dfrac{C \psi(x)}{\sqrt{n(T-t^n_k)}}, \quad \quad (t,x) \in [t^n_k,t^n_{k+1}){\times}\mathcal{S}^{n}_{z_0}, \, 0 \leq k < \tfrac{n}{2},
\end{align}
where $\psi(x) = \psi(|x|,g,\sigma,T) > 0$ depends on the properties of $g$ and will be given explicitly in Section \ref{settingsection}.

Inequality (\ref{the-main-statement}) suggests that the convergence is not uniform in $(t,x)$. However, if we consider uniform convergence on any compact subset of $[0,T) \times \R$, the rate for this class is $n^{-1/2}$, and it will be shown in Subsection \ref{subseq-sharpness} that this rate is also sharp. 

Already in 1953, Juncosa \& Young \cite{Juncosa} considered a finite difference approximation of the forward heat equation on a semi-infinite strip $[0,\infty) \times [0,1]$, where the initial condition was assumed to have bounded variation. Using Fourier methods, they proved in \cite[Theorem 7.1]{Juncosa} that the error is $O(n^{-1/2})$ uniformly on $[t,\infty)\times[0,1]$ for any fixed $t > 0$, but they did not study the blow-up of the error as $t \downarrow 0$. Notice that the right-hand side of (\ref{the-main-statement}) corresponding to the backward heat equation undergoes a blow-up as $t \uparrow T$. Such an explosion does not occur, however, if the terminal condition $g$ is H\"older continuous. Indeed, suppose that $g$ belongs to the class $\cae$ (see Definition \ref{caedef}), which consists of exponentially bounded, locally $\alpha$-H\"older continuous functions. By Theorem \ref{THE-THEOREM} (B), there exists a constant $C > 0$ depending only on $g$ such that
\begin{align}\label{the-secondary-statement}
\abs{\varepsilon_n(t,x)} \leq \dfrac{C \psi(x)}{n^{\frac{\alpha}{2}} (T-t^n_k)^{\frac{\alpha}{2}}}, \quad \quad (t,x) \in [t^n_k,t^n_{k+1}){\times}\mathcal{S}^{n}_{z_0}, \, \, 0 \leq k < \tfrac{n}{2},
\end{align}
where the function $\psi(x) = \psi(|x|,g, \sigma,T) > 0$ plays a similar role as in (\ref{the-main-statement}). Note that the error is bounded in $t$ for fixed $n \in 2\N$, since $(T - t^n_k)^{-\frac{\alpha}{2}} \leq (\frac{n}{2T})^{\frac{\alpha}{2}}$ for each $t_k^n < T$.

Dong \& Krylov (2005) \cite{Dong} considered the convergence of a finite-difference scheme for a very general parabolic PDE. By specializing their result \cite[Theorem 2.12]{Dong} to the setting of this paper, the error is seen to converge \emph{uniformly} in $(t,x)$ with rate $n^{-1/4}$ for a bounded and Lipschitz continuous terminal condition, in contrast to the time-dependent rate $n^{-1/2}$ implied by (\ref{the-secondary-statement}). In fact, an analogous uniform rate $n^{-\alpha/4}$ can be shown for the class $\cae$ in our setting. The proof is sketched in Remark \ref{one-fourth-rate}.

The main result of this paper is derived using the following probabilistic approach. Let $(\xi_i)_{i=1,2, \dots}$ be a sequence of i.i.d.~Rademacher random variables, and define
\begin{align}\label{rw-approximation}
u^n(t,x) := \E\big[g\big(x{+} \sigma W^{n}_{T-t}\big)\big], \quad (t,x) \in [0,T]\times\R,
\end{align}
where $(W^{n}_{t})_{t \in [0,T]}$ is the random walk given by
\begin{align}\label{RandomWalkW}
W^{n}_{t} := \sqrt{\tfrac{T}{n}}\sum_{i=1}^{2\ceiling{\frac{t}{2T/n}}} \xi_i, \quad t \in [0,T]
\end{align}
($\ceiling{ \cdot }$ denotes the ceiling function). The key observation is that the function $u^n$, when restricted to $\mathcal{G}^{n}_{z_0}$, is the unique solution of (\ref{simpler-finite-difference}) for every $z_0 \in \R$; relation (\ref{cont_time_extension}) also holds for $u^n$ by definition. 
Moreover, since the random walk $(W^{n}_{t})_{t \in [0,T]}$ influences the value of $u^n$ only through its distribution, we may consider a special setting where the Rademacher variables $\xi_1, \xi_2, \dots$ are chosen in a suitable way. Defining these variables as the values of the Brownian motion $(W_t)_{t \geq 0}$ sampled at certain stopping times (see Subsection \ref{firstexittimes}) enables us to apply techniques from stochastic analysis for the estimation of the error (\ref{weakerrorforg}) where $v^n = u^n$.

The above procedure was used in J. B.~Walsh (2003) \cite{Walsh} (cf.~Rogers \& Stapleton (1997) \cite{RogersStapleton}) in relation to a problem arising in mathematical finance. More precisely, the weak rate of convergence of European option prices given by the binomial tree scheme (Cox-Ross-Rubinstein model) to prices implied by the Black-Scholes model is analyzed (cf.~Heston \& Zhou (2000) \cite{Heston}). A detailed error expansion is presented in \cite[Theorem 4.3]{Walsh} for terminal conditions belonging to a certain class of piecewise $C^2$ functions. Using similar ideas, we complement this result by considering a large class of functions which contains the class considered in \cite{Walsh}. Moreover, instead of considering the error only at time $t = 0$, we derive time-dependent error bounds. Finally, we close two gaps in the proof done by Walsh. The first concerns an asymptotic estimate \cite[relation (20)]{Walsh} related to the first exit time of a Brownian bridge. This estimate is discussed in detail in a forthcoming paper \cite{GLS} and certain results therein are generalized in Subsection \ref{condit_probab_subsect}. The second gap concerns the estimate \cite[Proposition 11.2 $(iv)$]{Walsh} for which a proof is given in Section \ref{MomenttiestimaatitJiille}.

It is argued in \cite[Sections 7 and 12]{Walsh} that the rate remains unaffected if the geometric Brownian motion is replaced with a Brownian motion, and the binomial tree is replaced with a random walk. It seems plausible that also our time-dependent results in the Brownian setting can be transferred into the geometric setting with essentially the same upper bounds.
\vspace{1.5mm}

The paper is organized as follows. In Section 2 we introduce the notation, recall the construction of a simple random walk using first hitting times of the Brownian motion, and formulate the main result Theorem \ref{THE-THEOREM}. Using the sequence of stopping times, we split the error (\ref{weakerrorforg}) into three parts. Estimates for the \emph{adjustment error}, the \emph{local error}, and the \emph{global error} are derived in Sections 3--5, respectively. Section 6 contains a collection of moment estimates and tail behaviors of random times appearing in the description of the local and the global error. 

\section{The setting and the main result}\label{settingsection}
\subsection{Notation related to the random walk}\label{firstexittimes}
Consider a standard Brownian motion $(W_t)_{t \geq 0}$ on a stochastic basis $(\Omega, \F, \p, (\F_t)_{t \geq 0})$, where $(\F_t)_{t \geq 0}$ stands for the natural filtration of $(W_t)_{t \geq 0}$. We also let $(X_t)_{t \geq 0} := (\sigma W_t)_{t \geq 0}$,
where $\sigma > 0$ is a given constant. By $\tau_{(-h,h)}$ we denote the first exit time of the process $(X_t)_{t \geq 0}$ from the open interval $(-h, h)$,
\begin{align*}
\tau_{(-h,h)} := \inf \set{t \geq 0: \abs{X_t} = h} = \inf \set{t \geq 0: \abs{W_t} = h/\sigma}, \quad h > 0.
\end{align*}
The random variable $\tau_{(-h,h)}$ is a $(\F_t)_{t \geq 0}$-stopping time and its moment-generating function is given by
\begin{align}\label{laplacecoshcos}
\E\big[{e^{\lambda \tau_{(-h,h)}}}\big] = \left\{ \begin{array}{ll}
\mathrm{cosh}(h\sqrt{2|\lambda|}/\sigma)^{-1}, \quad & \lambda \leq 0,\\
\cos(h\sqrt{2\lambda}/\sigma)^{-1}, \quad & \lambda \in (0, \frac{\pi^2}{8}\frac{\sigma^2}{h^2}).
\end{array}\right.
\end{align}
It follows that the exit time $\tau_{(-h,h)}$ has finite moments of all orders, and for every $K \in \N$ there exists a constant $C_K > 0$ such that
\begin{align}\label{tauhoomoments}
\E \big[\tau_{(-h,h)}^{K}\big] = C_K (h/\sigma)^{2K}.
\end{align}
In particular, $C_1 = 1$ and $C_2 = 5/3$. For relations (\ref{laplacecoshcos}) and (\ref{tauhoomoments}), see \cite[Proposition 11.1]{Walsh}.

In order to represent the error (\ref{weakerrorforg}), we construct a random walk
on the space $(\Omega, \F, \p, (\F_t)_{t \geq 0})$. Following \cite{Walsh}, we define 
\begin{align}\label{stopping-time-sequence}
\tau_0 := 0 \quad \text{ and } \quad \tau_{k} = \tau_k(h) := \inf \set{t \geq \tau_{k-1}: \abs{X_t - X_{\tau_{k-1}}} = h}
\end{align}
recursively for $k = 1,2, \dots$. Then $\tau_k$ is a $\p$-a.s.~finite $(\F_t)_{t \geq 0}$-stopping time for all $k \geq 0$, and
the process $(X_{\tau_k})_{k=0,1, \dots}$ is a symmetric simple random walk on $\Z^h := \set{mh: m \in \Z}$.
For every integer $k \geq 1$, we also let
\begin{align*}
\Delta \tau_k & := \tau_{k} - \tau_{k-1}  \quad \text{ and } \quad \Delta X_{\tau_k} := X_{\tau_k} - X_{\tau_{k-1}}.
\end{align*}
The strong Markov property of $(X_t)_{t \geq 0}$ implies that
$(\Delta \tau_k, \Delta X_{\tau_k})_{k=1,2, \dots}$ is an i.i.d.~process such that, for each $k \geq 1$, we have $\p(\Delta X_{\tau_k} {=} \pm h) = 1/2$,
$$(\Delta \tau_k, \Delta X_{\tau_k}) \stackrel{d}{=} (\tau_{(-h,h)}, X_{\tau_{(-h,h)}}), \quad \text{ and } \quad (\Delta \tau_k, \Delta X_{\tau_k}) \text{ is independent of } \F_{\tau_{k-1}+}.$$
Moreover, as shown in \cite[Proposition 1]{RogersStapleton}, the increments $\Delta X_{\tau_1}$ and $\Delta \tau_1$ are independent. Consequently, the \emph{processes} $(\Delta {\tau_k})_{k=1,2, \dots}$ and $(\Delta X_{\tau_k})_{k=1,2, \dots}$ are independent (see also \cite[Proposition 11.1]{Walsh} and \cite[Proposition 2.4]{LambertonRogers}).

We deduce, in particular, that for all $N \geq 1$ the random variable  $X_{\tau_N}$ is distributed as $h \sum_{k=1}^{N} \xi_{k}$,
where $(\xi_k)_{k=1,2, \dots}$ is an i.i.d.~sequence of Rademacher random variables. Therefore, for $W^{n}_{T-t}$ defined in (\ref{RandomWalkW}), we have the equality in law
$$X_{\tau_N} \stackrel{d}{=} \sigma W^{n}_{T-t} \quad \text{ provided that } \,(h,N) = \big(\sigma \sqrt{\smash[b]{\tfrac{T}{n}}}, 2\ceiling{\tfrac{T{-}t}{2T/n}}\big).$$
Note that in this case the sequence of stopping times $(\tau_k)_{k=0,1, \dots}$ (\ref{stopping-time-sequence}) depends on $n$ via $h = h(n)$.

The error (\ref{weakerrorforg}) will be split into three parts, where each of these parts will take into account different properties of the given function $g$. For this purpose, let us introduce some more notation. For given $n \in 2\N$ and $t \in [0,T)$, we let
\begin{align}\label{deftnnt}
\tn := \frac{\nt T}{n}, \quad \text{ where } \quad \nt := 2 \left\lceil \frac{T{-}t}{2T/n} \right\rceil \in \set{2,4, \dots, n}.
\end{align}
By definition, $\theta_n$ is the smallest multiple of $\frac{2T}{n}$ greater than or equal to $T{-}t$. It is clear that
$$0 \leq \theta_n - (T{-}t) \leq \frac{2T}{n} \quad \text{ and } \quad \theta_n \downarrow T-t \quad \text{ as } \, n \to \infty.$$
The connection between lattice points $t^n_k = \frac{2kT}{n} \in \mathcal{T}^n$ and the time instant $\theta_n \in (0,T]$ is explained by
\begin{align}\label{theta_lattice_relation}
t \in [t_k^n, t_{k+1}^n) \quad \text{ if and only if } \quad \theta_n = T-t_k^n, \quad 0 \leq k \leq \tfrac{n}{2}{-}1.
\end{align}
\subsection{The function classes under consideration}
The error (\ref{weakerrorforg}) will be estimated for functions $g$ belonging to the class $\gbv$ or $\cae$ which are introduced below. 
Both of these classes are contained in the class of exponentially bounded Borel functions.
\begin{definition}[The class $\EB$]\label{expboundeddef} A function $g: \R \to \R$ is said
to be exponentially bounded if there exist constants $A,b \geq 0$ such that 
\begin{align}\label{exponentialboundednesscondition}
\abs{g(x)} \leq A e^{b\abs{x}} \quad \text{ for all } x \in \R.
\end{align}
The class of all Borel functions with the above property will be denoted by $\EB$.
\end{definition}

The function class $\gbv$ generalizes functions of bounded variation (which are bounded) by allowing exponential growth.
See \cite{Avikainen} and Appendix \ref{appendixGBV} for more information on this topic. To introduce this class, let us recall
\begin{definition}[{\cite[Definition 3.2]{Avikainen}}]\label{class-of-set-functions}
Denote by $\mathcal{M}$ the class of all set functions 
$$\mu : \set{G \in \B(\R): G \text{ is bounded}} \to \R$$
that can be written as a difference of two measures $\mu^{1}, \mu^{2} : \B(\R) \to [0,\infty]$ such that $\mu^1(K) , \mu^2(K) < \infty$
for all compact sets $K \in \B(\R)$.
\end{definition}
Below it is understood that $[a,b) = \emptyset$ whenever $a \geq b$.
\begin{definition}[The class $\gbv$]\label{DefinitionGBV}
Denote by $\gbv$ the class of functions $g : \R \to \R$ which can be represented as
\begin{align}\label{GBVdefrep}
g(x) = c + \mu([0,x)) - \mu([x,0)) + \sum_{i=1}^{\infty} \alpha_i \IN{x_i}(x), \quad x \in \R,
\end{align}
where $c \in \R$ is a constant, $\mu \in \mathcal{M}$, and $\mathcal{J} = (\alpha_i, x_i)_{i=1,2, \dots} \subset \R^2$ is a countable set such that
$x_i \neq x_j$ whenever $i \neq j$. In addition, we require that for some constant $\beta \geq 0$,
\begin{align}\label{bump-function-condition}
\int_{\R} e^{-\beta\abs{x}} d|\mu|(x) + \sum_{i=1}^{\infty} |\alpha_i| e^{-\beta \abs{x_i}} < \infty.
\end{align}
\end{definition}
The following remark provides some examples of functions belonging to this class.
\begin{remark}[Examples of functions contained in $\gbv$]
\begin{itemize}
\item[]
\item[$(i)$:] Every polynomial belongs to the class $\gbv$ (see Remark \ref{polyonomials_are_included}).
\item[$(ii)$:] Each increasing (resp.~decreasing) function $g \in \EB$ belongs to $\gbv$. 
\item[$(iii)$:] Each convex (resp.~concave) function $g \in \EB$ belongs to $\gbv$.
\item[$(iv)$:] $\mathcal{K}_{\text{exp}} \subset \gbv$, where $\mathcal{K}_{\text{exp}}$ is the class considered in
Walsh \cite{Walsh} (pp.~340, 345--346, and 348) of functions $g: \R \to \R$ satisfying the below criteria:
\begin{itemize}
\item[$\bullet$] $g, g'$, and $g''$ belong to $\EB$
\item[$\bullet$] $g, g'$, and $g''$ have at most finitely many discontinuities and no oscillatory discontinuities
\item[$\bullet$] $g(x) = \frac{1}{2}(g(x+) + g(x-))$ at each point $x \in \R$.
\end{itemize}
\end{itemize}
\end{remark}

\vspace{0,25cm}
Let us finally introduce the class $\cae$ of exponentially bounded locally H\"older continuous functions. See Subsection \ref{Holdersection} for some facts about this class.
\begin{definition}[The class $\cae$]\label{caedef}
Denote by $\cae$ the class of all functions $g : \R \to \R$ for which there exist constants $A, \beta \geq 0$ such that for all $R > 0,$
\begin{align}\label{superholder-condition}
\sup_{x,y \in [-R,R], \, x \neq y} \frac{\abs{g(x) - g(y)}}{\abs{x-y}^{\alpha}} \leq A e^{\beta R}.
\end{align}
\end{definition}

\subsection{The main result}
The following theorem is the main result of this paper.
\begin{theorem}\label{THE-THEOREM}
Let $n \in 2\N$, and let $u$ and $u^n$ be the functions introduced in (\ref{heat-equation-solution}) and (\ref{rw-approximation}).
\begin{itemize}
\item[\emph{(A)}] Suppose that $g \in \gbv$ is a function given by (\ref{GBVdefrep}) and that $\beta \geq 0$ is as in (\ref{bump-function-condition}). Then, for all $(t,x) \in [0,T){\times}\R$, 
\begin{align*}
(i) \quad & \abs{u^n(t,x) - u(t,x)} \quad \, \leq \, \quad \frac{C_{\beta, \sigma, T}}{\sqrt{n (T-t)}}e^{\beta \abs{x}}, \, \, \quad \quad t \neq t^n_k, \, \, 0 \leq k < \tfrac{n}{2},\\
(ii)\quad & \abs{u^n(t^n_k,x) - u(t^n_k,x)} \, \leq \, \frac{C_{\beta, \sigma, T}}{\sqrt{n(T-t^n_k)}}e^{\beta \abs{x}}, \quad \quad \quad \quad \quad 0 \leq k < \tfrac{n}{2},
\end{align*}
where $C_{\beta, \sigma, T} := C\sqrt{T} e^{3 \beta^2 \sigma^2 T}$ and $C > 0$ is a constant depending only on $g$.
\item[\emph{(B)}] Suppose that the function $g \in \cae$ and that $\beta \geq 0$ is as in (\ref{superholder-condition}). Then, for all $(t,x) \in [0,T){\times}\R$,
\begin{align*}
(iii) \quad \abs{u^n(t,x) - u(t,x)} \leq \frac{C_{\beta, \sigma,T}}{n^{\frac{\alpha}{2}}(T-t^n_k)^{\frac{\alpha}{2}}}e^{(\beta + 1)\abs{x}}, \, \quad t \in [t^n_k, t^n_{k+1}), \, 0 \leq k < \tfrac{n}{2},
\end{align*} 
where $C_{\beta, \sigma,T} := (1+T)(2+\sigma)Ce^{4(\beta + 1)^2 \sigma^2 T}$ and $C > 0$ is a constant depending only on $g$.
\end{itemize}
\end{theorem}
\begin{remark}
Properties of the error bounds in (A) and (B) were already discussed in Section 1. Here we only point out that in general these error bounds grow exponentially as functions of $x$. A uniform bound w.r.t.~$x$ can be shown under additional assumptions: For $g \in \gbv$, it is sufficient that $g$ satisfies the condition (\ref{bump-function-condition}) with $\beta = 0$. For $g \in \cae$, it suffices to assume that $g$ is bounded and satisfies (\ref{superholder-condition}) with $\beta = 0$.
\end{remark}
\begin{proof}[Proof of Theorem \ref{THE-THEOREM}] Following \cite{Walsh}, we define an auxiliary random variable $J_n$ on $(\Omega, \F, \p, (\F_t)_{t \geq 0})$ by
\begin{align}\label{defofJ}
J_n(\omega) := \inf \{2m \in 2\N: \tau_{2m}(\omega) > \theta_n\},
\end{align}
where we assume that the step size related to $(\tau_k)_{k=0,1, \dots}$ is $h = \sigma \sqrt{\smash[b]{\frac{T}{n}}}$. By definition, $J_n$ is the index of the first even stopping time $\tau_0, \tau_2, \dots$ exceeding the value $\theta_n$. It holds that $J_n$ is a stopping time w.r.t.~$(\F_{\tau_k})_{k=0,1, \dots}$. Moreover, $\tau_{J_n}$ is a stopping time w.r.t.~$(\F_t)_{t \geq 0}$, and both $J_n$ and $\tau_{J_n}$ are $\p$-a.s.~finite.
The error $\varepsilon_{n}(t,x)$ given by (\ref{weakerrorforg}) is then decomposed as follows:
\begin{align}\label{varepsilon_n_splitting}
\varepsilon_{n}(t,x) = \varepsilon^{\text{glob}}_{n}(t,x) + \varepsilon^{\text{loc}}_{n}(t,x) + \varepsilon^{\text{adj}}_n(t,x),
\end{align}
where
\begin{align}
\varepsilon^{\text{glob}}_{n}(t,x) & := \E[g(x{+}X_{\tau_{\nt}})-g(x {+} X_{\tau_{J_n}})], \quad \quad \, (\text{''the global error''}) \label{errorglobal}\\
\quad \varepsilon^{\text{loc}}_n(t,x) & := \E[g(x{+}X_{\tau_{J_n}})-g(x{+}X_{\tn})], \quad \quad \, \, (\text{''the local error''})\label{errorlocal}\\
\varepsilon^{\text{adj}}_n(t,x) & := \E\pl{g(x{+}X_{\tn})- g(x{+}X_{T-t})}. \quad \, \,\, \,\, \,(\text{''the adjustment error''}) \label{erroradjustment}
\end{align}
The adjustment error is a consequence of the fact that the approximation $u^n(t,x)$ is constant 
in $t$ on intervals of length $\frac{2T}{n}$, while $t \mapsto u(t,x)$ is continuous. The remaining two parts of the error appear because the construction of the simple random walk uses the Brownian motion sampled at a stopping time $\tau_{\nt}$ which can be larger or smaller than $\theta_n$. The local error is influenced by the smoothness properties of the terminal condition $g$, while for the global error only integrability properties of $g$ are needed.

Assume that $0 \leq k < \frac{n}{2}$ is the integer for which $t \in [t^n_k, t^n_{k+1})$ holds. 

(A):  There exists a constant ${A = A(\beta) \geq 0}$ such that $\abs{g(x)} \leq A e^{\beta \abs{x}}$ for all $x \in \R$. Indeed, (\ref{exponentialboundednesscondition}) is satisfied for a function $g$ given by \eqref{GBVdefrep} by taking $b = \beta$ and $A$ to be equal to the sum of $\abs{c}$ and the left-hand side of (\ref{bump-function-condition}). Hence, by Propositions \ref{ADJPROP} and \ref{globerrorfinally} and Corollary \ref{CorollaryGBVlocal}, there exists a constant $C > 0$ such that
\begin{align*}
\abs{\varepsilon_{n}(t,x)} & \leq C e^{\beta \abs{x} + 3 \beta^2 \sigma^2 T} \Big(\frac{\sqrt{T}}{\sqrt{n(T-t)}}\I_{\{t \neq t^n_k\}} + \frac{\sqrt{T}}{\sqrt{n(T-t^n_k)}} + \frac{T}{n(T-t^n_k)}\Big).
\end{align*}
It remains to observe that since $\sqrt{n(T-t^n_k)} \geq \sqrt{2T}$ for all integers $0 \leq k < \frac{n}{2}$, it holds
$$\frac{T}{n(T-t^n_k)} \leq \frac{\sqrt{T}}{\sqrt{n(T-t^n_k)}} \leq \frac{\sqrt{T}}{\sqrt{n(T-t)}}.$$

(B): Given a constant $\delta > 0$, by assumption, we can derive the exponential bound
\begin{align*}
\abs{g(x)} \leq A|x|^{\alpha}e^{\beta|x|} + \abs{g(0)} \leq C e^{(\beta+\delta)\abs{x}}, \quad x \in \R,
\end{align*}
for some constant $C > 0$. For simplicity, let us choose $\delta = 1$. Consequently, by Propositions \ref{ADJPROP} and \ref{globerrorfinally} (put $b = \beta+\delta$), and Corollary \ref{LocalErrorHoldCor}, we find another constant $\tilde{C} > 0$ such that
\begin{align*}
\abs{\varepsilon_{n}(t,x)} & \leq \tilde{C} e^{(\beta + 1)\abs{x} +4(\beta + 1)^2 \sigma^2 T} \pr{\frac{\sigma^{\alpha} T^{\alpha/2}}{n^{\alpha/2}} + \frac{T}{n(T-t^n_k)}}.
\end{align*}
The claim follows, since $\big(\frac{T}{n(T-t^n_k)}\big)^{\gamma} \leq \big(\frac{T}{n(2T/n)}\big)^{\gamma} \leq 1$ for all $\gamma \in [0,1]$, and thus
\begin{align*}
\frac{\sigma^{\alpha} T^{\alpha/2}}{n^{\alpha/2}} + \frac{T}{n(T-t^n_k)} \leq \frac{\sigma^{\alpha} T^{\alpha/2}}{n^{\alpha/2}} + \pr{\frac{T}{n(T-t^n_k)}}^{\alpha/2} \leq \frac{(T^{\alpha/2}+\sigma^{\alpha}T^{\alpha})}{n^{\alpha/2}(T-t^n_k)^{\alpha/2}} \leq \frac{(1+T)(2+\sigma)}{n^{\alpha/2}(T-t^n_k)^{\alpha/2}}.
\end{align*}
\end{proof}
\begin{remark}\label{one-fourth-rate}
For $g \in \cae$, there exists a constant $C = C(A, \sigma, T) > 0$ such that for all $x \in \R$,
\begin{align}\label{the-rate-1-4}
\sup_{t \in [0,T)} \abs{u^n(t,x) - u(t,x)} \leq \frac{C}{n^{\frac{\alpha}{4}}} e^{4 \beta \abs{x} + 8 \beta^2 \sigma^2 T},
\end{align}
where $A, \beta \geq 0$ are as in (\ref{superholder-condition}). Hence, we get the uniform rate $n^{-\alpha/4}$ instead of the time-dependent rate $n^{-\alpha/2}$ implied by Theorem \ref{THE-THEOREM} (B). Note that for $g \in \cae$, the time-dependence of the error bound in Theorem \ref{THE-THEOREM} (B) is caused solely by the global error, and it remains unclear whether the associated upper bound (\ref{glob-error-upper-bound}) can be improved using the additional information about the regularity of $g$. 

For the proof of (\ref{the-rate-1-4}), notice first that by the H\"older continuity and by H\"older's inequality,
\begin{align*}
\abs{u^n(t,x) - u(t,x)} \leq A \sigma^{\alpha} \pr{\E e^{q \beta \abs{x + \sigma W_{T-t}} + q \beta |x + \sigma W_{\tau_{\nt}}|}}^{1/q} \pr{ \E \abs{W_{T-t} - W_{\tau_{\nt}}}^{p\alpha}}^{1/p},
\end{align*}
where $p := \frac{2}{\alpha}$ and $q := \frac{p}{p-1}$. To proceed, apply Lemma \ref{expExp} $(i)$ and the fact that for some $C(T) > 0,$
\begin{align*}
\E \big|W_{T-t} - W_{\tau_{\nt}}\big|^2 = \E \big|(T{-}t) - \tau_{\nt}\big| \leq C(T)n^{-1/2},
\end{align*}
which follows from It\^o's isometry and a slight generalization of \cite[Proposition 11.1 $(iv)$]{Walsh}.
\end{remark}
\section{The adjustment error}
In this section we derive an upper bound for the adjustment error (\ref{erroradjustment}) for the classes $\gbv$ and $\cae$.
\begin{proclaim}\label{ADJPROP}
Let $n \in 2\N$.
\begin{itemize}
\item[$(i)$] Let $g \in \gbv$ and let $\beta \geq 0$ be as in (\ref{bump-function-condition}). Then, for all $(t_0,x_0) \in [0,T) {\times} \R$,
\begin{align*}
\big|\varepsilon_{n}^{\text{adj}}(t_0,x_0)\big| \leq \frac{A_{\beta} \sqrt{T}}{\sqrt{n(T-t_0)}}e^{\beta \abs{x_0} + \beta^2 \sigma^2 T} \I_{\{t_0 \neq t^n_k \, \forall \, 0 < k < \frac{n}{2}\}},
\end{align*}
where $A_\beta = \frac{e}{\sqrt{\pi}} \int_{\R} e^{-\beta |y|} d|\mu|(y)$.
\item[$(ii)$] Let $g \in \cae$ and let $A, \beta \geq 0$ be as in (\ref{superholder-condition}). Then, for all $(t_0,x_0) \in [0,T) {\times} \R$,
\begin{align*}
\big|\varepsilon_{n}^{\text{adj}}(t_0,x_0)\big| \leq \frac{2 A \sigma^{\alpha} T^{\alpha/2}}{n^{\alpha/2}} e^{\beta|x_0| + 4\beta^2\sigma^2 T} \I_{\{t_0 \neq t^n_k \, \forall \, 0 < k < \frac{n}{2}\}}.
\end{align*}
\end{itemize}
\end{proclaim}
\begin{proof}
$(i)$: Denote by $p_t$ the density of $X_t = \sigma W_t$ for $t > 0$, and consider the function
\begin{align*}
u(t,x_0) = \E[g(x_0{+}X_{T-t})] = \int_{\R} g(x_{0}{+}y) p_{T-t}(y) dy, \quad 0 \leq t < T.
\end{align*}
Fix $n \in 2\N$ and suppose that $t^n_k = \frac{2kT}{n}$ is the lattice point such that $t_0 \in [t^n_k, t^n_{k+1})$. If $t_0 = t^n_k$, (\ref{theta_lattice_relation}) implies that $\theta_n = T{-}t_0$, and thus $\varepsilon_{n}^{\text{adj}}(t_0,x_0) = 0$ by (\ref{erroradjustment}). Suppose then $t_0 \in (t^n_k, t^n_{k+1})$ and use the representation \eqref{newfRep} below for the function $g(x_0 + \, \cdot \,)$ in order to rewrite
\begin{align}\label{u-difference-gbv}
u(t^n_k, x_0) - u(t_0, x_0) & = \int_{\R} \Big(g(x_0{+}z \sqrt{T{-}t^n_k}) - g(x_0{+}z\sqrt{T{-}t_0})\Big)p_1(z) dz \nonumber\\
& = \int_{\R}\int_{[0,\infty)} \Big(\I_{(y{-}x_0,\infty)}(z \sqrt{T{-}t^n_k}) - \I_{(y{-}x_0, \infty)}(z\sqrt{T{-}t_0}) \Big)d\mu(y) p_1(z) dz \nonumber\\
& \quad - \int_{\R} \int_{(-\infty,0)} \Big(\I_{(-\infty,y{-}x_0]}(z \sqrt{T{-}t^n_k}) - \I_{(-\infty,y{-}x_0]}(z \sqrt{T{-}t_0}) \Big) d\mu(y) p_1(z) dz \nonumber\\
& =: I_1 - I_2.
\end{align}
Since $g$ is exponentially bounded, one may apply Fubini's theorem to rewrite
\begin{align}\label{estimate_for_I1_last_thing}
I_1 & = \int\displaylimits_{\R} \int\displaylimits_{[0,\infty)} \Big[ \I_{\big(\tfrac{y{-}x_0}{\sqrt{T-t^n_k}},\infty\big)}(z) - \I_{\big(\tfrac{y{-}x_0}{\sqrt{T{-}t_0}}, \infty\big)}(z) \Big] d \mu(y) p_1(z) dz = \int\displaylimits_{[0,\infty)} \int_{\frac{y{-}x_0}{\sqrt{T{-}t^n_k}}}^{\frac{y{-}x_0}{\sqrt{T{-}t_0}}} p_1(z) dz  d\mu(y).
\end{align}
The mean value theorem and the fact $\sqrt{T{-}t_0} < \sqrt{T{-}t^n_k}$ imply for arbitrary $y \in \R$ that
\begin{align*}
e^{\beta|y|} \Bigg| \int_{\frac{y{-}x_0}{\sqrt{T{-}t^n_k}}}^{\frac{y{-}x_0}{\sqrt{T{-}t_0}}} p_1(z) dz \Bigg| & \leq e^{\beta|x_0| + \beta|y-x_0|} p_1 \Big(\tfrac{|y-x_0|}{\sqrt{T{-}t^n_k}} \Big) \frac{|y-x_0|}{\sqrt{T-t^n_k}} \frac{\sqrt{T-t^n_k} - \sqrt{T-t_0}}{\sqrt{T-t_0}}\\
& \leq \frac{e^{\beta |x_0|}}{\sqrt{2\pi}} \Big( \sup_{z \in (0,\infty)} z e^{z\beta \sigma \sqrt{T} - z^2/2} \Big) \frac{\sqrt{T-t^n_k} - \sqrt{T-t_0}}{\sqrt{T-t_0}}\\
& \leq \frac{e^{1+ \beta |x_0| + \beta^2 \sigma^2 T}}{\sqrt{\pi}}  \frac{\sqrt{T}}{\sqrt{n(T-t_0)}}
\end{align*}
where the estimates $\sqrt{T-t^n_k} - \sqrt{T-t_0} \leq \sqrt{t_0-t^n_k} \leq \sqrt{2T/n}$ and 
$$\sup_{z \in (0,\infty)} z e^{z \beta \sigma \sqrt{T} - z^2/2} \leq \sup_{z \in (0,\infty)} e^{z(1+\beta \sigma \sqrt{T}) - z^2/2} \leq e^{(1+\beta \sigma \sqrt{T})^2/2} \leq e^{1+\beta^2\sigma^2T}$$
were applied. Consequently, it follows by \eqref{estimate_for_I1_last_thing} that
\begin{align}\label{outcome-I1}
|I_1| \leq \frac{e^{1+\beta |x_0| + \beta^2 \sigma^2 T}}{\sqrt{\pi}} \Big( \int_{[0,\infty)} e^{-\beta|y|} d|\mu|(y) \Big) \frac{\sqrt{T}}{\sqrt{n(T-t_0)}},
\end{align}
and an analogous computation for the integral $I_2$ yields
\begin{align}\label{outcome-I2}
|I_2| \leq \frac{e^{1+\beta |x_0| + \beta^2 \sigma^2 T}}{\sqrt{\pi}} \Big( \int_{(0,\infty)} e^{-\beta|y|} d|\mu|(y) \Big) \frac{\sqrt{T}}{\sqrt{n(T-t_0)}}.
\end{align}
Since $\big|\varepsilon^{\text{adj}}_{n}(t_0,x_0)\big| = \abs{u(t^n_k,x_0) - u(t_0,x_0)}$, relations (\ref{u-difference-gbv}), (\ref{outcome-I1}), and (\ref{outcome-I2}) imply the claim.

$(ii)$: Let $0 \leq k < \frac{n}{2}$ be such that $t_0 \in (t^n_k, t^n_{k+1})$ holds; the case $t_0 = t^n_k$ follows from (\ref{theta_lattice_relation}) and (\ref{erroradjustment}). H\"older's inequality implies that
\begin{align}\label{caezerothbound}
\big|\varepsilon^{\text{adj}}_{n}(t_0,x_0)\big| & \leq \E \abs{g(x_0{+}X_{T{-}t^n_k}) - g(x_0{+}X_{T-t_0})} \nonumber\\
& \leq A \E \pl{e^{\beta(|x_0| + |X_{T{-}t^n_k}| + |X_{T{-}t_0}|)} \abs{X_{T{-}t^n_k} - X_{T{-}t_0}}^{\alpha}} \nonumber\\
& \leq A \!\pr{\E \pl{e^{q \beta(|x_0| + |X_{T{-}t^n_k}| + |X_{T{-}t_0}|)}}}^{1/q} \pr{\E \abs{X_{T{-}t^n_k} - X_{T{-}t_0}}^{p\alpha}}^{1/p},
\end{align}
for some $p,q \in (1,\infty)$ with $\frac{1}{p}+\frac{1}{q} = 1$. The choice $p = \frac{2}{\alpha}$, $q = \frac{2}{2-{\alpha}}$ and the fact $\abs{t_0 - t^n_k} \leq \frac{2T}{n}$ yield
\begin{align}\label{caefirstbound}
\pr{\E \abs{X_{T{-}t^n_k} {-} X_{T{-}t_0}}^{p\alpha}}^{1/p} {\leq} \pr{\sigma^2 \E \abs{W_{T{-}t^n_k} {-} W_{T{-}t_0}}^2}^{\alpha/2} \leq \frac{\sigma^{\alpha} 2^{\alpha/2} T^{\alpha/2}}{n^{\alpha/2}}.
\end{align}
Moreover, for a standard normal random variable $Z$, H\"older's inequality implies that
\begin{align}\label{caesecondbound}
\E \pl{e^{q\beta(|x_0| + |X_{T{-}t^n_k}| + |X_{T{-}t_0}|)}} & \leq e^{q \beta|x_0|} \pr{\E \pl{e^{2q\beta \sigma \sqrt{T{-}t^n_k}|Z|}}}^{1/2}\pr{\E \pl{e^{2q\beta \sigma \sqrt{T{-}t_0}|Z|}}}^{1/2}\nonumber\\
& \leq 2 e^{q \beta|x_0| + 2q^2 \beta^2 \sigma^2 T}.
\end{align}
The claim then follows by (\ref{caezerothbound}), (\ref{caefirstbound}), and (\ref{caesecondbound}).
\end{proof}
\section{The local error}\label{local-error-section}
\subsection{Notation and definitions}
Suppose that $(h,\theta) \in (0,\infty) {\times} (0,T]$. The aim of this section is to derive an upper bound for the absolute value of the error
\begin{align}
\varepsilon^{\text{loc}}_{h,\theta}(g) := \E[g(X_{\tau_{J}})-g(X_{\TT})] \quad \label{localstaticerror}
\end{align}
as a function of $(h,\TT)$, where the function $g$ belongs to $\gbv$ or $\cae$. The random variable $J$ is given by
\begin{align}\label{Jwithoutn}
J = J(h,\theta) = \inf \set{2m: \tau_{2m} > \TT}.
\end{align}
Afterwards, upper bounds for the error (\ref{localstaticerror}) are derived in the dynamical setting, where the step size $h$ and the level $\theta$ will depend on $n$. Observe that $J$ agrees with $J_n$ defined in (\ref{defofJ}) for $(h,\theta) = (\sigma \sqrt{\smash[b]{\frac{T}{n}}}, \frac{2T}{n}\ceiling{\frac{T-t}{2T/n}})$.

Let us start by introducing the following notation:
\begin{align*}
\Z^h_o := \set{(2k{+}1)h: k \in \Z}, \quad \Z^h_e := \set{2kh: k \in \Z}
\end{align*}
($o$ refers to 'odd' and $e$ refers to 'even'); then $\Z^h = \Z^h_o \cup \Z^h_e$. In addition, we will abbreviate
\begin{align}
d_o(x) := \dist(x, \Z^h_o), \quad d_e(x) := \dist(x, \Z^h_e) = h - d_o(x), \quad x \in \R. \label{defofdode}
\end{align}
As in \cite{Walsh}, we project functions onto piecewise linear functions in order to compute the conditional
expectation $\E[g(X_{\tau_J})|\F_{\TT}]$.
\begin{definition}\label{Deflinearizationop} Define operators $\varPi_o$ and $\varPi_e$ acting on functions $u : \R \to \R$ by
\begin{itemize}
\item[] $\varPi_e u(x) := u(x)$ if $x \in \Z^h_e$ \, and \, $x \mapsto \varPi_e u(x)$ \, linear in \, $[2kh, (2k{+}2)h]$  \, $\forall k \in \Z$,
\item[] $\varPi_o u(x) := u(x)$ if $x \in \Z^h_o$ \, and \, $x \mapsto \varPi_o u(x)$ \, linear in \, $[(2k{-}1)h, (2k{+}1)h]$ \,$\forall k \in \Z$.
\end{itemize}
\end{definition}
The key ingredient in the estimation of the error $\varepsilon_{h,\TT}^{\text{loc}}(g)$ is the following result, which was proposed in \cite[Section 9]{Walsh}.
For the convenience of the reader, a sketch of the proof is given below. Recall Definition \ref{expboundeddef} for the class $\EB$, and denote by $\N_0 := \set{0,1,2, \dots}$ the
set of non-negative integers.
\begin{proclaim}\label{locexprepr} Let $(h,\theta) \in (0,\infty) {\times} (0,T]$ and define a random variable
\begin{align}\label{DefOfL}
L = \, L(h, \theta) := \sup\set{m \in \N_{0}: \tau_m < \TT}
\end{align}
($\tau_{L}$ is equal to the largest of the stopping times $\tau_0, \tau_1, \dots$ less than $\TT$). Then, given a function $g \in \EB$,
\begin{align}\label{sumofexpectationslocalerror}
\varepsilon^{\text{loc}}_{h,\theta}(g) & = \E\big[ \varPi_e g(X_{\TT}) - g(X_{\TT}) \big] + \E \big[\big(\varPi_o \varPi_e g(X_{\TT}) - \varPi_e g(X_{\TT})\big)\p(L \text{ even}|X_{\TT})\big].
\end{align}
\end{proclaim}
\begin{proof}
If $g \in \EB$, then also $\varPi_e g \in \EB$ and $\varPi_o \varPi_e g \in \EB$. The expectations on the right-hand side of (\ref{sumofexpectationslocalerror}) thus exist and are finite. Using the Markov property of the process $(X_t)_{t \geq 0}$, it can be shown that
\begin{align*}
& \E \big[g(X_{\tau_{J}}) \big| \F_{\TT}\big] = \varPi_e g(X_{\TT}) \quad \, \quad \p\text{-a.s. on } \{L \text{ odd}\}, \nonumber\\
& \E \big[g(X_{\tau_{J}}) \big| \F_{\TT}\big]= \varPi_o \varPi_e g(X_{{\TT}}) \quad \p\text{-a.s. on } \{L \text{ even}\},
\end{align*}
see \cite[Section 9]{Walsh}. Consequently, since $\I_{\set{L \text{ odd}}} + \I_{\set{L \text{ even}}} = 1$ $\p$-a.s.,
\begin{align*}
\E[g(X_{\tj})] & = \E \pl{ \E \pl{ g(X_{\tj}) \big| \F_{\TT}} \I_{\set{L \text{ odd}}}}
+ \E \pl{ \E \pl{ g(X_{\tj}) \big| \F_{\TT}} \I_{\set{L \text{ even}}}}\\
& = \E \pl{ \varPi_e g(X_{\TT}) \p \pr{L \text{ odd} \big| X_{\TT}}} + \E \pl{ \varPi_o \varPi_e g(X_{\TT}) \p\pr{L \text{ even} \big| X_{\TT}}}\\
& = \E\pl{\varPi_e g(X_{\TT})} + \E \pl{\pr{\varPi_o \varPi_e g(X_{\TT}) - \varPi_e g(X_{\TT})}\p(L \text{ even}|X_{\TT})}.
\end{align*}
\end{proof}
\subsection{Evaluation of the conditional probability $\p(L \text{ even}|X_{\theta})$}\label{condit_probab_subsect}
In this subsection we derive a representation for the function
\begin{align}\label{condprobfunction}
y \mapsto \p(L \text{ even}|X_{\TT} = y)
\end{align}
based on first exit time probabilities of a Brownian bridge. This representation (\ref{PropqIII}) together with the associated bounds presented in this subsection are applied in the proof of Propositions \ref{prop_lower_bound} and \ref{jtntjn2} below.
\begin{definition}[Brownian bridge]
Let $x,y \in \R$ and $l > 0$. A Gaussian process $(B^{x,l,y}_t)_{t \in [0,l]}$ with mean and covariance functions given by
\begin{align*}
\E[B^{x,l,y}_t] & = x + \tfrac{t}{l}(y - x), \quad 0 \leq t \leq l, \\
\text{Cov}(B^{x,l,y}_s, B^{x,l,y}_t) & = s\pr{1 - \tfrac{t}{l}}, \quad 0 \leq s \leq t \leq l,
\end{align*}
is called a (generalized) Brownian bridge from $x$ to $y$ of length $l$.
\end{definition}
\begin{remark}
By comparing mean and covariance functions, it is easy to verify that a Brownian bridge $(B^{x,l,y}_t)_{t \in [0,l]}$ is equal in law with the transformed processes below:
\begin{align}
& (B^{y,l,x}_{l-t})_{t \in [0,l]} \quad & (\text{'time reversal'}) \label{bridgetimereversal}\\
& (x + B_t^{0,l,y{-}x})_{t \in [0,l]} \quad & (\text{'translation'}) \label{bridgeshift}\\
& (-B^{-x,l,-y})_{t \in [0,l]} \quad & (\text{'reflection around the } x\text{-axis'}) \label{bridgereflection}.
\end{align}
\end{remark}
A continuous version of a Brownian bridge $(B_t^{x,\TT,y})_{t \in [0,\TT]}$ can be thought as a random
function on the canonical space $(C[0,\TT], \B(C[0,\TT]), \p_{x,\TT,y})$, where $\p_{x,\TT,y}$ denotes the associated probability measure. In the following proposition we give different characterizations for the function (\ref{condprobfunction}) in terms of hitting times.
For all $c \in \R$, $a < b$, and $\omega \in C[0,\TT]$, we let
\begin{align*}
H_c(\omega) & := \inf \set{t \in [0,\TT]: \omega_t = c}, & H_{(a,b)}(\omega) & := \inf \set{t \in [0,\TT]: \omega_t \notin (a,b)},\\
\Hl_c(\omega) & := \sup \set{t \in [0,\TT]: \omega_t = c}, & \Hl_{(a,b)}(\omega) & := \sup \set{t \in [0,\TT]: \omega_t \notin (a,b)}.
\end{align*}
\begin{proclaim}\label{Propq} Let $(h,\theta) \in (0,\infty) {\times} (0,T]$. Suppose that $(B^{y/\sigma, \TT,0}_t)_{t \in [0,\TT]}$ is a Brownian bridge on a probability space $(\tilde{\Omega}, \tilde{\F}, \tilde{\p})$, and define
\begin{align}\label{defq}
q(y) = q(y,h,\theta) := \tilde{\p}( (B_t^{y/\sigma,\theta,0})_{t \in [0,\TT]} \text{ hits } \Z^{h/\sigma}_e \text{ before hitting } \Z^{h/\sigma}_o), \quad y \in \R.
\end{align}
Then, for all $k \in \Z,$
\begin{align}
(i) \quad & q(y) = \p(L \text{ even}|X_{\TT} = y), \quad y \notin \Z^h, \label{PropqI}\\
(ii) \quad & q(y) = \left\{ \begin{array}{ll}
\p_{y/\sigma,\TT,0}(H_{2kh/\sigma} < H_{(2k+1)h/\sigma}), \quad y \in (2kh,(2k{+}1)h),\\
\p_{y/\sigma,\TT,0}(H_{2kh/\sigma} < H_{(2k-1)h/\sigma}), \quad y \in ((2k{-}1)h,2kh),\\
\end{array}\right.\label{PropqII}\\
(iii) \quad & q(y) = \left\{ \begin{array}{ll}
\dfrac{d_o(y)}{h} + \dfrac{\sigma}{h} \E_{\tilde{\p}}\Big[B^{0,\theta, y/\sigma}_{\tilde{H}_{(-((2k+1)h-y)/\sigma, (y-2kh)/\sigma)}}\Big], & y \in (2kh,(2k{+}1)h),\\
\dfrac{d_o(y)}{h} - \dfrac{\sigma}{h} \E_{\tilde{\p}}\Big[B^{0,\theta, y/\sigma}_{\tilde{H}_{(-(2kh-y)/\sigma, (y-(2k-1)h)/\sigma)}}\Big], & y \in ((2k{-}1)h,2kh).
\end{array}
\right.\label{PropqIII}
\end{align}
Here $\tilde{H}_{(a,b)} = \inf \{t \in [0,\theta]: B^{0,\theta,y/\sigma}_t \notin (a,b)\}$, and $\p$ refers to the probability measure on the space $(\Omega, \F, \p)$ considered in Section \ref{settingsection}.
\end{proclaim}
\begin{remark}
It is clear by (\ref{bridgereflection}) that the function $q$ is symmetric.
\end{remark}
\begin{proof}[Proof of Proposition \ref{Propq}]
Item $(ii)$ is clear. To show $(i)$, observe that if $X_{\TT}(\omega) \in (2kh, (2k{+}1)h)$ and $L(\omega)$ is even, the path $t \mapsto X_t(\omega)$ does hit $2kh$ at $\tau_{L}(\omega)$
and afterwards, i.e.~on $[\tau_{L}(\omega), \TT)$, it does not hit any other $mh$ ($m \neq 2k$) and hence stays inside $((2k{-}1)h, (2k{+}1)h)$.
Therefore, the last entry of this path into $(2kh, (2k{+}1)h)$ occurs via $2kh$, and thus
\begin{align*}
\p({L \text{ even}, X_{\TT} \in (2kh, (2k{+}1)h)})
& = \p_{0}(\sigma \omega_{\hat{H}_{(2kh,(2k+1)h)}(\omega)} = 2kh,\, \sigma \omega_{\TT} \in(2kh, (2k{+}1)h))\\
& = \p_{0}\big(\omega_{\hat{H}_{(2kh/\sigma,(2k+1)h/\sigma)}(\omega)} = \tfrac{2kh}{\sigma},\, \omega_{\TT}\in (\tfrac{2kh}{\sigma}, \tfrac{(2k{+}1)h}{\sigma})\big)\\
& = \p_{0}(\hat{H}_{2kh/\sigma} > \hat{H}_{(2k+1)h/\sigma}),
\end{align*}
where $\p_{0}$ denotes the Wiener measure on $(C[0,\TT], \B(C[0,\TT]))$. Thus, for $y \in (2kh, (2k{+}1)h)$,
\begin{align*}
\p(L \text{ even}|X_{\TT} = y) & = \p_{0,\TT,y/\sigma}(\hat{H}_{2kh/\sigma} > \hat{H}_{(2k+1)h/\sigma}) = \p_{y/\sigma, \TT,0}(H_{2kh/\sigma} < H_{(2k+1)h/\sigma}) = q(y),
\end{align*}
where we used relations (\ref{bridgetimereversal}), (\ref{PropqII}), and the fact that $\p( \, \cdot \, | X_{\TT} = y) = \p_{0,\TT,y/\sigma}$ on $(C[0,\theta],\B(C[0,\theta]))$ (see e.g. \cite[Chapter 1]{RevuzYor}). The case $y \in ((2k{-}1)h, 2kh)$ is similar.

For $(iii)$, assume $y \in ((2k{-}1)h, 2kh)$; the case $y \in (2kh, (2k{+}1)h)$ is similar. 
It is clear that whenever $z \notin (a,b)$, $a < 0 < b$, and $\tilde{H}_{(a,b)} = \inf \{t \in [0,\TT]: B_t^{0,\TT,z} \notin (a,b)\}$,
\begin{align}\label{niceqrep}
\p_{0,\TT,z}(H_a < H_b) = \frac{b}{b-a} - \frac{1}{b-a}\E_{\tilde{\p}}\Big[B^{0,\TT,z}_{\tilde{H}_{(a,b)}}\Big].
\end{align}
In addition, from (\ref{PropqII}) we deduce that
\begin{align}\label{FFIVsymmetries}
q(y) & = \p_{y/\sigma,\TT,0}(H_{2kh/\sigma} < H_{(2k-1)h/\sigma}) \nonumber\\
& = \p_{0, \TT, -y/\sigma}(H_{(2kh-y)/\sigma} < H_{((2k-1)h-y)/\sigma}) \nonumber\\
& = \p_{0, \TT, y/\sigma}(H_{(y-2kh)/\sigma} < H_{(y-(2k-1)h)/\sigma})
\end{align}
by (\ref{bridgeshift}) and (\ref{bridgereflection}). Substitute $z = \tfrac{y}{\sigma}$, $a = \frac{y-2kh}{\sigma}$, and $b = \frac{y-(2k{-}1)h}{\sigma}$. Then $z \notin (a,b)$, $a < 0 < b$, $b-a = \frac{h}{\sigma}$, and hence by (\ref{niceqrep}), (\ref{FFIVsymmetries}), and $d_o(y) = y-(2k{-}1)h$,
\begin{align*}
q(y) & = \frac{d_o(h)}{h} - \frac{\sigma}{h} \E_{\tilde{\p}}\pl{B^{0, \TT, y/\sigma}_{\tilde{H}_{((y-2kh)/\sigma,(y-(2k{-}1)h)/\sigma)}}}.
\end{align*}
\end{proof}
The probability for the Brownian motion $(W_t + y/\sigma)_{t \geq 0}$ to hit the set $\Z_e^{h/\sigma}$ before hitting the set $\Z_o^{h/\sigma}$ is equal to $d_o(y)/h$ (cf.~(\ref{defq})). As pointed out in \cite[Section 9]{Walsh}, the piecewise linear function $y \mapsto d_o(y)/h$ can be used to approximate the function $y \mapsto q(y)$ for small $h > 0$. Estimates related to this approximation, which are also applied in the proof of Lemma \ref{PropJ}, are presented in the proposition below. We denote by $p = p(\, \cdot \,, \theta)$ the density of the random variable $X_{\TT}$.
\begin{proclaim}\label{PROPqprobestimate}
Suppose that $(h,\theta) \in (0,\infty) {\times} (0,T]$ and define
\begin{align}\label{varrhon}
\varrho : \R \to \R, \quad \varrho(y) = \varrho(y, h, \theta) := q(y) - d_o(y)/h,
\end{align}
where $q = q(\, \cdot \,, h, \theta)$ was introduced in (\ref{defq}). Then $\varrho$ is symmetric, and it holds that
\begin{align*}
(i) \quad \int_{0}^h \abs{\varrho(y)} p(y) dy \leq \frac{26}{10}\frac{h}{\sigma \sqrt{\theta}} + \frac{ h^2}{\sigma^2 \theta}, \quad \quad (ii) \quad  \int_h^{\infty} \abs{\varrho(y)} p(y) dy \leq \frac{29}{10}\frac{h}{\sigma \sqrt{\theta}} + \frac{h^2}{\sigma^2 \theta}.
\end{align*}
\end{proclaim}
Proposition \ref{PROPqprobestimate} can be seen as a generalization of \cite[Corollary 3.3]{GLS} to the time-dependent setting. The proof uses certain estimates of \cite{GLS} related to the expected first hitting times of Brownian bridges \cite[Lemma 3.1 and Lemma 3.2 $(i)$]{GLS}. For the convenience of the reader, we collect those results in the lemma below using the notation of this subsection.
\begin{lemma} Let $(h,\theta) \in (0,\infty) {\times} (0,\infty)$ and suppose that $a < 0 < b$ and $y \notin (a,b)$. Then
\begin{align}
\E_{\tilde{\p}}\big[B^{0,\theta,y}_{\tilde{H}_{(a,b)}}\big] & \leq \frac{\E_{0,\theta,y}[H_{(a,b)}]}{\theta}\pr{|y| + 2(|a|\vee b)+3\sqrt{2\theta}}, \label{bridge_sampled_at_hitting_time_estimate}\\
\E_{0,\theta,y}[H_{(a, b)}] & \leq \left\{ \begin{array}{ll}
b(2|a|+y) \wedge \theta, & y \geq b,\\
|a|(2b+|y|) \wedge \theta, & y \leq a,
\end{array}
\right. \label{bridge_hitting_upper_estimate} \\
\E_{0,\theta,y}[H_{(-h,h)}] & = h \int_{0}^{\theta} \frac{\gamma_{\theta - t}(0,y)}{\gamma_{\theta}(0,y)}\frac{F(h/\sqrt{t})}{\sqrt{2\pi t}} dt, \quad y \notin (-h,h), \label{symmetric_bridge_hitting_eq}
\end{align}
where
\begin{align}\label{gamma_ja_iso_F}
\gamma_t(x,y) := \frac{1}{\sqrt{2\pi t}}e^{-(x-y)^2/2t} \quad \text{ and } \quad F(x) := \sum_{m = -\infty}^{\infty} (-1)^m e^{-2m^2 x^2}.
\end{align}
\end{lemma}
\begin{proof}[Proof of Proposition \ref{PROPqprobestimate}]
Recall the functions $q$, $d_o, d_e$, and $\varrho$, given by (\ref{defq}), (\ref{defofdode}), and (\ref{varrhon}), respectively.
The function $\varrho$ is symmetric as a linear combination of the symmetric functions  $q$ and $d_o$.

To show $(i)$, we use (\ref{PropqIII}) and (\ref{bridge_sampled_at_hitting_time_estimate}) to obtain
\begin{align}\label{absrhooupperbdd}
\abs{\varrho(y)} \leq \left\{ \begin{array}{ll}
\dfrac{|y| + 2h + 3 \sigma\sqrt{2\TT}}{\TT h} \E_{0,\theta,y/\sigma}[H_{(-((2k+1)h-y)/\sigma, (y-2kh)/\sigma)}], & y \in (2kh,(2k{+}1)h),\\
\dfrac{|y| + 2h + 3 \sigma \sqrt{2\TT}}{\TT h} \E_{0,\theta,y/\sigma}[H_{(-(2kh-y)/\sigma, (y-(2k-1)h)/\sigma)}], & y \in ((2k{-}1)h,2kh).
\end{array}
\right.
\end{align}
In addition, (\ref{bridge_hitting_upper_estimate}) implies that for $y \in (0,h)$,
\begin{align}\label{nollahooestimate}
\E_{0,\theta,y/\sigma}[H_{(-(h{-}y)/\sigma, y/\sigma)}] & \leq \frac{y}{\sigma}\Big(\frac{2|h-y|}{\sigma} + \frac{y}{\sigma}\Big) = \frac{y}{\sigma^2}\big(2(h-y)+y) \leq \frac{h^2}{\sigma^2}.
\end{align}
Consequently, by (\ref{absrhooupperbdd}) (with $k = 0$) and (\ref{nollahooestimate}),
\begin{align*}
\int_0^h \abs{\varrho(y)} p(y) dy & \leq \int_{0}^h \frac{y + 2h + 3 \sigma \sqrt{2 \TT}}{\theta h} \E_{0,\theta,y/\sigma}[H_{(-(h{-}y)/\sigma, y/\sigma)}] p(y) dy\\
& \leq \frac{h^2}{\sigma^2}\pl{\int_{0}^h \frac{y}{\TT h} p(y) dy + \frac{2h + 3 \sigma \sqrt{2\TT}}{\TT h} \int_{0}^h p(y) dy}\\
& \leq \frac{h}{\sigma \sqrt{\theta}}\int_{0}^\infty \frac{y}{\sigma \sqrt{\theta}} p(y) dy +  \frac{2h^2 + 3 h \sigma \sqrt{2 \TT}}{\sigma^2 \TT}\int_{0}^\infty p(y) dy\\
& \leq \frac{1}{\sqrt{2\pi}}\frac{h}{\sigma \sqrt{\theta}} + \frac{2h^2 + 3h\sigma \sqrt{2 \theta}}{2\sigma^2 \theta}\\
& \leq \frac{26}{10}\frac{h}{\sigma \sqrt{\theta}} + \frac{h^2}{\sigma^2 \theta}
\end{align*}
since $\frac{1}{\sqrt{2\pi}} + \frac{3}{\sqrt{2}} \leq \frac{26}{10}$. This proves item $(i)$.

$(ii)$: By extending the exit intervals, we get for all $k \in \Z$ that
\begin{align}\label{all-the-expectations}
\E_{0,\theta,y/\sigma}[H_{(-h/\sigma, h/\sigma)}] \geq \left\{ \begin{array}{ll}
\E_{0,\theta,y/\sigma}[H_{(-((2k+1)h-y)/\sigma, (y-2kh)/\sigma)}], & \quad y \in (2kh, (2k{+}1)h),\\
\E_{0,\theta,y/\sigma}[H_{(-(2kh-y)/\sigma, (y-(2k-1)h)/\sigma)}], & \quad y \in ((2k{-}1)h,2kh)).
\end{array}
\right.
\end{align}
In terms of the functions $\gamma$ and $F$ defined in (\ref{gamma_ja_iso_F}), we let
\begin{align}\label{def_of_C_the_integral_monster}
C(\theta, \tfrac{h}{\sigma}, \tfrac{y}{\sigma}) := \int_{0}^{\sigma^2 \theta/h^2} \!\! \frac{\gamma_{\sigma^2 \theta - h^2 u}(0,y)}{\gamma_{\sigma^2 \theta}(0,y)} \frac{F(1/\sqrt{u})}{\sqrt{2\pi u}} = \frac{\sigma^2}{h^2} \E_{0,\theta,y/\sigma}[H_{(-h/\sigma, h/\sigma)}],
\end{align}
recall (\ref{symmetric_bridge_hitting_eq}). Hence, by (\ref{absrhooupperbdd}) and (\ref{all-the-expectations}),
\begin{align}\label{exitJ0}
\int_{h}^{\infty} \abs{\varrho(y)} p(y) dy & \leq \frac{h^2}{\sigma^2}\int_{h}^{\infty} \frac{y + 2h + 3\sigma \sqrt{2\theta}}{\theta h} C(\theta, \tfrac{h}{\sigma}, \tfrac{y}{\sigma}) p(y)dy.
\end{align}
Using the knowledge that $u \mapsto \frac{F(1/\sqrt{u})}{\sqrt{2\pi u}}$ is a p.d.f. on $(0,\infty)$ (see \cite[Proposition 2.8]{GLS}), a standard computation yields
$C(\theta, \tfrac{h}{\sigma}, \tfrac{y}{\sigma}) \leq 1 \vee \frac{\sigma \sqrt{\theta}}{y}$, and thus
\begin{align}\label{exitJ1}
\frac{h^2}{\sigma^2} \int_{h}^{\infty} \!\!\frac{y}{\theta h} \Big(1 \vee \frac{\sigma \sqrt{\theta}}{y}\Big)p(y) dy & \leq \frac{h}{\sigma \sqrt{\theta}} \int_{0}^{\infty} \!\! \pr{1 \vee u} \gamma_1(0,u) du \leq \frac{h}{\sigma \sqrt{\theta}} \Big(\frac{1}{2} + \frac{e^{-1/2}}{\sqrt{2\pi}}\Big).
\end{align}
In addition, by (\ref{def_of_C_the_integral_monster}) and by the fact that $p = \gamma_{\sigma^2 \theta}(0,\, \cdot \,)$,
\begin{align*}
\int_{h}^{\infty} C(\theta, \tfrac{h}{\sigma}, \tfrac{y}{\sigma}) p(y) dy = \int_h^{\infty} \int_{0}^{\sigma^2 \theta/h^2} \!\! \!\!\gamma_{\sigma^2 \theta - h^2 u}(0,y) \frac{F(1/\sqrt{u})}{\sqrt{2\pi u}} du dy & \leq \frac{1}{2} \int_{0}^{\sigma^2 \theta/h^2} \!\! \frac{F(1/\sqrt{u})}{\sqrt{2\pi u}} du \leq \frac{1}{2}
\end{align*}
since the function $u \mapsto \frac{F(1/\sqrt{u})}{\sqrt{2\pi u}}$ integrates to one over the interval $(0,\infty)$. Consequently,
\begin{align}\label{exitJ2}
\frac{h^2}{\sigma^2} \cdot \frac{2h + 3 \sigma \sqrt{2\TT}}{\TT h} \int_{h}^{\infty} C(\theta, \tfrac{h}{\sigma}, \tfrac{y}{\sigma}) p(y) dy \leq \frac{h^2}{\sigma^2 \theta} + \frac{3}{\sqrt{2}}\frac{h}{\sigma \sqrt{\theta}}.
\end{align}
The claim then follows by applying (\ref{exitJ1}) and (\ref{exitJ2}) to the right-hand side of (\ref{exitJ0}) and by observing that $\frac{1}{2} + \frac{e^{-1/2}}{\sqrt{2\pi}}+\frac{3}{\sqrt{2}} \leq \frac{29}{10}$. 
\end{proof}
\subsection{The local error for $g \in \gbv$}
The estimation of the local error for the class $\gbv$ relies on the following observation: If $g \in \gbv$ is given by (\ref{GBVdefrep}) and if $g^{x_0} := g(x_0{+}\, \cdot \,)$ for some $x_0 \in \R$, then
\begin{align}\label{newfRep}
g^{x_0}(x) = c+\!\!\!\int\displaylimits_{[0,\infty)} \!\!\!\I_{(\yx, \infty)}(x) d\mu(y)-\!\!\!\int\displaylimits_{(-\infty,0)} \!\!\! \I_{(-\infty, \yx]}(x) d\mu(y)+\sum_{i=1}^{\infty}\alpha_i \IN{x_i{-}x_0}(x).
\end{align}
Using the representation (\ref{newfRep}) and linearity, the estimation of the error $\varepsilon^{\text{loc}}_{h,\theta}(g^{x_0})$ essentially reduces to the estimation of integrals, where the integrands consist of indicator functions or their linear approximations given by the operators $\varPi_e$ and $\varPi_o$ (introduced in Definition \ref{Deflinearizationop}). The following proposition enables us to interchange the order of integration or summation with the application of these operators.

{Recall that $p = p(\, \cdot \,, \theta)$ is the density of $X_{\TT}$ and that $q = q(\, \cdot \,, h,\TT)$ is the function defined in (\ref{defq})}.

\begin{proclaim}\label{Propgbvlinearized}
Suppose that $(h,\theta) \in (0,\infty) {\times} (0,T]$ and that $g \in \gbv$ admits the representation (\ref{GBVdefrep}). Then, for all $x_0 \in \R$,
\begin{align*}
(i) \quad \varPi_e g^{x_0} (x) = c & + \int\displaylimits_{[0,\infty)} \varPi_e \I_{(y{-}x_0,\infty)}(x) d\mu(y) - \int\displaylimits_{(-\infty,0)} \varPi_e \I_{(-\infty,y{-}x_0]}(x) d\mu(y) \nonumber\\
& + \sum_{i \in \N: x_i{-}x_0 \in \Z^h_e} \alpha_i \varPi_e \IN{x_i - x_0}(x), \quad x \in \R,\\ 
(ii) \quad\varPi_o \varPi_e g^{x_0}(x) = c & + \int\displaylimits_{[0,\infty)} \varPi_o \varPi_e \I_{(y{-}x_0,\infty)}(x) d\mu(y) - \int\displaylimits_{(-\infty,0)} \varPi_o \varPi_e \I_{(-\infty,y{-}x_0]}(x) d\mu(y) \nonumber\\
& + \sum_{i \in \N: x_i{-}x_0 \in \Z^h_e} \alpha_i \varPi_o \varPi_e \IN{x_i - x_0}(x), \quad x \in \R.
\end{align*}
\end{proclaim}
\begin{proof}[Idea of the proof]
Items $(i)$--$(ii)$ follow by using the representation (\ref{newfRep}), linearity of the operations\\$f \mapsto \varPi_e f$, $f \mapsto \varPi_o f$, and $f \mapsto \int f d|\mu|$, and relation (\ref{DiracRep}).
\end{proof}
\begin{proclaim}\label{localerrorrepresentationPROP}
Let $(h,\theta) \in (0,\infty) {\times} (0,T]$. Suppose that $g \in \gbv$ admits the representation (\ref{GBVdefrep}) and that $\beta \geq 0$ is as in (\ref{bump-function-condition}). Then, for all $x_0 \in \R$, 
\begin{align}\label{localerrorboundforgbv}
\big|\E[g(x_0{+}X_{\tau_{J}})-g(x_0{+}X_{\theta})]\big| & \leq \frac{7}{\sqrt{2\pi}} \frac{h}{\sigma \sqrt{\TT}} e^{3\beta h + \beta\abs{x_0} + \beta^2 \sigma^2 T/2} \nonumber \\
& \quad \quad \times \bigg(\int_{\R} e^{-\beta\abs{y}} d|\mu|(y) + \!\!\!\sum_{i \in \N: x_i{-}x_0 \in \Z^h_e} \abs{\alpha_i} e^{-\beta\abs{x_i}}\bigg).
\end{align}
\end{proclaim}
\begin{proof}
For given $x_0 \in \R$, we apply (\ref{sumofexpectationslocalerror}) for the function $g(x_0{+}\, \cdot \,)$. By Proposition \ref{Propgbvlinearized} and by the relation $\p(L \text{ even}|X_{\TT} = x) = q(x)$ (Leb-a.e.), we may decompose the expectation on the left-hand side of (\ref{localerrorboundforgbv}) in the following way:
\begin{align*}
& \E[g(x_0{+}X_{\tau_{J}})-g(x_0{+}X_{\theta})]\\
& = \int\displaylimits_{\R} \int\displaylimits_{[0,\infty)} \pl{\varPi_e \I_{(\yx, \infty)}(x) - \I_{(\yx,\infty)}(x)} d\mu(y) p(x) dx\\
& \quad \quad + \int\displaylimits_{\R} \int\displaylimits_{(-\infty, 0)} \pl{\varPi_e \I_{(-\infty, \yx]}(x) - \I_{(-\infty,\yx]}(x)}d\mu(y) p(x) dx\\
& \quad \quad + \int\displaylimits_{\R} \int\displaylimits_{[0,\infty)} \pl{\varPi_o \varPi_e \I_{(y-x_0, \infty)}(x) - \varPi_e \I_{(y-x_0,\infty)}(x)} d\mu(y) q(x) p(x) dx\\
& \quad \quad + \int\displaylimits_{\R} \int\displaylimits_{(-\infty, 0)} \pl{\varPi_o \varPi_e \I_{(-\infty, \yx]}(x) - \varPi_e \I_{(-\infty,\yx]}(x)} d\mu(y) q(x) p(x) dx\\
& \quad \quad + \int\displaylimits_{\R} \sum_{i \in \N: x_i{-}x_0 \in \Z^h_e} \alpha_i \varPi_e \IN{x_i {-} x_0}(x)p(x)dx\\
& \quad \quad + \int\displaylimits_{\R} \sum_{i \in \N: x_i{-}x_0 \in \Z^h_e} \alpha_i  \pl{\varPi_o \varPi_e \IN{x_i{-}x_0}(x) - \varPi_e \IN{x_i {-} x_0}(x)} q(x)p(x) dx\\
& =: E^{(1)} + E^{(2)} + E^{(3)} +E^{(4)}+E^{(5)}+E^{(6)}.
\end{align*}
We will derive upper estimates for the quantities $|E^{(i)}|, 1 \leq i \leq 6,$ in the following steps.\\
\textbf{Step 1: $E^{(1)}$ and $E^{(2)}$.}
Suppose that $\yx \in [2kh, (2k{+}2)h)$ for some $k \in \Z$. Then
$$\abs{\varPi_e \I_{(\yx,\infty)}(x) - \I_{(\yx,\infty)}(x)} \leq \I_{[2kh,(2k{+}2)h)}(x),$$ 
and since for each $x \in [2kh, (2k{+}2)h)$ it holds that $\abs{y} \leq 2h + \abs{x_0} + \abs{x}$, we have
\begin{align*}
& e^{\beta \abs{y}} \int_{\R} \abs{\varPi_e \I_{(\yx,\infty)}(x) - \I_{(\yx,\infty)}(x)}p(x)dx\\
& \leq e^{2\beta h + \beta \abs{x_0}} \int_{\R} e^{\beta \abs{x}} \abs{\varPi_e \I_{(\yx,\infty)}(x) - \I_{(\yx,\infty)}(x)} p(x)dx\\
& \leq e^{2\beta h + \beta \abs{x_0}} \int_{2kh}^{(2k{+}2)h} e^{\beta\abs{x}} p(x) dx\\
& \leq \frac{2}{\sqrt{2\pi}}e^{2\beta h + \beta \abs{x_0} + \beta^2 \sigma^2 T/2} \frac{h}{\sigma \sqrt{\TT}}.
\end{align*}
Consequently, by Fubini's theorem,
\begin{align}\label{EEE1bound}
\big| E^{(1)} \big| & \leq \int\displaylimits_{[0,\infty)} e^{-\beta\abs{y}} \Big(e^{\beta\abs{y}} \int\displaylimits_{\R} \abs{\varPi_e \I_{(\yx,\infty)}(x) - \I_{(\yx,\infty)}(x)}p(x)dx\Big) d|\mu|(y) \nonumber\\
& \leq \frac{h}{\sigma \sqrt{\TT}}  \frac{2}{\sqrt{2\pi}}e^{2\beta h + \beta \abs{x_0} + \beta^2 \sigma^2 T/2} \int_{[0,\infty)} e^{-\beta\abs{y}} d|\mu|(y).
\end{align}
In fact, it also holds that
\begin{align}\label{EEE2bound}
\big|E^{(2)}\big| \leq \frac{h}{\sigma \sqrt{\TT}}  \frac{2}{\sqrt{2\pi}}e^{2\beta h + \beta \abs{x_0} + \beta^2 \sigma^2 T/2} \int_{(-\infty, 0)} e^{-\beta\abs{y}} d|\mu|(y)
\end{align}
since $\abs{\varPi_e \I_{(-\infty,\yx]}(x) - \I_{(-\infty,\yx]}(x)} = \abs{ \varPi_e \I_{(\yx,\infty)}(x) - \I_{(\yx,\infty)}(x)}$ for all $x \in \R$, which is a direct consequence of the relation
\begin{align}\label{reversionLin}
\varPi_e \I_{(-\infty, r]} =  1 - \varPi_e \I_{(r, \infty)}, \quad r \in \R.
\end{align}
\textbf{Step 2: $E^{(3)}$ and $E^{(4)}$.} Suppose $y{-}x_0 \in [2kh, (2k{+}2)h)$ for some $k \in \Z$. Then
$\abs{y} \leq 3h + \abs{x_0} + \abs{x}$ holds for all $x \in [(2k{-1})h, (2k{+}3)h)$, and by (\ref{form}) we may estimate
\begin{align*}
& e^{\beta\abs{y}}\int_{\R} \abs{\varPi_o \varPi_e \I_{(y-x_0, \infty)}(x) - \varPi_e \I_{(y-x_0,\infty)}(x)}q(x) p(x) dx\\
& \leq e^{3\beta h + \beta\abs{x_0}}\int_{\R} e^{\beta\abs{x}} \abs{\varPi_o \varPi_e \I_{(y-x_0, \infty)}(x) - \varPi_e \I_{(y-x_0,\infty)}(x)}q(x) p(x) dx\\
& \leq e^{3\beta h + \beta\abs{x_0}}\int_{(2k{-}1)h}^{(2k{+}3)h} e^{\beta\abs{x}}\frac{d_o(x)}{4h} q(x) p(x) dx\\
& \leq \frac{1}{4}e^{3\beta h + \beta\abs{x_0}}\int_{(2k{-}1)h}^{(2k{+}3)h} e^{\beta \abs{x}}p(x) dx\\
& \leq \frac{1}{\sqrt{2\pi}} e^{3\beta h + \beta\abs{x_0}+ \beta^2\sigma^2 T/2} \frac{h}{\sigma \sqrt{\theta}}.
\end{align*} 
Hence, by Fubini's theorem,
\begin{align}\label{EEE3bound}
\big|E^{(3)}\big| & \leq \int\displaylimits_{[0,\infty)} e^{-\beta\abs{y}} \Big(e^{\beta\abs{y}} \int\displaylimits_{\R} \abs{\varPi_o \varPi_e \I_{(y-x_0, \infty)}(x) {-} \varPi_e \I_{(y-x_0,\infty)}(x)}q(x) p(x) dx\Big) d|\mu|(y) \nonumber\\
& \leq \frac{h}{\sigma \sqrt{\TT}} \frac{1}{\sqrt{2\pi}} e^{3\beta h + \beta\abs{x_0} + \beta^2 \sigma^2 T/2} \int\displaylimits_{[0,\infty)} e^{-\beta\abs{y}}d|\mu|(y).
\end{align}
Moreover, by (\ref{reversionLin}) and by the linearity of $\varPi_o$, we obtain
\begin{align}\label{EEE4bound}
\big|E^{(4)}\big| & \leq \frac{h}{\sigma \sqrt{\TT}} \frac{1}{\sqrt{2\pi}} e^{3\beta h + \beta\abs{x_0} + \beta^2 \sigma^2 T/2}\int\displaylimits_{(-\infty, 0)} e^{-\beta\abs{y}}d|\mu|(y),
\end{align}
since $\abs{\varPi_o \varPi_e \I_{(-\infty, y{-}x_0]}(x) - \varPi_e \I_{(-\infty, y{-}x_0]}(x)} = \abs{\varPi_o \varPi_e \I_{(y{-}x_0, \infty)}(x) - \varPi_e \I_{(y{-}x_0, \infty)}(x)}$, $x \in \R$.
\\
\textbf{Step 3: $E^{(5)}$.} By (\ref{DiracRep}), $\varPi_e \IN{\xi} \equiv 0$ if $\xi \notin \Z^h_e$, and by (\ref{piecewiseVarPie}), $\varPi_e \IN{\xi} \leq \I_{[\xi-2h, \xi+2h]}$ if $\xi \in \Z^h_e$. In addition, since
$\abs{x_i} \leq 2h + \abs{x_0} + \abs{x}$ whenever $\abs{x - (x_i{-}x_0)} \leq 2h$,
\begin{align}\label{EEE5bound}
\big|E^{(5)}\big| & \leq \sum_{i \in \N: x_i{-}x_0 \in \Z^h_e} \abs{\alpha_i \int_{\R} \varPi_e \IN{x_i{-}x_0}(x)p(x) dx} \nonumber\\
& \leq \sum_{i \in \N: x_i{-}x_0 \in \Z^h_e} \abs{\alpha_i} e^{-\beta\abs{x_i}} \int_{\R} e^{\beta\abs{x_i}}p(x) \I_{[(x_i{-}x_0)-2h, (x_i{-}x_0)+2h]}(x) dx\nonumber\\
& \leq \sum_{i \in \N: x_i{-}x_0 \in \Z^h_e} \abs{\alpha_i} e^{-\beta\abs{x_i}} \int_{(x_i{-}x_0)-2h}^{(x_i{-}x_0)+2h} e^{\beta\abs{x_i}}p(x) dx\nonumber\\
& \leq \frac{h}{\sigma \sqrt{\TT}} \frac{4}{\sqrt{2\pi}}e^{2\beta h + \beta\abs{x_0} + \beta^2\sigma^2 T/2} \sum_{i \in \N: x_i{-}x_0 \in \Z^h_e} \abs{\alpha_i} e^{-\beta\abs{x_i}}.
\end{align}
\\
\textbf{Step 4: $E^{(6)}$.} If $\xi \in \Z^h_e$, relations (\ref{DiracRep}), (\ref{varpiocoincides}), and the linearity of $\varPi_o$ imply that
\begin{align*}
& \varPi_o \varPi_e \IN{\xi}(x) - \varPi_e \IN{\xi}(x)\\
& = \frac{1}{4h} \Big(\varPi_o \abs{\, \cdot \,-(\xi{-}2h)}(x) - \abs{x-(\xi{-}2h)}\Big) - \frac{1}{2h} \Big(\varPi_o \abs{\, \cdot \,-\xi}(x) - \abs{x-\xi}\Big)\\
& \quad \quad +\frac{1}{4h}\Big(\varPi_o\abs{\, \cdot \,-(\xi{+} 2h)}(x) - \abs{x-(\xi{+} 2h)}\Big)\\
& = \frac{d_o(x)}{4h} \Big( \I_{[\xi - 3h, \xi - h)}(x) - 2\I_{[\xi-h, \xi+h)}(x) + \I_{[\xi + h, \xi + 3h)}(x)\Big), \quad x \in \R.
\end{align*}
In addition, we have $\varPi_o \varPi_e \IN{\xi} - \varPi_e \IN{\xi} \equiv 0$ for $\xi \notin \Z^h_e$ by (\ref{DiracRep}).
Therefore, since \\
$\abs{x_i} \leq 3h + \abs{x} + \abs{x_0}$ whenever $\abs{x-(x_i{-}x_0)} \leq 3h$, we get
\begin{align}\label{EEE6bound}
\big|E^{(6)}\big| & \leq\sum_{i \in \N: x_i{-}x_0 \in \Z^h_e} \abs{\alpha_i}  \int_{\R} \abs{\varPi_o \varPi_e \IN{x_i{-}x_0}(x) - \varPi_e \IN{x_i {-} x_0}(x)} q(x)p(x) dx \nonumber\\
& \leq \sum_{i \in \N: x_i{-}x_0 \in \Z^h_e} \abs{\alpha_i} e^{-\beta\abs{x_i}}  \int_{(x_i{-}x_0)-3h}^{(x_i{-}x_0)+3h} e^{\beta\abs{x_i}} \frac{d_o(x)}{2h}  q(x)p(x)dx \nonumber\\
& \leq \sum_{i \in \N: x_i{-}x_0 \in \Z^h_e} \frac{\abs{\alpha_i}}{2} e^{-\beta\abs{x_i} + 3\beta h + \beta\abs{x_0}} \int_{(x_i{-}x_0)-3h}^{(x_i{-}x_0)+3h} e^{\beta\abs{x}} p(x)dx \nonumber\\
& \leq \frac{h}{\sigma \sqrt{\TT}} \frac{3}{\sqrt{2\pi}} e^{3 \beta h + \beta|x_0| + \beta^2\sigma^2T/2} \sum_{i \in \N: x_i{-}x_0 \in \Z^h_e} \abs{\alpha_i} e^{-\beta\abs{x_i}}.
\end{align}
It remains to observe that the sum the right-hand sides of (\ref{EEE1bound}), (\ref{EEE2bound}),
(\ref{EEE3bound}), (\ref{EEE4bound}), (\ref{EEE5bound}), and (\ref{EEE6bound}) are bounded from above by the right-hand side of (\ref{localerrorboundforgbv}).
\end{proof}
In order to distinguish between the general setting $(h,\theta)$ and the specific $n$-dependent setting $(h_n, \theta_n)$, we will refer to the assumption below.
\begin{assumption}\label{assumption}
For given $t_0 \in [0,T)$ and $n \! \in \! 2\N$, we substitute $(h,\theta) = (h_n, \tn)$, where
\begin{align*}
h_n = \sigma \sqrt{\frac{T}{n}}, \quad \theta_n = \frac{\nt T}{n} \quad \text{ and } \quad \nt = 2 \left \lceil \frac{T{-}t_0}{2T/n} \right \rceil
\end{align*}
as in (\ref{deftnnt}).
For notational convenience, we will drop the subscript $n$ from $h_n$.
\end{assumption}
\begin{remark}
The special choice $(h,\theta) = (h_n, \tn)$ in Assumption \ref{assumption} affects the objects below used throughout this text:
\begin{align*}
\tau_k & = \inf \big\{t > \tau_{k-1}: |X_t - X_{\tau_{k-1}}| = h\big\}, \quad (X_{\tau_{k}})_{k=0,1, \dots}, \quad (\F_{\tau_{k}})_{k=0,1, \dots},\\
\quad J_n & = J = \inf \{2m \in 2\N: \tau_{2m} > \tn\}, \quad L_n = L = \sup \{m \in \N_{0}: \tau_m < \tn\},\\
\Z^{h}_e & = \set{2kh: k \in \Z}, \quad \Z^{h}_o = \set{(2k{+}1)h: k \in \Z}, \quad \Z^{h} = \Z^{h}_o \cup \Z^{h}_e,\\
d_o(x) & = \dist (x, \Z^{h}_o), \quad d_e(x) = \dist (x, \Z^{h}_e), \quad p(x) = \p(X_{\tn} \in dx)/dx.
\end{align*}
This choice also affects the functions $q = q(\, \cdot\,, h, \theta)$ and $\varrho = \varrho(\, \cdot \,,h, \theta)$ defined in
(\ref{Propq}) and (\ref{varrhon}), respectively. In particular, Proposition \ref{Propq} implies that
$$q(x) = \p(L_n \text{ even}|X_{\tn} = x), \quad x \notin \Z^h.$$
\end{remark}
For the main result of this subsection, recall that $\varepsilon^{\text{loc}}_{n}(t_0,x_0) = \E[g(x_0{+}X_{\tau_{J_n}})-g(x_0{+}X_{\tn})]$.
\begin{corollary}\label{CorollaryGBVlocal}
Let $n \in 2\N$. Suppose that the function $g \in \gbv$ admits the representation (\ref{GBVdefrep}) and that $\beta \geq 0$ is as in (\ref{bump-function-condition}). Then, under Assumption \ref{assumption}, there exists a constant $C > 0$ such that for all $(t_0,x_0) \in [0, T){\times} \R$,
\begin{align*}
& \abs{\varepsilon_n^{\text{loc}}(t_0,x_0)} \leq \frac{C \sqrt{T} }{\sqrt{n(T-t^n_k)}}e^{\beta \abs{x_0} + 3 \beta^2 \sigma^2 T}, \quad t_0 \in [t^n_k, t^n_{k{+}1}), \quad 0 \leq k < \tfrac{n}{2}.
\end{align*}
\end{corollary}
\begin{proof}
Proposition \ref{localerrorrepresentationPROP} and the relation $h (\sigma^2 \tn)^{-1/2} = \nt^{-1/2}$ imply that
\begin{align*}
\abs{\varepsilon^{\text{loc}}_{n}(t_0,x_0)} \leq \frac{C_{\beta, \sigma, T} e^{\beta\abs{x_0}}}{\sqrt{\nt}} \bigg(\int_{\R} e^{-\beta\abs{y}} d|\mu|(y) + \sum_{i \in \N: x_i{-}x_0 \in \Z^h_e} \abs{\alpha_i} e^{-\beta\abs{x_i}}\bigg),
\end{align*}
where the coefficient $C_{\beta, \sigma, T} > 0$ implied by (\ref{localerrorboundforgbv}) can be estimated as follows:
\begin{align*}
C_{\beta, \sigma, T} = \frac{7}{\sqrt{2\pi}}e^{3\beta h + \beta^2 \sigma^2 T/2} \leq \frac{7}{\sqrt{2\pi}}e^{\frac{5}{2}\beta \sigma \sqrt{T} + \beta^2 \sigma^2 T/2} \leq C e^{3\beta^2 \sigma^2 T}
\end{align*}
for a constant $C > 0$. Since $\nt T = n(T-t^n_k)$ for $t_0 \in [t^n_k, t^n_{k+1})$ by (\ref{theta_lattice_relation}), we obtain the desired result.
\end{proof}
\subsection{On the sharpness of the rate for the class $\gbv$}\label{subseq-sharpness}
The following lemma indicates that the rate $n^{-1/2}$ for the class $\gbv$ is sharp. 
\begin{proclaim}\label{prop_lower_bound}
Under Assumption \ref{assumption}, there exists a function $g \in \gbv$ such that
\begin{align}\label{liminf_limsup_for_loc}
0 < \liminf_{n \to \infty} n^{1/2}\varepsilon_{n}(0,0) \leq \limsup_{n \to \infty} n^{1/2} \varepsilon_n(0,0) < \infty.
\end{align}
\end{proclaim}
\begin{proof}
For simplicity, let $T = \sigma = 1$ and $g := \I_{[0,\infty)}$. Then $h = n^{-1/2}$, $g \in \gbv$, and the location of the jump of $g$ belongs to the set $\Z^h_e$ for all $n \in \N$. 
Observe that then $\varepsilon^{\text{adj}}_n(0,0) = 0$ by Proposition \ref{ADJPROP} and $|\varepsilon^{\text{glob}}_{n}(0,0)| \leq Cn^{-1}$ by Proposition \ref{globerrorfinally} below, where $C > 0$ is some constant. Consequently, 
it suffices to show that (\ref{liminf_limsup_for_loc}) is valid for the local error $\varepsilon_n^{\text{loc}}(0,0)$; recall (\ref{varepsilon_n_splitting}).

The expression $n^{1/2} \varepsilon^{\text{loc}}_n(0,0)$ is bounded from above by Corollary \ref{CorollaryGBVlocal}. For the lower bound, we note that by Definition \ref{Deflinearizationop},
$$\varPi_e \I_{[0,\infty)}(x) = \pr{1 \wedge \tfrac{x {+} 2h}{2h}}\I_{[-2h,\infty)}(x), \quad \varPi_o \varPi_e \I_{[0,\infty)}(x) = \pr{1 \wedge \tfrac{x{+}3h}{4h}} \I_{[-3h, \infty)}(x), \quad x \in \R.$$
Consequently, for $(h,\theta) = (n^{-1/2},1)$, Proposition \ref{locexprepr} and relation (\ref{PropqI}) yield
\begin{align*}
\varepsilon^{\text{loc}}_{n}(0,0) & = \E \pl{\varPi_e \I_{[0,\infty)}(W_1) - \I_{[0,\infty)}(W_1)} + \E \pl{\pr{\varPi_o \varPi_e \I_{[0,\infty)}(W_1) - \varPi_e \I_{[0,\infty)}(W_1)}q(W_1)} \nonumber\\
& = \int_{-2h}^{0} \frac{x{+}2h}{2h} p(x)dx + \int_{-3h}^{h} \pl{\frac{x{+}3h}{4h} - \pr{\frac{x{+}2h}{2h} \I_{[-2h,0)}(x) + \I_{[0,\infty)}(x)}}q(x)p(x)dx \nonumber\\
& = \int_{-2h}^0 \frac{x{+}2h}{2h}(1-q(x))p(x)dx + \int_{-3h}^{0} \frac{x{+}3h}{4h}q(x)p(x) dx + \int_0^h \frac{x{-}h}{4h}q(x)p(x)dx \nonumber\\
& \geq p(h) \int_{-h}^{0} \frac{x{+}2h}{2h}(1-q(x))dx
\end{align*}
by the symmetry of the functions $p > 0$ and $q \in [0,1]$. Moreover, in terms of the function $\varrho$ defined in \eqref{varrhon}, we deduce for $x \in (-h,0)$ that
\begin{align}\label{one-minus-q-estimate}
1-q(x) = 1 - \frac{d_o(x)}{h} - \varrho(x) \geq \frac{d_e(x)}{h} - \abs{\varrho(x)} \geq \frac{|x|}{h} - 3h(h+\sqrt{2}),
\end{align}
where the last inequality on the right-hand side of \eqref{one-minus-q-estimate} follows by applying relations \eqref{PropqIII} , \eqref{bridge_sampled_at_hitting_time_estimate}, and \eqref{bridge_hitting_upper_estimate}
for $k = 0$ and $\sigma = \theta = 1$;
\begin{align*}
|\varrho(x)| = \frac{1}{h} \E_{\tilde{\p}}\big[B^{0,1,x}_{\tilde{H}_{(x,h+x)}}\big] & \leq \frac{1}{h} \Big(\abs{x} + 2(|x| \vee (x{+}h)) + 3\sqrt{2}\Big) \E_{0,1,x}[H_{(x,h+x)}]\\
& \leq (3h+3\sqrt{2}) \frac{|x|}{h}\Big(2h - |x|\Big)\\
& \leq 3h(h + \sqrt{2}).
\end{align*}
Hence, there exist constants $C_1, C_2 > 0$ not depending on $h$ such that
\begin{align*}
\varepsilon^{\text{loc}}_{n}(0,0) & \geq p(h)\int_{-h}^0 \frac{x{+}2h}{2h} \frac{|x|}{h}dx - 3h(h+\sqrt{2}) p(h)\int_{-h}^0 \frac{x{+}2h}{2h}dx \geq \pl{C_1 h - C_2 h^2(h+\sqrt{2})}p(h).
\end{align*}
The relation $h = n^{-1/2}$ then implies that $\liminf_{n \to \infty} n^{1/2} \varepsilon^{\text{loc}}_{n}(0,0) \geq C_1p(0) > 0$.
\end{proof}
\begin{remark}
In \cite[Proposition 9.8]{Walsh} it is stated that the rate for the local error is $h$ (i.e.~$n^{-1/2}$) instead of $h^2$ (i.e.~$n^{-1}$) whenever the terminal condition $g$ has a discontinuity at a non-lattice point $x \notin \Z^h$. By contrast, Proposition \ref{localerrorrepresentationPROP} implies that only the jumps that occur at even lattice points contribute to the error. This discrepancy is a result of the choice of different step functions:
In \cite{Walsh}, only step functions of the type $\tilde{\I}_{[a,\infty)} := \I_{(a,\infty)} + \frac{1}{2}\IN{a}$ are considered.
\end{remark}
\subsection{The local error for $g \in \C^{0,\alpha}_{\mathrm{exp}}$}\label{Holdersection}
A function $g: \R \to \R$ is called locally $\alpha$-H\"older continuous (write $g \in \Calpha$), if for each compact $K \subset \R$
$$\sup_{x,y \in K, \, x \neq y} \frac{\abs{g(x) - g(y)}}{\abs{x-y}^{\alpha}} < \infty.$$
The class $\cae$ (see Definition {\ref{caedef}) consists of all locally $\alpha$-H\"older continuous functions with exponentially bounded H\"older constants in the sense of (\ref{superholder-condition}). In fact, $\cae \subset \Calpha \cap \EB$, $\alpha \in (0,1]$, and this inclusion is strict at least for $\alpha = 1$: The function $f(x) = \sin(e^{x^2} - 1)$ belongs to $C^{0,1}_{\text{loc}} \cap \EB$, whereas $f \notin C^{0,1}_{\text{exp}}$, since $f' \notin \EB$.
\smallskip

Recall that $p = p(\, \cdot \,, \theta)$ denotes the density of $X_{\theta}$ and that $\varepsilon^{\text{loc}}_{h,\theta}(g) = \E[g(X_{\tau_{J}})-g(X_{\TT})]$.
\begin{proclaim}\label{HoldKestimates}
Let $(h,\theta) \in (0,\infty) {\times} (0,T]$. Suppose that $g \in \cae$ and that $A, \beta \geq 0$ are as in (\ref{superholder-condition}). Then, for every $x_0 \in \R$ it holds that
\begin{align*}
\big|\E[g(x_0{+}X_{\tau_{J}})-g(x_0{+}X_{\theta})]\big| \leq 2^{3+\alpha} h^{\alpha} Ae^{2\beta h + \beta\abs{x_0} + \beta^2 \sigma^2 \theta/2}.
\end{align*}
\end{proclaim}
\begin{proof}
The property $g \in \cae$ implies that both $g$ and $g^{x_0} = g(x_0{+}\, \cdot \,)$ belong to $\EB$, and
\begin{align}\label{holderabsest}
\abs{\varepsilon^{\text{loc}}_{h,\theta}(g^{x_0})} & \leq \E|\varPi_e g^{x_0}(X_{\TT}) - g^{x_0}(X_{\TT})| + \E|\varPi_o \varPi_e g^{x_0}(X_{\TT}) - \varPi_e g^{x_0}(X_{\TT})|
\end{align}
holds by Proposition \ref{locexprepr}. Moreover, whenever $x \in [2kh, (2k{+}2)h]$ for some $k \in \Z$,
\begin{align*}
\abs{\varPi_e g^{x_0}(x){-}g^{x_0}(x)} & \leq \frac{(2k{+}2)h{-}x}{2h}\abs{g^{x_0}(2kh)-g^{x_0}(x)} {+} \frac{x{-}2kh}{2h}\abs{g^{x_0}((2k{+}2)h) {-} g^{x_0}(x)}\\
& \leq 2^{\alpha} h^{\alpha} A e^{\beta \abs{x_0} + 2\beta h + \beta\abs{x}}
\end{align*}
since $g \in \cae$ and $\abs{2kh} \vee \abs{(2k{+}2)h} \leq 2h + |x|$. Hence, by $\E\pl{e^{\beta \abs{X_{\theta}}}} \leq 2 e^{\beta^2 \sigma^2 \TT/2}$,
\begin{align}\label{locerrorholderexpcomp}
\E\abs{\varPi_e g^{x_0}(X_{\TT}) - g^{x_0}(X_{\TT})} & \leq 2^{\alpha} h^{\alpha} A e^{\beta\abs{x_0} + 2\beta h} \sum_{k = -\infty}^{\infty} \int_{2kh}^{(2k{+}2)h} e^{\beta \abs{x}} p(x) dx \nonumber\\
& \leq 2^{1+\alpha} h^{\alpha} Ae^{\beta\abs{x_0} + 2\beta h + \beta^2 \sigma^2 \TT/2}.
\end{align}
For the remaining expectation on the right-hand side of (\ref{holderabsest}), observe that if $y,z \in [2mh,(2m{+}2)h]$ for given $m \in \Z$, then
\begin{align}\label{VarPieLipbound}
\abs{\varPi_e g^{x_0}(y) - \varPi_e g^{x_0}(z)} & = \Big|\frac{z-y}{2h}g^{x_0}(2mh) - \frac{z-y}{2h}g^{x_0}((2m{+}2)h)\Big| \nonumber\\
& \leq 2^{\alpha} h^{\alpha} A e^{\beta\pr{\abs{x_0} + \abs{2mh}\vee \abs{(2m{+}2)h}}}.
\end{align}
Therefore, for $x \in [2kh, (2k{+}2)h]$ with $k \in \Z$, by (\ref{VarPieLipbound}) it holds that
\begin{align*}
\abs{\varPi_o \varPi_e g^{x_0}(x) - \varPi_e g^{x_0}(x)} & = \frac{(2k{+}1)h-x}{2h}\abs{\varPi_e g^{x_0}((2k{-}1)h) - \varPi_e g^{x_0}(x)} \nonumber \\
& \quad +\frac{x-(2k{-}1)h}{2h}\abs{\varPi_e g^{x_0}((2k{+}1)h) - \varPi_e g^{x_0}(x)}\nonumber\\
& \leq \frac{(2k{+}1)h-x}{2h} Ae^{\beta|x_0|}\big[ e^{\beta(|(2k{-}2)h| \vee |2kh|)} + e^{\beta(\abs{2kh} \vee \abs{(2k{+}2)h})}\big] 2^{\alpha} h^{\alpha}\nonumber \\
& \quad + \frac{x-(2k{-}1)h}{2h} A e^{\beta(|x_0| + \abs{2kh} \vee \abs{(2k{+}2)h})} 2^{\alpha} h^{\alpha}\nonumber\\
& \leq 2^{\alpha+1} h^{\alpha} A e^{\beta(|x_0| + |(2k{-}2)h| \vee |(2k{+}2)h|)}.
\end{align*}
Using the symmetry (in $k$) of this upper bound, we obtain
\begin{align}\label{locerrorholderexpcomp2}
& \E |\varPi_o \varPi_e g^{x_0}(X_{\TT}) - \varPi_e g^{x_0}(X_{\TT})| \nonumber\\
& \leq 2^{\alpha + 2} h^{\alpha} A \sum_{k = 0}^{\infty} \int_{2kh}^{(2k{+}2)h} e^{\beta(|x_0| + |(2k{-}2)h| \vee |(2k{+}2)h|)} p(x) dx \nonumber\\
& \leq 2^{\alpha + 2} h^{\alpha} A e^{\beta\abs{x_0} + 2\beta h} \sum_{k = 0}^{\infty} \int_{2kh}^{(2k{+}2)h} e^{\beta x} p(x) dx \nonumber\\
& \leq 2^{\alpha + 2} h^{\alpha} A e^{\beta\abs{x_0} + 2\beta h + \beta^2 \sigma^2 \TT/2}.
\end{align}
The claim follows by applying the estimates (\ref{locerrorholderexpcomp}) and (\ref{locerrorholderexpcomp2}) to (\ref{holderabsest}).
\end{proof}
\begin{corollary}\label{LocalErrorHoldCor}
Let $n \in 2\N$. Suppose that $g \in \cae$ and that $\beta \geq 0$ is as in (\ref{superholder-condition}).
Then, under Assumption \ref{assumption}, there exist a constant $C > 0$ such that for all $(t_0, x_0) \in [0,T) \times \R$,
\begin{align*}
\abs{\varepsilon^{\text{loc}}_{n}(t_0,x_0)} \leq \frac{C \sigma^{\alpha} T^{\alpha/2}}{n^{\alpha/2}}e^{\beta \abs{x_0} + 2\beta^2\sigma^2 T}.
\end{align*}
\end{corollary}
\begin{proof}
Since $h = \sigma\sqrt{\smash[b]{\frac{T}{n}}} \leq \sigma \sqrt{\smash[b]{\frac{T}{2}}}$ and $\varepsilon^{\text{loc}}_{n}(t_0,x_0) = \varepsilon^{\text{loc}}_{h, \tn}(g^{x_0})$ by Assumption \ref{assumption} and (\ref{errorlocal}), Proposition \ref{HoldKestimates} implies the result.
\end{proof}
\section{The global error}\label{Sectionglobalerror}
Our aim is to derive an upper bound for the global error
\begin{align*}
\varepsilon^{\text{glob}}_{n}(t_0,x_0) = \E[g(x_0{+}X_{\tau_{\nt}})- g(x_0 {+} X_{\tau_{J_n}})]
\end{align*}
defined in (\ref{errorglobal}), where $g$ is an exponentially bounded Borel function and $(X_{\tau_k})_{k=0,1, \dots}$ is the random walk considered in Subsection \ref{firstexittimes}. For this purpose, we need a collection of estimates related to 
the behavior of the random walk $(X_{\tau_k})$ and the stopping time $J_n$. A part of these are given in this section, while the more involved ones are presented later in Section \ref{MomenttiestimaatitJiille}.\\
\newline
\textbf{Note:} \textit{The Assumption \ref{assumption} is taken as a standing assumption throughout Section \ref{Sectionglobalerror}.}\\
\newline
Recall the definitions of $\nt$ and $\theta_n$ 
given in (\ref{deftnnt}). Recall also that $J_n(\omega) = \inf \{2m \in 2\N: \tau_{2m}(\omega) > \theta_n\}$ as was defined in (\ref{defofJ}). A result similar to the lemma below was proved in \cite[Corollary 11.4]{Walsh}.
\begin{lemma}\label{expExp}
For any $b \geq 0$, it holds that
\begin{align}
(i) & \quad \quad \E \big[e^{b |X_{\tau_{\nt}}|}\big] \leq 2 e^{b^2\sigma^2T/2} \label{AAA},\\
(ii) & \quad \quad \E \big[e^{b |X_{\tau_{\jtn}}|}\big] \leq 2 e^{b \sigma \sqrt{2T} + b^2\sigma^2T/2}. \label{BBB}
\end{align}
\end{lemma}
\begin{proof}
$(i)$: Since $X_{\tau_{\nt}} = \sum_{k=1}^{\nt} \Delta X_{\tau_k}$, where $(\Delta X_{\tau_k})_{k=1,2, \dots}$ is a sequence of i.i.d. random variables with
$\p(\Delta X_{\tau_k} = \pm h) = 1/2$ for $h = \sigma\sqrt{\smash[b]{\frac{T}{n}}}$ (see Subsection \ref{firstexittimes}),
\begin{align*}
\E \big[e^{b|X_{\tau_{\nt}}|}\big] \leq 2 \E \big[e^{b X_{\tau_{\nt}}}\big] = 2 \pr{\E \big[e^{b \Delta X_{\tau_1}} \big]}^{\nt}= 2 \pr{\cosh(bh)}^{\nt} \leq 2 e^{b^2 h^2 \nt/2} \leq 2 e^{b^2 \sigma^2 T/2}
\end{align*}
by the inequality $\cosh(x) \leq e^{x^2/2}$, $x \in \R$.

$(ii)$: Firstly, observe that by the definition of $J_n$ we have $\abs{X_{\tau_{\jtn}}{-} X_{\tn}} \leq 2h$. Secondly, since for a standard normal $Z$ random variable it holds that $\E\big[ e^{u \abs{Z}}\big] \leq 2e^{u^2/2}$ ($u \in \R$),
\begin{align*}
\E\big[e^{b |X_{\tau_{\jtn}}|}\big] \leq \E\big[e^{b |X_{\tau_{\jtn}}-X_{\tn}| + b\abs{X_{\tn}}}\big] & \leq e^{2bh} \E[e^{b  \sigma\sqrt{\tn} \abs{Z}}] \leq 2e^{b\sigma \sqrt{2T} + b^2\sigma^2 T/2}.
\end{align*}
\end{proof}
In Proposition \ref{FourEstimatesForTheGlobalError}, we present some more upper bounds which are used to estimate the global error. 
\begin{proclaim}\label{FourEstimatesForTheGlobalError}
\hspace{0.1mm}
\begin{itemize}
\item[$(i)$]
Suppose that $p \geq 0$, $g \in \EB$, and that $b \geq 0$ is as in (\ref{exponentialboundednesscondition}). Then there exists a constant $C_{p} > 0$ such that for all $x_0 \in \R$,
\begin{align}\label{DDD}
\sup_{(n,t_0) \in 2\N{\times}[0,T)} \abs{\E\pl{\pr{ \tfrac{|X_{\tau_{{\nt}}}|}{\sqrt{\sigma^2 \tn}}}^p g(x_0{+}X_{\tau_{{\nt}}})}} \leq C_{p} e^{b\abs{x_0} + b^2\sigma^2 T}.
\end{align}
\end{itemize}
Moreover, for every $p > 0$ there exists a constant $C_p > 0$ such that 
\begin{align}
(ii) & \quad \quad \sup_{(n,t_0) \in 2\N{\times}[0,T)} \nt^{p} \p\big(\big|X_{\tau_{\nt}}/h\big| > \nt^{3/5}\big) \leq C_p, \quad \quad \quad \quad \label{EEE}\\
(iii) & \quad \quad \sup_{(n,t_0) \in 2\N{\times}[0,T)} \nt^{p} \p\big(\abs{J_n - \nt} > \nt^{3/5}\big) \leq C_p. \quad \quad \quad \quad \label{FFF}
\end{align}
\end{proclaim}
\begin{proof}
$(i)$: Observe that
\begin{align*}
S_{\nt} := \frac{X_{\tnt}}{\sqrt{\sigma^2 \tn}} = \frac{1}{\sqrt{\sigma^2 \tn}}\sum_{k=1}^{\nt} \Delta X_{\tau_k} \stackrel{d}{=} \frac{1}{\sqrt{\nt}} \sum_{k=1}^{\nt} \xi_i,
\end{align*}
where $(\xi_i)_{i=1,2, \dots}$ is an i.i.d.~Rademacher sequence (see Subsection \ref{firstexittimes}). Hence,
\begin{align*}
\E\pl{e^{t S_{\nt}}} = (\cosh(\tfrac{t}{\sqrt{\nt}}))^{\nt} \leq (e^{t^2/(2\nt)})^{\nt} = e^{t^2/2}, \quad t \in \R.
\end{align*}
Consequently, by the symmetry of $S_{\nt}$ and Markov's inequality,
\begin{align*}
\p(\abs{S_{\nt}} > t) & = 2 \p(e^{t S_{\nt}} > e^{t^2}) \leq 2 e^{-t^2} \E\pl{e^{t S_{\nt}}} \leq 2e^{-t^2/2}, \quad t > 0,
\end{align*}
and thus, uniformly in $(n,t_0)$, for $p > 0$,
\begin{align}\label{GGG}
\E |S_{\nt}|^{p} = p\int_{0}^{\infty} t^{p-1} \p(\abs{S_{\nt}} > t) dt \leq 2p \int_{0}^{\infty} t^{p-1} e^{-t^2/2} dt := \tilde{C}_p < \infty.
\end{align}
H\"older's inequality, (\ref{GGG}), and (\ref{AAA}) then imply that
\begin{align*}
\abs{\E\pl{\abs{S_{\nt}}^p g(x_0{+}X_{\tau_{{\nt}}})}} & \leq Ae^{b|x_0|} \pr{\E|S_{\nt}|^{2p}}^{1/2} \big(\E\big[e^{2b|X_{\tau_{\nt}}|}\big]\big)^{1/2}
\leq 2A \tilde{C}_{2p}^{1/2} e^{b|x_0|+b^2\sigma^2T}.
\end{align*}
This proves (\ref{DDD}) for $p > 0$, and the case $p = 0$ can be seen from the last line as well.

$(ii)$: Since $h\sqrt{\nt} = \sqrt{\sigma^2 \tn}$, by Markov's inequality and (\ref{GGG}) we obtain
\begin{align}\label{kiireformula1}
\p\big(|X_{\tau_{\nt}}/h| > \nt^{3/5}) = \p \big(\abs{S_{\nt}} > \nt^{1/10} \big) \leq \E\abs{S_{\nt}}^{q} \nt^{-q/10} \leq C_q \nt^{-q/10}
\end{align}
for all $q > 0$. Choose $q \geq 10p$ and multiply both sides of (\ref{kiireformula1}) by $\nt^{p}$ to obtain (\ref{EEE}). 

$(iii)$: For every $K > 0$, Markov's inequality and Proposition \ref{PropAzumaMoments} below imply that
\begin{align}\label{azumanicebound}
\p\big(\abs{J_n - \nt} > \nt^{3/5}\big) \leq \E\abs{J_n - \nt}^K \nt^{-3K/5}  \leq C_{K} \nt^{-K/10}
\end{align}
for some constant $C_K > 0$. For given $p > 0$, it remains to choose $K \geq 10p$ and multiply both sides of (\ref{azumanicebound}) by $\nt^{p}$. 
\end{proof}
The proof of the main result of this section follows closely the proof of \cite[Theorem 8.1]{Walsh}.
\begin{proclaim}\label{globerrorfinally}
Let $n \in 2\N$. Suppose that $g \in \EB$ and that $b \geq 0$ is as in (\ref{exponentialboundednesscondition}). Then there exists a constant $C > 0$ such that for all 
$(t_0,x_0) \in [0,T) \times \R$,
\begin{align}\label{glob-error-upper-bound}
\abs{\varepsilon^{\text{glob}}_n(t_0,x_0)} \leq \frac{C T}{n(T-t^n_k)} e^{b\abs{x_0} +3 b^2 \sigma^2 T}, \quad t_0 \in [t^n_k, t^n_{k+1}), \quad 0 \leq k < \tfrac{n}{2}.
\end{align}
\end{proclaim}
\begin{proof}
Define a set
\begin{align}\label{GammaJoukko}
\Gamma_{\nt} := \big\{\big|X_{\tau_{\nt}}/h\big| \vee \abs{J_{\nt}{-}\nt} \leq \nt^{3/5}\big\}
\end{align}
and decompose the error $\varepsilon^{\text{glob}}_n(t_0,x_0)$ into the sum of expectations $E^{(1)}$ and $E^{(2)}$, where
\begin{align}
E^{(1)} := \E[g(x_0{+}X_{\tau_{\nt}})-g(x_0 {+} X_{\tau_{J_n}});\Gamma_{\nt}], \quad E^{(2)} := \E[g(x_0{+}X_{\tau_{\nt}})-g(x_0 {+} X_{\tau_{J_n}});\Gamma_{\nt}^{\complement}]. \label{GammaExpectSum}
\end{align}
Using the estimates of Lemma \ref{expExp} and Proposition \ref{FourEstimatesForTheGlobalError}, it can be shown that
\begin{align}\label{E2poisGE}
\big|E^{(2)}\big| \leq \tilde{C}_0 \nt^{-3/2} e^{b\abs{x_0} + b^2 \sigma^2 T + b\sigma \sqrt{2 T}}
\end{align}
for some constant $\tilde{C}_0 > 0$; this is done in Lemma \ref{Q3isok} $(i)$. Estimation of $\abs{E^{(1)}}$ requires more subtlety. Denote the probability mass functions of $X_{\tau_{\nt + k}}/h$ and $J_n{-}\nt$ by
\begin{align}\label{pmfs}
P_{\nt + k}(x) & := \p(X_{\tau_{\nt + k}} = hx) \quad \text{and } \quad P^{J}_{\nt}(x) := \p(J_{n}{-}\nt = x), \quad \quad x \in \Z.
\end{align}
By Lemma \ref{Q3isok} $(ii)$, there exists a constant $\tilde{C}_1 > 0$ such that
\begin{align}\label{toughproofpt1CHANGED}
\big|E^{(1)}\big| & \leq \bigg|\sum_{k=2-\nt}^{\infty} \sum_{x= -\nt}^{\nt} g(x_0{+}xh) P^{J}_{\nt}(k) P_{\nt}(x)\Big(\frac{k}{2\nt} - \frac{3k^2 + 4kx^2}{8\nt^2} + \frac{3k^2 x^2}{4 \nt^3} - \frac{k^2 x^4}{8 \nt^4} \Big)\bigg| \nonumber\\
& \quad \quad + \tilde{C}_1 \nt^{-3/2} e^{b\abs{x_0} + b^2 \sigma^2 T}.
\end{align}
Next, we use relation (\ref{globercompformula}) in order to rewrite the double sum on the right-hand side of (\ref{toughproofpt1CHANGED}) as
\begin{align}\label{longrimpsu}
& E^{(3)} := \sum_{k=2-\nt}^{\infty} \sum_{x= -\nt}^{\nt} g(x_0{+}xh) P^{J}_{\nt}(k) P_{\nt}(x) \Big(\frac{k}{2\nt} - \frac{3k^2 + 4kx^2}{8\nt^2} + \frac{3k^2 x^2}{4 \nt^3} - \frac{k^2 x^4}{8\nt^4} \Big) \nonumber\\
& = \frac{1}{{\nt}} \Bigg\{ \frac{1}{2} \E \pl{g(x_0{+}X_{\tau_{{\nt}}})} \E [J_n{-}\nt] - \frac{3}{8} \E \pl{g(x_0{+}X_{\tau_{{\nt}}})} \E \pr{\tfrac{J_n{-}\nt}{\sqrt{\nt}}}^2 \nonumber\\
& \quad \quad - \frac{1}{2}\E\big[\pr{\tfrac{X_{\tau_{{\nt}}}}{\sqrt{\sigma^2 \tn}}}^2 g(x_0{+}X_{\tau_{{\nt}}})\big] \E[J_n{-}\nt]
 {+} \frac{3}{4} \E\big[\pr{\tfrac{X_{\tau_{{\nt}}}}{\sqrt{\sigma^2 \tn}}}^2 g(x_0{+}X_{\tau_{{\nt}}}) \big] \E \pr{\tfrac{J_n{-}\nt}{\sqrt{\nt}}}^2 \nonumber\\
& \quad \quad \quad - \frac{1}{8} \E\pl{\pr{ \tfrac{X_{\tau_{{\nt}}}}{\sqrt{\sigma^2 \tn}}}^4 g(x_0{+}X_{\tau_{{\nt}}})} \E\pr{\tfrac{J_n{-}\nt}{\sqrt{\nt}}}^2\Bigg\}\nonumber\\
& = \frac{1}{{\nt}} \Bigg\{ \E \pl{g(x_0{+}X_{\tau_{{\nt}}})} \pr{\frac{1}{2}\E [J_n{-}\nt] - \frac{3}{8} \E \pr{\tfrac{J_n{-}\nt}{\sqrt{\nt}}}^2 }\nonumber\\
& \quad \quad + \E\big[\pr{\tfrac{X_{\tau_{{\nt}}}}{\sqrt{\sigma^2 \tn}}}^2 g(x_0{+}X_{\tau_{{\nt}}})\big] \pr{\frac{3}{4} \E \pr{\tfrac{J_n{-}\nt}{\sqrt{\nt}}}^2- \frac{1}{2} \E[J_n{-}\nt]}\nonumber\\
& \quad \quad - \frac{1}{8} \E\pl{\pr{ \tfrac{X_{\tau_{{\nt}}}}{\sqrt{\sigma^2 \tn}}}^4 g(x_0{+}X_{\tau_{{\nt}}})} \E\pr{\tfrac{J_n{-}\nt}{\sqrt{\nt}}}^2 \Bigg\}.
\end{align}
By Proposition \ref{jtntjn2}, there exist constants $c_1, c_2 > 0$ such that $\abs{\E[J_n{-}\nt] - \frac{4}{3}} \leq \frac{c_1}{\sqrt{\nt}}$ and\\ $\E\big[\big(\tfrac{J_n{-}\nt}{\sqrt{\nt}}\big)^2 - \frac{2}{3}\big] \leq \frac{c_2}{\sqrt{\nt}}$, and thus
\begin{align*}
& \Big|\frac{1}{2}\E [J_n{-}\nt] - \frac{3}{8} \E \pr{\tfrac{J_n{-}\nt}{\sqrt{\nt}}}^2 -\frac{5}{12}\Big| \leq \frac{c_1+c_2}{\sqrt{\nt}}, \quad \Big|\frac{1}{8}\E\pr{\tfrac{J_n{-}\nt}{\sqrt{\nt}}}^2 - \frac{1}{12}\Big| \leq \frac{c_2}{\sqrt{\nt}} \quad \text{ and }\\
& \Big|\frac{3}{4} \E \pr{\tfrac{J_n{-}\nt}{\sqrt{\nt}}}^2 -\frac{1}{2} \E[J_n{-}\nt] + \frac{1}{6}\Big| \leq \frac{c_1+c_2}{\sqrt{\nt}}.
\end{align*}
Consequently, by (\ref{longrimpsu}) and (\ref{DDD}), there exist constants $\tilde{C}_2, \tilde{C}_3 > 0$ such that
\begin{align}\label{structureofglobber}
\big|E^{(3)}\big| & \leq \frac{5}{12 \nt}\abs{\E \pl{g(x_0{+}X_{\tau_{\nt}})}} + \frac{1}{6 \nt}\Big|\E\Big[\Big(\tfrac{X_{\tau_{\nt}}}{\sqrt{\sigma^2 \theta_n}}\Big)^2 g(x_0{+}X_{\tau_{\nt}})\Big]\Big| \nonumber\\
& \quad \quad + \frac{1}{12 \nt} \Big|\E\Big[\Big( \tfrac{X_{\tau_{\nt}}}{\sqrt{\sigma^2 \theta_n}}\Big)^4 g(x_0{+}X_{\tau_{\nt}})\Big]\Big| + \frac{\tilde{C}_2 e^{b\abs{x_0} + b^2 \sigma^2 T}}{\nt^{3/2}} \nonumber\\
& \leq \frac{\tilde{C}_3}{\nt} e^{b|x_0|+b^2\sigma^2 T} + \frac{\tilde{C}_2 e^{b\abs{x_0} + b^2 \sigma^2 T}}{\nt^{3/2}}.
\end{align}
To complete the proof, it remains to observe that $\frac{1}{\nt^{3/2}} \leq \frac{1}{\sqrt{2}} \frac{1}{\nt}$, to combine (\ref{GammaExpectSum}), (\ref{E2poisGE}), (\ref{toughproofpt1CHANGED}), and (\ref{structureofglobber}), and to recall that $\nt T = n(T-t^n_k)$ for $t_0 \in [t^n_k, t^n_{k+1})$.
\end{proof}
\section{Moment estimates for the stopping time $J_n$}\label{MomenttiestimaatitJiille}
In this section we present moment estimates for the random variable
$J_n = \inf \set{2m \in 2\N: \tau_{2m} > \tn}$
introduced in (\ref{defofJ}), which are used for the estimation of the global error in Section \ref{Sectionglobalerror}.

\subsection{Estimates for the first and the second moment of \texorpdfstring{$J_n-\nt$}{J_n-\nt}}
The purpose of this subsection is to provide estimates for the first and the second moment of the random variable $J_n - \nt$.
We begin by deriving an estimate for the expectation $\E[\tau_{J_n}]$ and then use martingale techniques to
obtain estimates for $\E[J_n] - \nt$ and for $\E(J_n - \nt)^2$. The results of this subsection are closely related to \cite[Proposition 11.2]{Walsh}. 
Recall the random times $J = \inf \set{2m \in 2\N: \tau_{2m} > \TT}$ and $L = \sup\set{m \in \N_{0}: \tau_m < \TT}$
defined for each $(h,\theta) \in (0,\infty) {\times} (0,T]$ in (\ref{Jwithoutn}) and (\ref{DefOfL}). Recall also the functions $q$, $d_o, d_e$, and $\varrho$, defined in (\ref{defq}), (\ref{defofdode}), and (\ref{varrhon}), respectively.
\begin{proclaim}\label{simplify}
Suppose that $(h,\TT) \in (0,\infty) {\times} (0,T]$. Then
\begin{align}
(i) & \quad \E \pl{\tau_{J} - \TT | \F_{\TT}} = \sigma^{-2} (h^2-d^2_o(X_{\TT})) \quad \p\text{-a.s. on } \{L \text{ odd}\}, \nonumber\\
(ii) & \quad \E \pl{\tau_{J} - \TT | \F_{\TT}} = \sigma^{-2}(2h^2 - d^2_e(X_{\TT})) \quad \p\text{-a.s. on } \{L \text{ even}\}, \nonumber\\
(iii) & \quad \frac{\E[J(\tau_J - \TT)]}{\E[\tau_1] \E[J]} \in [0, 2], \nonumber\\
(iv) & \quad \E[\tau_{J} - \TT|X_{\TT} = x] = \sigma^{-2}(h^2-d^2_o(x))\big(1-q(x)\big) + \sigma^{-2}(2h^2 - d^2_e(x))q(x) \nonumber\\
& \quad \text{Leb-a.e. on } \R. \nonumber
\end{align}
\end{proclaim}
\begin{proof}
Items $(i)$, $(ii)$, and $(iv)$ are proved in \cite[p.~348 and p.~356]{Walsh}. For the convenience of the reader, we give the general idea for the proof of these statements.

$(i)$--$(ii):$ For all $k \in \Z$, let $A_{2k+1} := \set{X_{\tl} = (2k{+}1)h}$ and $B_{2k} = \set{X_{\tl} = 2kh}$.
The Markov property of $(X_t)_{t \geq 0}$ implies that $\p$-a.s~on $A_{2k+1}$,
\begin{align*}
\E\pl{\tau_J - \TT| \F_{\TT}} & = \sigma^{-2} \abs{2kh - X_{\TT}}((2k{+}2)h - X_{\TT})\\
& = \sigma^{-2} d_e(X_{\TT})(h+d_o(X_{\TT}))\\
& = \sigma^{-2} (h^2 - d_o^2(X_{\TT})).
\end{align*}
A similar observation applies to $(ii)$ by first writing $\tau_J = \tau_{J} - \tau_{J-1} + \tau_{J-1}$ on 
$B_{2k}$, since $\p$-a.s. on $B_{2k}$,
\begin{align*}
\tau_{J-1} & = \inf \set{t \geq \TT: X_t \notin ((2k{-}1)h, (2k{+}1)h)}, \text{ and }\\
\tau_J & = \inf \set{t \geq \tau_{J-1} : \abs{X_t - X_{\tau_{J-1}}} = h}.
\end{align*}

$(iii):$ Since $J$ is $\F_{\TT}$-measurable, $\E[J (\tau_J - \TT)] = \E[J \E[\tau_J - \TT|\F_{\TT}]]$. By $(i)$ and $(ii)$,
\begin{align*}
\E\pl{\tau_J - \TT|\F_{\TT}}\IN{L \text{ odd\,}} & = \sigma^{-2}(h^2 - d^2_o(X_{\TT}))\IN{L \text{ odd}} \quad \p\text{-a.s.}, \quad \text{ and }\\
\E\pl{\tau_J - \TT|\F_{\TT}}\IN{L \text{ even}} & = \sigma^{-2}(h^2 - d^2_o(X_{\TT}))\IN{L \text{ even}} + 2\sigma^{-2}hd_o(X_{\TT})\IN{L \text{ even}} \quad \p\text{-a.s.},
\end{align*}
where we used the equality $d^2_e(x) = (h-d_o(x))^2, x \in \R$. Consequently, 
\begin{align*}
\sigma^2 \E\pl{J (\tau_J - \TT)} & = \E\pl{J(h^2 - d_o^2(X_{\TT}))} + 2h \E[J d_o(X_{\TT})\IN{L \text{ even}}]\\
& \leq h\E\pl{J(h - d_o^2(X_{\TT})/h + 2d_o(X_{\TT}))}.
\end{align*}
Since $d_o \in [0,h]$, $\E[\tau_1] = (h/\sigma)^2$, and $J \geq 0$, the lower bound is clear. For the upper bound, it suffices to further observe that
$h - x^2/h + 2x \leq 2h$ for $x \in [0,h]$.

$(iv):$ By the tower property of the conditional expectation and items $(i)$ and $(ii)$, $\p$-a.s,
\begin{align*}
\E[\tau_{J} - \TT|X_{\TT}] & = \E\pl{ \E[\tau_{J} - \TT|\F_{\TT}] \IN{L \text{ odd}} + \E[\tau_{J} - \TT|\F_{\TT}] \IN{L \text{ even}}\Big|X_{\TT}}\\
& = \sigma^{-2} (h^2{-}d^2_o(X_{\TT}))\p(L \text{ odd}|X_{\TT}) + \sigma^{-2}(2h^2{-}d^2_e(X_{\TT}))\p(L \text{ even}|X_{\TT}).
\end{align*}
Moreover, $q(x) = \p(L \text{ odd}|X_{\TT} = x) = 1 - \p(L \text{ even}|X_{\TT} = x)$ on $\R \backslash \Z^h$ by (\ref{PropqI}),
and the claim follows.
\end{proof}

\begin{lemma}\label{PropJ}
Suppose that Assumption \ref{assumption} holds. Then for all $(n,t_0) \in 2\N \times [0,T)$,
\begin{align}\label{PropJestimaaatti}
\abs{\E \big[ \tau_{J_n} \big] - \tn - \frac{4}{3}\frac{T}{n}} \leq \frac{48}{\sqrt{\nt}} \frac{T}{n}.
\end{align}
\end{lemma}
\begin{proof}
For each $(h,\theta) \in (0,\infty)\times(0,T]$, define
\begin{align}
I_1(h,\TT) & := \sigma^{-2}\int_{\R} p(x) \Big(h^2 - d^2_o(x) \Big)(1-q(x))dx,\label{busyweekI1}\\
I_2(h,\TT) & := \sigma^{-2}\int_{\R} p(x) \Big(2h^2 - d^2_e(x) \Big)q(x)dx \label{busyweekI2}.
\end{align}
Then, by Proposition \ref{simplify} $(iv)$, it holds for Leb-a.e.~$x \in \R$ that
\begin{align*}
\E[\tau_{J_n}] - \tn = \int_{\R} \E \pl{\tau_{J_n} - \tn|X_{\tn} = x} p(x) dx = I_1(h, \tn) + I_2(h, \tn).
\end{align*}
By Lemma \ref{Firstandsecondintegral} below and by the fact that $\frac{h^2}{\sigma^2} = \frac{T}{n}$ and $\frac{h}{\sigma \sqrt{\tn}} = \frac{1}{\sqrt{\nt}} \leq \frac{1}{\sqrt{2}}$, the left-hand side of (\ref{PropJestimaaatti}) is bounded from above by
\begin{align*}
\abs{I_1(h, \tn) - \frac{5}{12}\frac{h^2}{\sigma^2}} + \abs{I_2(h, \tn) - \frac{11}{12}\frac{h^2}{\sigma^2}}\leq \pr{33 + \frac{14}{\sqrt{2\pi}}}\frac{h^3}{\sigma^3\sqrt{\tn}} + \frac{12h^4}{\sigma^4 \tn} \leq \frac{48}{\sqrt{\nt}} \frac{T}{n}.
\end{align*}
\end{proof}
The estimate below will be used in the proof of Lemma \ref{Firstandsecondintegral}.
\begin{lemma}
Let $(h,\theta) \in (0,\infty){\times}(0,T]$ and denote by $p = p(\, \cdot \, ,\theta)$ the density of $X_{\theta} = \sigma W_{\theta}$. Then
\begin{align}
\Big|\!\!\sum_{m=-\infty}^{\infty} 2h p((2m{+}1)h) - 1\Big| \leq \frac{6}{\sqrt{2\pi}}\frac{h}{\sigma \sqrt{\TT}}. \label{oRiemann}
\end{align}
\end{lemma}
\begin{proof}
By the symmetry of the Gaussian density $p$, it holds that
\begin{align}\label{Sriemann}
\mathcal{S}_{h} := \sum_{m=-\infty}^{\infty} 2h p((2m{+}1)h) & = 4hp(h) + 2 \sum_{m=1}^{\infty} 2h p((2m{+}1)h).
\end{align}
In addition, since $p$ is decreasing on $[0,\infty)$,
\begin{align*}
\int_{3h}^{\infty} p(x) dx & = \sum_{m=1}^{\infty} \int_{(2m{+}1)h}^{(2m+3)h} p(x) dx \leq \sum_{m=1}^{\infty} 2h p((2m{+}1)h) \leq  \int_{h}^{\infty} p(x) dx,
\end{align*}
which together with (\ref{Sriemann}) implies that
\begin{align}\label{ducktales2}
\int_{-h}^{h} p(x) dx - 4hp(h) \leq \int_{-\infty}^{\infty} p(x) dx - \mathcal{S}_{h} \leq \int_{-3h}^{3h} p(x) dx - 4hp(h).
\end{align}
For each $\beta > 0$, the mean value theorem implies that for a constant $\xi = \xi(h, \sigma, \TT, \beta) \in (-\beta h, \beta h)$ we have
\begin{align*}
\sqrt{2 \pi \sigma^2 \theta} \int_{-\beta h}^{\beta h} p(x) dx = \int_{-\beta h}^{\beta h} e^{-\frac{x^2}{2 \sigma^2 \TT}} dx = 2 \beta h e^{-\frac{\xi^2}{2 \sigma^2 \TT}},
\end{align*}
and (\ref{oRiemann}) then follows by (\ref{ducktales2}). 
\end{proof}
\begin{lemma}\label{Firstandsecondintegral}
Let $(h,\TT) \in (0,\infty){\times}(0,T]$. Then, for $I_1$ and $I_2$ defined in (\ref{busyweekI1})--(\ref{busyweekI2}), it holds that
\begin{align*}
(i) \quad & \Big|I_1(h, \TT) - \frac{5}{12}\frac{h^2}{\sigma^2}\Big| \leq \pr{11 + \frac{9}{2\sqrt{2\pi}}}\frac{h^3}{\sigma^3\sqrt{\TT}} + \frac{4h^4}{\sigma^4 \TT},\\
(ii) \quad & \Big|I_2(h, \TT) - \frac{11}{12}\frac{h^2}{\sigma^2}\Big| \leq \pr{22 + \frac{19}{2\sqrt{2\pi}}}\frac{h^3}{\sigma^3 \sqrt{\TT}} + \frac{8 h^4}{\sigma^4 \TT}.
\end{align*}
\end{lemma}
\begin{proof}
$(i)$: Since $1 - q(x) = h^{-1}(h - d_o(x)) - \varrho(x)$ for $x \in \R$ by the definition of $\varrho$, and since 
$$\mathcal{D}(x) := \big(h^2 - d^2_o(x)\big)h^{-1} (h - \dox), \quad x \in \R,$$ 
is symmetric, periodic with period $2h$, and symmetric around $h$ on $[0,2h]$, we may decompose the integral $I_1 = I_1(h,\theta)$ into $I_1 = S_1 + S_2 + S_3$, where
\begin{align*}
S_1 & := \sigma^{-2} \sum_{m = -\infty}^{\infty} p((2m{+}1)h) \int_{0}^{2h} \mathcal{D}(x) dx,\\
S_2 & := \sigma^{-2} \sum_{m = -\infty}^{\infty} \int_{2mh}^{(2m{+}2)h} \big(p(x) - p((2m{+}1)h)\big) \mathcal{D}(x) dx,\\
S_3 & :=- \sigma^{-2} \sum_{m=-\infty}^{\infty}\int_{2mh}^{(2m{+}2)h} p(x) \Big(h^2 - d^2_o(x) \Big) \varrho(x) dx.
\end{align*}
A standard calculation yields $\int_{0}^{2h} \mathcal{D}(x) dx = 5h^3/6$, and thus
\begin{align}\label{hessuhopo}
S_{1} & = \frac{5}{12} \frac{h^2}{\sigma^2}\sum_{m =-\infty}^{\infty} \!\!2h p((2m{+}1)h) = \frac{5}{12}\frac{h^2}{\sigma^2} + \frac{5}{12}\frac{h^2}{\sigma^2}\Big(\!\!\sum_{m =-\infty}^{\infty} \!\!2h p((2m{+}1)h)-1\Big).
\end{align}
Hence, by applying the estimate (\ref{oRiemann}) to the right-hand side of (\ref{hessuhopo}), we obtain
\begin{align}\label{I1pt1}
\abs{S_{1}-\frac{5}{12}\frac{h^2}{\sigma^2}} \leq \frac{5}{2\sqrt{2\pi}} \frac{h^3}{\sigma^3 \sqrt{\TT}}.
\end{align}
In order to estimate $S_2$, notice that for each integer $k \geq 0$,
\begin{align*}
\int_{kh}^{(k{+}1)h} \mathcal{D}(x) dx = \int_{0}^{h}\mathcal{D}(x) dx = h^2 \int_{0}^{h} \Big(1- \frac{d^2_o(x)}{h^2}\Big)\Big(1 - \frac{\dox}{h}\Big) dx \leq h^3.
\end{align*}
By the symmetricity and the monotonicity properties of $p$, we thus obtain
\begin{align}\label{new-fast-estimate}
S_2 & \leq 2\sigma^{-2} \sum_{m = 0}^{\infty} \int_{2mh}^{(2m{+}1)h} \big(p(x) - p((2m{+}1)h)\big) \mathcal{D}(x) dx\nonumber\\
& \leq 2\sigma^{-2} \sum_{m = 0}^{\infty} \pr{p(2mh) - p((2m{+}1)h)} \int_{0}^h \mathcal{D}(x) dx\nonumber\\
& \leq \frac{2h^3}{\sigma^{2}} \sum_{m=0}^{\infty} \int_{2mh}^{(2m{+}1)h} \frac{x}{\sigma^2 \theta}p(x) dx\nonumber\\
& \leq \frac{2}{\sqrt{2\pi}} \frac{h^3}{\sigma^3 \sqrt{\theta}}.
\end{align}
A similar computation yields the lower bound $-\frac{2}{\sqrt{2\pi}} \frac{h^3}{\sigma^3 \sqrt{\theta}}$ for $S_2$, and consequently,
\begin{align}\label{I1pt2}
\abs{S_2} \leq \frac{2}{\sqrt{2\pi}} \frac{h^3}{\sigma^3 \sqrt{\theta}}.
\end{align}
It remains to estimate $S_3$. Using the inequality $h^2 - d^2_o(x) \leq h^2$ and estimates of Proposition \ref{PROPqprobestimate}, we have

\begin{align}\label{I1pt3}
\abs{S_3} & = \bigg|\frac{1}{\sigma^2}\sum_{m=-\infty}^{\infty} \int_{2mh}^{(2m{+}2)h} p(x) \Big(h^2 - d^2_o(x) \Big) \varrho(x) dx\bigg|
\leq \frac{h^2}{\sigma^2} \int_{\R} \abs{\varrho(x)} p(x) dx \nonumber\\
& \leq \frac{11 h^3}{\sigma^3 \sqrt{\TT}} + \frac{4 h^4}{\sigma^4 \TT}.
\end{align}
The claim then follows by applying (\ref{I1pt1}), (\ref{I1pt2}), and (\ref{I1pt3}) to the right-hand side of the estimate below,
\begin{align*}
\Big|I_1(h,\TT) - \frac{5}{12}\frac{h^2}{\sigma^2}\Big| \leq \Big|S_{1} - \frac{5}{12}\frac{h^2}{\sigma^2}\Big| + \abs{S_{2}} + \abs{S_3}.
\end{align*}

$(ii)$: The proof is similar to the proof of item $(i)$, and thus we omit most of the details. We write $q(x) = h^{-1}\dox + \varrho(x)$, let 
$$\HH(x) := \big(2h^2 - d^2_e(x) \big)h^{-1} \dox, \quad x \in \R,$$ 
and decompose $I_2 = S_4 + S_5 + S_6$, where
\begin{align*}
S_4 & := \sigma^{-2} \sum_{m=-\infty}^{\infty} p((2m{+}1)h) \int_{0}^{2h} \mathcal{H}(x) dx,\\
S_5 & := \sigma^{-2} \sum_{m=-\infty}^{\infty} \int_{2mh}^{(2m{+}2)h} (p(x) - p((2m{+}1)h) \mathcal{H}(x) dx,\\
S_6 & := \sigma^{-2} \sum_{m=-\infty}^{\infty} \int_{2mh}^{(2m{+}2)h} p(x)\Big(2h^2 - d^2_e(x) \Big)\varrho(x)dx.
\end{align*}
The fact that $\int_{0}^{2h} \HH(x) dx = \frac{11}{6}h^3$ and inequality (\ref{oRiemann}) yield the estimate
\begin{align}\label{esyksykstokaintegral}
\abs{S_{4}-\frac{11}{12}\frac{h^2}{\sigma^2}} \leq \frac{11}{2\sqrt{2\pi}}\frac{h^3}{\sigma^3 \sqrt{\TT}}.
\end{align}
By the properties of $\HH$, for each integer $k \geq 0$, it holds that
$$\int_{kh}^{(k{+}1)h} \HH(x) dx = \int_0^{h} \HH(x) dx = h^2 \int_{0}^{h} \pr{2 - \frac{h-d^2_o(x)}{h^2}}\frac{d_o(x)}{h} dx \leq 2h^3,$$
and by proceeding as in (\ref{new-fast-estimate}) with $\mathcal{D}$ replaced by $\HH$, it is easy to verify that
\begin{align}\label{tokaintegraltoka}
\abs{S_5} \leq \frac{4}{\sqrt{2\pi}} \frac{h^3}{\sigma^3 \sqrt{\theta}}.
\end{align}
Finally, since $2h^2 - d^2_e(x) \leq 2h^2$, by (\ref{I1pt3}) we obtain
\begin{align}\label{eskaksitokaintegral}
\abs{S_6} \leq 2|S_3| \leq \frac{22 h^3}{\sigma^3 \sqrt{\TT}} + \frac{8 h^4}{\sigma^4 \theta}.
\end{align}
The triangle inequality together with (\ref{esyksykstokaintegral}), (\ref{tokaintegraltoka}), and (\ref{eskaksitokaintegral}) then implies $(ii)$.
\end{proof}
\begin{proclaim}\label{jtntjn2}
Suppose that Assumption \ref{assumption} holds. Then there exists a constant $C > 0$ such that for all $(n,t_0) \in 2\N {\times} [0,T)$,
\begin{align*}
& (i) \quad \, \abs{\E[J_n] - \nt - \frac{4}{3}} \leq \frac{48}{\sqrt{\nt}}, \quad \quad (ii) \quad \bigg|\E\pr{\tfrac{J_n-\nt}{\sqrt{\nt}}}^2 - \frac{2}{3} \bigg| \leq \frac{C}{\sqrt{\nt}}.
\end{align*}
\end{proclaim}
\begin{proof}
$(i)$: Define a process $(M_k)_{k=0,1,\dots}$ by setting $M_k := {\tau_{k} - k \E \pl{\tau_1}}$ for $k \geq 0$. Since $\tau_k = \sum_{j=1}^{k} \Delta \tau_j$ is a sum of $k$ i.i.d.~random variables $\Delta \tau_j$ distributed as $\tau_1$, the process $(M_k)_{k=0,1, \dots}$ is a {$(\F_{\tau_k})_{k=0,1, \dots}$-martingale}. In addition, since $J_n$ is a $(\F_{\tau_k})_{k=0,1, \dots}$-stopping time and since $J_n{\wedge}N$ is a bounded stopping time for all $N \in \N$, 
the optional stopping theorem implies that
\begin{align*}
0 = \E \pl{M_N} = \E \pl{ \E\pl{M_N|\F_{\tau_{J_n \wedge N}}}} = \E[M_{J_n \wedge N}] = \E[\tau_{J_n \wedge N}] - \E[J_n {\wedge} N]\E[\tau_1],
\end{align*}
i.e. $\E[J_n {\wedge} N] = \E[\tau_{J_n {\wedge} N}]/\E[\tau_1]$.
Moreover, since $N \mapsto \tau_{J_n \wedge N}$ is increasing, where $\tau_{J_n}$ is an integrable upper bound (by Lemma \ref{PropJ}), the monotone convergence theorem implies
that 
\begin{align}
\E[J_n] = \lim_{N \to \infty} \E[J_n {\wedge} N] = \lim_{N \to \infty} \frac{\E[\tau_{J_n {\wedge} N}]}{\E[\tau_1]} = \frac{\E[\tau_{J_n}]}{\E[\tau_1]} < \infty. \label{jntaujtnexp}
\end{align}
From $\E[\tau_1] = \frac{h^2}{\sigma^2} = \frac{T}{n}$ (see (\ref{tauhoomoments})) we conclude that 
$$\E[\tau_1]\pr{\E[J_n - \nt] - \tfrac{4}{3}} = \E[\tau_{J_n} - \tn] - \tfrac{4}{3}\tfrac{T}{n},$$ 
and the claim then follows by Lemma \ref{PropJ}.

$(ii)$: Let $\zeta := \tn + \inf \set{t \geq 0: \abs{X_{t + \tn} - X_{\tn}} = 2h}$. Then by the Markov property and the scaling property of $(X_t)_{t \geq 0}$,
\begin{align*}
\zeta - \tn \stackrel{d}{=} \inf \set{t \geq 0: \abs{X_t} = 2h} \stackrel{d}{=} \inf \set{t \geq 0 : \abs{X_{t/4}} = h} \stackrel{d}{=} 4 \tau_1.
\end{align*}
Since $\p(\tau_{J_n} \leq \zeta)= 1$ and $\E[\tau_1^2] = \frac{5}{3}\frac{h^4}{\sigma^4} = \frac{5}{3}\frac{T^2}{n^2}$ by (\ref{tauhoomoments}), it also holds that
\begin{align}\label{taujiin-teen}
0 \leq \E(\tau_{J_n} - \tn)^2 \leq \E(\zeta- \tn)^2 = 16 \E[\tau_1^2] = \frac{80}{3} \frac{T^2}{n^2}.
\end{align}
By the definition of the process $(M_k)_{k=0,1, \dots},$
\begin{align*}
\E[M^2_{J_n}] = \E(\tau_{J_n} - J_n \E[\tau_1])^2 = \E[\tau_{J_n}^2] - 2 \E[J_n \tau_{J_n}] \E[\tau_1] + \E[J_n^2] (\E[\tau_1])^2.
\end{align*}
Moreover, since $\E[J_n] < \infty$, Wald's second identity applies and thus $\E[M^2_{J_n}] = \E[J_n] \var[\tau_1]$. As a consequence,
since $\Var[\tau_1] = \frac{2}{3}(\E[\tau_1])^2$ and $\E[\tau_{J_n}] = \E[\tau_1]\E[J_n]$ by (\ref{jntaujtnexp}),
\begin{align}\label{jiikaksi}
\E[J_n^2] & = \frac{\E[J_n]\mathrm{Var}[\tau_1]}{(\E[\tau_1])^2} + \frac{2 \E[J_n \tau_{J_n}]}{\E[\tau_1]} - \frac{\E[\tau_{J_n}^2]}{(\E[\tau_1])^2} \nonumber\\
& = \tfrac{2}{3} \E[J_n] + \frac{2\E[J_n(\tau_{J_n}{-}\tn)]}{\E[\tau_1]} - \frac{\E(\tau_{J_n} {-} \tn)^2}{(\E[\tau_1])^2} \nonumber\\
& \quad \quad \quad \quad  + \frac{2\tn (\E[J_n]{-}\E[\tau_{J_n}]/\E[\tau_1])}{\E[\tau_1]} + \frac{\tn^2}{(\E[\tau_1])^2}\nonumber\\
& = \tfrac{2}{3} \E[J_n] + \frac{2\E[J_n(\tau_{J_n} {-} \tn)]}{\E[\tau_1]} - \frac{\E(\tau_{J_n} {-} \tn)^2}{(\E[\tau_1])^2} + \nt^2.
\end{align}
Denote
\begin{align}\label{triplet}
\alpha_{\nt} := \frac{2\E[J_n(\tau_{J_n} - \tn)]}{\E[\tau_1]\E[J_n]}, \, \beta_{\nt} & := \frac{\E(\tau_{J_n} - \tn)^2}{(\E[\tau_1])^2} \, \text{ and } \, \gamma_{\nt} := \E[J_n] - \nt - \tfrac{4}{3},
\end{align}
and observe that $\alpha_{\nt} \in [0, 4], \beta_{\nt} \in [0, 80/3]$, and $\abs{\gamma_{\nt}} \in [0, \tfrac{48}{\sqrt{\nt}}]$
by Proposition \ref{simplify} $(iii)$, (\ref{taujiin-teen}), and item $(i)$, respectively. In addition, by (\ref{jiikaksi}),
\begin{align*}
\E(J_n - \nt)^2 & = \E[J_n^2] - 2\nt \E[J_n] + \nt^2 \nonumber\\
& = \E[J_n](\tfrac{2}{3} - 2\nt + \alpha_{\nt}) - \beta_{\nt} + 2\nt^2 \nonumber\\
& = (\nt + \tfrac{4}{3})(\tfrac{2}{3} - 2\nt + \alpha_{\nt}) + \gamma_{\nt}(\tfrac{2}{3} - 2\nt + \alpha_{\nt}) - \beta_{\nt} + 2\nt^2\nonumber\\
& = \nt(\alpha_{\nt} - 2) + \tfrac{4}{3}(\tfrac{2}{3} + \alpha_{\nt}) + \gamma_{\nt}(\tfrac{2}{3} - 2\nt + \alpha_{\nt}) - \beta_{\nt}.
\end{align*}
In particular, using the above upper bounds for $\alpha_{\nt}, \beta_{\nt}$, and $\gamma_{\nt}$, and the fact that $\nt \geq 2$,
\begin{align}\label{alphaminus1}
\Big|\E\pr{\tfrac{J_n{-}\nt}{\sqrt{\nt}}}^2{-}(\alpha_{\nt}{-}2)\Big| & \leq \frac{1}{\sqrt{\nt}} \pr{\frac{56}{9}\frac{1}{\sqrt{\nt}} + \frac{48}{\nt}\pr{\frac{14}{3}{+}2\nt} + \frac{80}{3\sqrt{\nt}}} \leq \frac{C_1}{\sqrt{\nt}}
\end{align}
for some constant $C_1 > 0$.

Notice that $\tau_{J_n} \geq \tn$ $\p$-almost surely by the definition of $J_n$. Therefore, by (\ref{jntaujtnexp}), it holds that
$\E[\tau_1]\E[J_n] = \E[\tau_{J_n}] \geq \tn$, which yields
\begin{align*}
& \abs{(\alpha_{\nt} - 2) - \frac{2}{3}} = 2 \abs{\frac{\E[J_n(\tau_{J_n} - \tn)]}{\E[\tau_1]\E[J_n]} - \frac{4}{3}} \nonumber
\end{align*}
\begin{align}\label{alphanminus2}
& = \frac{2}{\E[\tau_1]\E[J_n]}\pr{\abs{\nt\E[\tau_{J_n} - \tn] + \E[(J_n - \nt)(\tau_{J_n}-\tn)] - \frac{4}{3}\E[\tau_1]\E[J_n]}} \nonumber\\
& \leq \frac{2}{\tn}\pr{\abs{\nt\E[\tau_{J_n} - \tn] - \frac{4}{3}\E[\tau_{J_n}]} + \E\big|(J_n - \nt)(\tau_{J_n}-\tn)\big|}.
\end{align}
By the relation $\theta_n = \frac{\nt T}{n}$ and Lemma \ref{PropJ},
\begin{align}\label{w_est_1}
\abs{\nt\E[\tau_{J_n} - \tn] - \frac{4}{3}\E[\tau_{J_n}]} & = \abs{\nt \pr{\E\pl{\tau_{J_n}} - \tn - \frac{4}{3}\frac{T}{n}} + \frac{4}{3}\pr{\tn - \E[\tau_{J_n}]}} \nonumber\\
& \leq \nt \abs{\E\pl{\tau_{J_n}} - \tn - \frac{4}{3}\frac{T}{n}} + \frac{4}{3}\abs{\tn - \E[\tau_{J_n}]}\nonumber\\
& \leq \nt \pr{\frac{48}{\sqrt{\nt}}\frac{T}{n}} + \frac{4}{3}\pr{\frac{4}{3}\frac{T}{n} + \frac{48}{\sqrt{\nt}}\frac{T}{n}}\nonumber\\
& = \frac{\tn}{\sqrt{\nt}}\pr{48 + \frac{16}{9}\frac{1}{\sqrt{\nt}} + \frac{64}{\nt}} \leq \frac{C_2 \tn}{\sqrt{\nt}}
\end{align}
for a constant $C_2 > 0$. Moreover, by H\"older's inequality, (\ref{CCC}), and (\ref{taujiin-teen}),
\begin{align}\label{w_est_2}
\E\big|(J_n {-} \nt)(\tau_{J_n}{-}\tn)\big| \leq \pr{\E \pr{J_n - \nt}^2 \E\pr{\tau_{J_n} - \tn}^2}^{1/2} \leq C_3\frac{\sqrt{\nt}T}{n} = \frac{C_3 \tn}{\sqrt{\nt}}
\end{align}
for some constant $C_3> 0$.
Consequently, by (\ref{alphaminus1}), (\ref{alphanminus2}), (\ref{w_est_1}), and (\ref{w_est_2}), it holds that
\begin{align*}
\abs{\E \pr{\tfrac{J_n - \nt}{\sqrt{\nt}}}^2 - \frac{2}{3}} = \abs{\E \pr{\tfrac{J_n - \nt}{\sqrt{\nt}}}^2 - (\alpha_{\nt} - 2)} + \abs{(\alpha_{\nt} - 2) - \frac{2}{3}} \leq \frac{C}{\sqrt{\nt}},
\end{align*}
where $C = C_1 + 2(C_2 + C_3) >0$.
\end{proof}
\begin{remark}
In the proof of \cite[Proposition 11.2]{Walsh}, an expression for $\alpha_{\nt}$ in (\ref{triplet}) is given based on the relation
\begin{align}\label{Walshfactorization}
\E[J_n \tau_{J_n}] = \E[J_n] \E[\tau_{J_n}].
\end{align}
However, we were not able to verify (\ref{Walshfactorization}), and thus had to use an estimate for $\alpha_{\nt}$ instead.
\end{remark}

\subsection{Tail behavior of \texorpdfstring{$\tau_{\nt}$ and $J_n - \nt$}{\tau_{\nt} and J_n}}
Lemmata \ref{Lemmatechnical} and \ref{J} below are essential for the proof of Proposition \ref{PropAzumaMoments}. 
\begin{lemma}\label{Lemmatechnical}
Under Assumption \ref{assumption}, suppose that $\nt \in 2\N$ and a constant $\xi > 0$ are such that $\nt \xi \in \N$. Then for every $\rho \in (0,\tfrac{\pi^2}{12} \xi \tn \sqrt{\nt})$ it holds that
\begin{align}
(i) & \quad \p \pr{\sqrt{\nt}(\tau_{\nt\xi} - \xi \tn) > \rho} \leq \exp \pr{-\tfrac{3}{2}\tfrac{\rho^2}{\xi\tn^2} H\pr{\sqrt{\tfrac{3 {\rho}}{\xi \tn \sqrt{\nt}}}}},\label{technicalbound1}\\
(ii) & \quad \p \pr{\sqrt{\nt}(\tau_{\nt\xi} - \xi \tn) < - \rho} \leq \exp \pr{-\tfrac{3}{2}\tfrac{\rho^2}{\xi \tn^2} H\pr{\sqrt{\tfrac{3 {\rho}}{\xi \tn \sqrt{\nt}}}}},\label{technicalbound2}
\end{align}
where the function $H : (0, \pi/2) \to \R$ is given by  
\begin{align}
\quad \quad H(x) := 1 + \tfrac{6}{x^4} \big(\tfrac{x^2}{2} + \log \cos x\big)\label{H}.
\end{align}
\end{lemma}
\begin{remark}
The above estimates are non-trivial only whenever $H$ is positive. Since $H(0+) = 1/2$, it holds that $H(x) > 0$ for small enough $x$.
Notice that the condition $\rho \in (0, \tfrac{\pi^2}{12} \xi \tn \sqrt{\nt})$ ensures that $\sqrt{\tfrac{3 {\rho}}{\xi \tn \sqrt{\nt}}} \in (0, \pi/2),$
which is the domain of $H$.
\end{remark}
\begin{proof}[Proof of Lemma \ref{Lemmatechnical}]
The proof uses ideas from the proof of \cite[Proposition 11.3]{Walsh}.

$(i)$: By Chebyshev's inequality, for any $\tilde{\lambda},\tilde{\rho} > 0$ it holds
\begin{align}\label{ekaestimaatti}
P_n^{(+)}(\tilde{\rho}) & := \p \pr{\sqrt{\nt}(\tau_{\nt\xi} - \xi \tn) > \tilde{\rho}} = \p \Big(e^{\tilde{\lambda} \sqrt{\nt} (\tau_{\nt{\xi}} - {\xi} \tn)} > e^{\tilde{\lambda} \tilde{\rho}}\Big) \nonumber\\
& \leq e^{-\tilde{\lambda} \tilde{\rho}} \E \pl{e^{\tilde{\lambda} \sqrt{\nt} ( \tau_{\nt {\xi}} - {\xi} \tn)}} = e^{-\tilde{\lambda} \tilde{\rho}} e^{-\tilde{\lambda} \sqrt{\nt} {\xi} \tn} \E[e^{\tilde{\lambda} \sqrt{\nt} \tau_1}]^{\nt {\xi}}
\end{align}
since $\tau_{\nt {\xi}}$ can be written as a sum of $\nt \xi$ independent random variables identically distributed as $\tau_1$ (see Subsection \ref{firstexittimes}).
In addition, since $\frac{h}{\sigma} = \sqrt{\smash[b]{\frac{T}{n}}}$ and $\theta_n = \frac{\nt T}{n}$, by (\ref{laplacecoshcos})
\begin{align}\label{levels}
e^{-\tilde{\lambda} \tilde{\rho}} e^{-\tilde{\lambda} \sqrt{\nt} {\xi} \tn} \E[e^{\tilde{\lambda} \sqrt{\nt} \tau_1}]^{\nt {\xi}} = e^{-\tilde{\lambda} \tilde{\rho}} \pr{e^{\tfrac{\tilde{\lambda} \tn}{\sqrt{\nt}}}\cos \pr{\sqrt{\tfrac{2 \tilde{\lambda}\tn}{\sqrt{\nt}}}}}^{-\nt {\xi}}
\end{align}
provided that $\tilde{\lambda} \in(0, \frac{\pi^2}{8T}\frac{n}{\sqrt{\nt}})$. Let $\rho \in (0, \tfrac{\pi^2}{12}\xi\tn \sqrt{\nt})$ and define
\begin{align}\label{lambda}
\lambda & := \tfrac{3\rho}{2 \xi\tn^2} \quad \text{ and } \quad \kappa := \sqrt{\tfrac{3 {\rho}}{\xi \tn \sqrt{\nt}}};
\end{align}
then $\lambda \in(0, \frac{\pi^2}{8T}\frac{n}{\sqrt{\nt}})$. Substitute $(\rho, \lambda) =  (\tilde{\rho},\tilde{\lambda})$ and use the relation 
$\kappa = \sqrt{ \frac{2 \lambda \tn}{\sqrt{\nt}}}$ in order to rewrite the right-hand side of (\ref{levels}) in terms of $(\rho, \lambda, \kappa, \xi)$:
\begin{align}\label{provinssitorstai}
e^{-\tilde{\lambda} \tilde{\rho}} \pr{e^{\tfrac{\tilde{\lambda} \tn}{\sqrt{\nt}}}\cos \pr{\sqrt{\tfrac{2 \tilde{\lambda} \tn}{\sqrt{\nt}}}}}^{-\nt {\xi}} = e^{-\lambda \rho} \pr{e^{\kappa^2/2} \cos \kappa}^{-\frac{4 \lambda^2 \tn^2 \xi}{\kappa^4}}.
\end{align}
Hence, by (\ref{ekaestimaatti}), (\ref{levels}), (\ref{provinssitorstai}), and finally by (\ref{lambda}),
\begin{align*}
\log P_n^{(+)}(\rho) & \leq -\lambda \rho - \tfrac{4 \lambda^2 \tn^2 \xi}{\kappa^4} \pr{\tfrac{\kappa^2}{2} + \log \cos \kappa} = -\tfrac{3}{2} \tfrac{\rho^2}{\xi \tn^2}\pr{1 + \tfrac{6}{\kappa^4} \pr{\tfrac{\kappa^2}{2} + \log \cos \kappa}}
\end{align*}
so that $P_n^{(+)}(\rho) \leq \exp \big(-\tfrac{3}{2}\tfrac{\rho^2}{{\xi} \tn^2} H(\kappa)\big)$ in terms of $H$ defined in (\ref{H}).

$(ii)$: By Chebyshev's inequality, for any $\tilde{\lambda}, \tilde{\rho} > 0$ it holds that
\begin{align*}
P^{(-)}_{n}(\tilde{\rho}) & := \p\pr{\sqrt{\nt}(\tau_{\nt \xi} - \xi \tn) < -\tilde{\rho}} = \p\Big(e^{-\tilde{\lambda} \sqrt{\nt} ( \tau_{\nt {\xi}} - {\xi} \tn)} > e^{\tilde{\lambda} \tilde{\rho}} \Big) \nonumber\\
& \leq e^{-\tilde{\lambda} \tilde{\rho}} e^{\tilde{\lambda} \sqrt{\nt} \xi\tn} \E \pl{e^{-\tilde{\lambda} \sqrt{\nt} \tau_{\nt {\xi}}}} = e^{-\tilde{\lambda} \tilde{\rho}} \pr{e^{-\frac{\tilde{\lambda} \tn}{\sqrt{\nt}}}\cosh \pr{\sqrt{\tfrac{2 \tilde{\lambda}\tn}{\sqrt{\nt}}}}}^{-\nt {\xi}},
\end{align*}
by (\ref{laplacecoshcos}). Proceed as in $(i)$ with the same substitution ${(\tilde{\rho},\tilde{\lambda}) = (\rho,\lambda)}$ 
for $\rho \in (0, \tfrac{\pi^2}{12}\xi\tn \sqrt{\nt})$ and $\lambda$ given by (\ref{lambda}) to get as the counterpart of (\ref{provinssitorstai}),
\begin{align*}
e^{-\tilde{\lambda} \tilde{\rho}} \pr{e^{-\frac{\tilde{\lambda} \tn}{\sqrt{\nt}}}\cosh\Big(\sqrt{\tfrac{2 \tilde{\lambda}\tn}{\sqrt{\nt}}}\Big)}^{-\nt {\xi}} = e^{-\lambda \rho} \pr{e^{-\frac{\kappa^2}{2}}\cosh\pr{\sqrt{\tfrac{3 {\rho}}{\xi \tn \sqrt{\nt}}}}}^{-\frac{4 \lambda^2 \tn^2 {\xi}}{\kappa^4}}.
\end{align*}
Similarly as in $(i)$, one then shows that $P^{(-)}_{n}(\rho) \leq \exp \pr{-\tfrac{3}{2}\tfrac{\rho^2}{{\xi} \tn^2} \hat{H}\big(\sqrt{\tfrac{3 {\rho}}{\xi \tn \sqrt{\nt}}}\big)}$ in terms of the function
$\hat{H}(x) := 1 + \tfrac{6}{x^4} \pr{-\tfrac{x^2}{2} + \log \cosh x}$ on $(0, \pi/2)$. It remains to show that $\hat{H} \geq H$ on $(0,\pi/2)$. Since $\hat{H}(x) - H(x) = 6x^{-4}\pr{-x^2 + \log \cosh(x) - \log \cos(x)}$
it suffices to show that
\begin{align}\label{varphiiiassertion}
\varphi(x) := -x^2 + \log \cosh(x) - \log \cos(x) \geq 0.
\end{align}
First, $\varphi'(x) = -2x + \tanh(x) + \tan(x)$. Secondly,
\begin{align*}
\varphi''(x) = -2 + \frac{1}{\cosh^2(x)} + \frac{1}{\cos^2(x)} & = \frac{\cos^2(x) + \cosh^2(x) - 2 \cosh^2(x) \cos^2(x)}{\cos^2(x)\cosh^2(x)}\\
& > \frac{(\cos(x)- \cosh(x))^2}{\cos^2(x)\cosh^2(x)} > 0
\end{align*}
since $\cos(x)\cosh(x) < 1$ for $x \in (0,\pi/2)$. Hence, $\varphi'$ is increasing on $(0, \pi/2)$, and since it also holds that $\varphi'(0+) = 0 = \varphi(0+)$, (\ref{varphiiiassertion}) follows.
\end{proof}
We continue with the tail estimate for $J_n$. The result resembles inequality (42) in \cite{Walsh}, but the time-dependent setting causes some changes.
\begin{lemma}\label{J}
Under Assumption \ref{assumption}, suppose that $\nt \in 2\N$, $\delta \in (0, \tfrac{\pi^2}{12+\pi^2}),$ and let $H$ be as in (\ref{H}). Then
\begin{align*}
& (i) \quad \p(\jtn > \nt(1+\delta)) \leq \exp \Big(-\tfrac{3}{2} \tfrac{\nt \delta^2}{1+\delta} H\Big(\sqrt{\tfrac{3\delta}{1+\delta}}\Big)\Big) \quad \text{ if } \quad \nt(1+\delta) \in 2\N, \\
& (ii) \quad \p(\jtn < \nt(1-\delta)) \leq \exp \Big(-\tfrac{3}{2} \tfrac{\nt \delta^2}{1-\delta} H\Big(\sqrt{\tfrac{3\delta}{1-\delta}}\Big)\Big) \quad \text{ if } \quad \nt(1-\delta) \in 2\N.
\end{align*}
\end{lemma}
\begin{proof}
Fix $\nt \in 2\N$, $\delta \in (0, \tfrac{\pi^2}{12+\pi^2})$, and let $\rho := \delta \tn \sqrt{\nt}$.
For $(i)$, let $\xi := 1 + \delta$ and suppose that $\nt(1+\delta) = \nt \xi \in 2\N$. Then
\begin{align*}
\p(J_n > \nt \xi) & = \p(\tau_{\nt \xi} < \tn) = \p(\sqrt{\nt}(\tau_{\nt\xi} - \xi\tn) < -\rho) \leq \exp \pr{-\tfrac{3}{2}\tfrac{\rho^2}{\xi \tn^2} H\pr{\sqrt{\tfrac{3 {\rho}}{\xi \tn \sqrt{\nt}}}}}
\end{align*}
by (\ref{technicalbound2}), since the choice of $\delta$ ensures that the pair $(\xi, \rho)$ satisfies the assumptions of Lemma \ref{Lemmatechnical}. To show $(ii)$,
let $\xi := 1 - \delta$ and suppose that $\nt(1-\delta) = \nt \xi \in 2\N$. Then by (\ref{technicalbound1}),
\begin{align*}
\p(J_n < \nt \xi) & = \p(\tau_{\nt \xi} > \tn) = \p(\sqrt{\nt}(\tau_{\nt \xi} - \xi \tn) > \rho) \leq \exp \pr{-\tfrac{3}{2}\tfrac{\rho^2}{\xi\tn^2} H\pr{\sqrt{\tfrac{3 {\rho}}{\xi \tn \sqrt{\nt}}}}}
\end{align*}
since the pair $(\xi, \rho)$ satisfies the assumptions of Lemma \ref{Lemmatechnical} due to the choice of $\delta$.
\end{proof}

\subsection{Moment estimates for the difference \texorpdfstring{$J_n - \nt$}{J_n - \nt}}
To derive an estimate for $\E \abs{\jtn - \nt}^{K}$ for any $K > 0$, we recall (see e.g.~\cite[Theorem 14.12]{DasGupta}) a version of the Azuma--Hoeffding inequality.
\begin{proclaim}[Azuma--Hoeffding inequality]\label{PropAH}
Suppose that $(M_j)_{j=0,1, \dots}$ is a martingale with ${M_0 = 0}$. In addition, assume that for all $i \geq 1$ there exists a constant $\alpha_i > 0$ such that ${\abs{M_{i} - M_{i-1}} \leq \alpha_i}$ a.s. Then, for all $k \in \N$ and every $t > 0$,
$$\p(M_k \geq t) \leq \exp \Big({-}\tfrac{t^2}{2 \sum_{j=1}^{k} \alpha_i^2}\Big).$$
\end{proclaim}
For $t_0 = 0$, the following statement can be found in \cite[Proposition 11.2 $(iv)$]{Walsh}. The proof, however, does not cover
the case which corresponds to \eqref{sufficientHighmoments} below for the set $A_3$. We will prove here a time-dependent extension of
this statement.
\begin{proclaim}\label{PropAzumaMoments}
Suppose that Assumption \ref{assumption} holds, and let $K > 0$. Then there exists a constant $C_K > 0$ depending at most on $K$ such that
\begin{align}\label{CCC}
\E \abs{\jtn - \nt}^{K} \leq C_K \nt^{K/2} \quad \text{ for all } \, (n, t_0) \in 2\N {\times} [0,T).
\end{align}
\end{proclaim}
\begin{proof}
It suffices to prove the claim for $K \geq 2$, since the case $K \in (0,2)$ then follows by Jensen's inequality.
Since $\abs{\jtn - {\nt}}$ is a non-negative random variable,
\begin{align*}
\frac{1}{K} \E \abs{\jtn - {\nt}}^K & = \int_{0}^{\infty} z^{K-1} \p(\abs{\jtn - {\nt}} > z) dz.
\end{align*}
We show that there exist constants $C^{(1)}_K, C^{(2)}_K, C^{(3)}_K > 0$ corresponding to the sets $A_1 = (0,2], A_2 =(2, \nt]$ and $A_3 = (\nt, \infty)$
such that
\begin{align}\label{sufficientHighmoments}
I_k(\nt) := \int_{A_k} z^{K-1} \p(\abs{\jtn - {\nt}} > z) dz \leq C_{K}^{(k)} \nt^{K/2} \quad \text{ for all } \nt.
\end{align}
\newline
\textbf{Step 1:} Since $K \geq 2$ and $\nt \geq 2$, we have that
\begin{align*}
I_1(\nt) = \int_{0}^2 z^{K-1} \p(\abs{\jtn - {\nt}} > z) dz \leq \int_{0}^2 z^{K-1} dz \leq 2^K/K \leq C^{(1)}_K \nt^{K/2}.
\end{align*}
\newline
\textbf{Step 2:} Suppose that $\nt > 2$ and define $\delta_{\nt}(u) := \frac{2}{\nt} \floor{\frac{\nt u}{2}}$. Then
\begin{align}\label{iikaksiekaestimate}
I_2(\nt) & = \int_2^{\nt} z^{K-1} \p(\abs{J_n - \nt} > z) dz = \nt^{K} \int_{2/\nt}^{1} u^{K-1} \p(\abs{J_n - \nt} > \nt u) du \nonumber\\
& \leq \nt^{K} \int_{2/\nt}^{1} u^{K-1} \p(\abs{J_n - \nt} > \delta_{\nt}(u) \nt) du.
\end{align}
Fix a constant $a \in (0,1]$ small enough such that for every $m \in \N$,
\begin{align}\label{conditionsona}
\delta_{m}(u) < \tfrac{\pi^2}{12 +\pi^2} \quad \text{ and } \quad H\Big(\sqrt{\tfrac{3\delta_{m}(u)}{1+\delta_{m}(u)}}\Big) \wedge H\Big(\sqrt{\tfrac{3\delta_{m}(u)}{1-\delta_{m}(u)}}\Big) > 1/4 \quad \text{hold for all } \quad u \leq a,
\end{align}
where the function $H$ is defined below in (\ref{H}). Depending on the value of $\nt$, we split the right-hand side of (\ref{iikaksiekaestimate}) into the sum of the integrals
\begin{align*}
I_{2,1}(\nt) & := \nt^{K}\int_{2/\nt}^{a} u^{K-1}\p(\abs{J_n - \nt} > \delta_{\nt}(u) \nt)du \quad (\text{for } a > 2/\nt, \text{ otherwise } 0),\\
I_{2,2}(\nt) & := \nt^{K}\int_{a \vee (2/\nt)}^{1} u^{K-1} \p(\abs{J_n - \nt} > \delta_{\nt}(u) \nt) du.
\end{align*}
If $a \in (2/\nt, 1)$, by (\ref{conditionsona}) and the fact that $\nt(1+\delta_{\nt}(u))$ and $\nt(1-\delta_{\nt}(u))$ are even integers, we may apply Lemma \ref{J} and estimate
\begin{align}\label{iikaksyks}
I_{2,1}(\nt) & = \nt^{K} \int_{2/\nt}^{a} u^{K-1} \Big[\p\big(J_n > \nt(1+\delta_{\nt}(u))\big) + \p\big(J_n < \nt(1-\delta_{\nt}(u))\big)\Big] du \nonumber\\
& \leq \nt^{K} \int_{2/\nt}^{a} u^{K-1} \pl{\exp \pr{-\tfrac{3}{8}\tfrac{\nt \delta_{\nt}^2(u)}{1+\delta_{\nt}(u)}}+ \exp \pr{-\tfrac{3}{8}\tfrac{\nt \delta_{\nt}^2(u)}{1-\delta_{\nt}(u)}}}du \nonumber\\
& \leq 2 \nt^{K} \int_{2/\nt}^{a} u^{K-1} \exp \pr{-\tfrac{3}{8}\tfrac{\nt \delta_{\nt}^2(u)}{1+\delta_{\nt}(u)}}du.
\end{align}
By the properties of the floor function, for $u \in (0,1]$ it holds that
\begin{align}\label{floorpuljaus}
\frac{\delta_{\nt}^2(u)}{1+\delta_{\nt}(u)} = \frac{\big( \tfrac{2}{\nt} \floor{\tfrac{\nt u}{2}}\big)^2}{1 + \tfrac{2}{\nt} \floor{\tfrac{\nt u}{2}}} \geq \frac{\big( \tfrac{2}{\nt} \pr{\tfrac{\nt u}{2} - 1}\big)^2}{1+u} > \frac{\big(u-\tfrac{2}{\nt}\big)^2}{2},
\end{align}
and thus the right-hand side of (\ref{iikaksyks}) can be bounded from above by
\begin{align}\label{gaussianintegraali}
& 2 \nt^{K} \int_{2/\nt}^{a} u^{K-1} e^{-\tfrac{3\nt(u-2/\nt)^2}{16}}du \leq 2 \nt^{K} \int_{0}^{1-2/\nt} \big(u+\tfrac{2}{\nt}\big)^{K-1} e^{-\frac{3\nt u^2}{16}} du\nonumber\\
& \leq 2^{K-1} \nt^{K} \int_{0}^{1-2/\nt} \pr{u^{K-1} + (\tfrac{2}{\nt})^{K-1}}e^{-\frac{3\nt u^2}{16}} du.
\end{align}
By substituting $x = u \sqrt{\nt}$ and identifying the right-hand side of (\ref{gaussianintegraali}) as an integral with respect to a Gaussian measure,
it can be verified that this integral multiplied by $\nt^{K/2}$ is bounded by some constant $\tilde{C}^{(2,1)}_K>0$. Hence,
$I_{2,1}(\nt) \leq C^{(2,1)}_K \nt^{K/2}$ for all $\nt$, where ${C^{(2,1)}_K > 0}$ depends only on $K$.

Let us then consider the integral $I_{2,2}(\nt)$. If $a/3 \leq 2/\nt$, then $\nt \leq 6/a$, $a \leq a \vee (2/\nt)$, and thus $I_{2,2}(\nt) \leq 6^K(K a^K)^{-1}.$
On the other hand, if $a/3 \in (2/\nt, 1)$, by {Lemma \ref{J}},
\begin{align*}
I_{2,2}(\nt) & = \nt^{K}\int_{a}^{1} u^{K-1} \Big[ \p\big(J_n > \nt(1+ \delta_{\nt}(u))\big) + \p\big(J_n < \nt(1-\delta_{\nt}(u))\big)\Big] du\\
& \leq \nt^{K} \int_{a}^{1} u^{K-1} \Big[ \p\big(J_n > \nt(1+ \delta_{\nt}(a/3))\big) + \p\big(J_n < \nt(1-\delta_{\nt}(a/3))\big)\Big] du\\
& \leq \nt^{K} \Big[ \exp \pr{-\tfrac{3}{8} \tfrac{\nt \delta_{\nt}(a/3)^2}{1+\delta_{\nt}(a/3)}} +\exp \pr{-\tfrac{3}{8} \tfrac{\nt \delta_{\nt}(a/3)^2}{1-\delta_{\nt}(a/3)}}\Big] \int_{a}^{1} u^{K-1}  du \quad \quad \quad\\
& \leq 2 \nt^{K} \exp \pr{-\tfrac{3}{8} \tfrac{\nt \delta_{\nt}(a/3)^2}{1+\delta_{\nt}(a/3)}} \int_{a}^{1} u^{K-1}du\\
& \leq 2 \nt^{K} \exp \pr{-\tfrac{3\nt(a/3-2/\nt)^2}{16}}.
\end{align*}
Here we used the fact that $u \mapsto \delta_{\nt}(u)$ is nondecreasing, that $\nt(1{+}\delta_{\nt}(a/3))$ and $\nt(1{-}\delta_{\nt}(a/3))$ are (even) integers, condition (\ref{conditionsona}), and inequality (\ref{floorpuljaus}). Notice that the right-hand side converges to zero as $\nt \to \infty$. Consequently, there exists a constant $C^{(2,2)}_K > 0$ such that $I_{2,2}(\nt) \leq C^{(2,2)}_K$ for all $\nt$, and 
(\ref{sufficientHighmoments}) for $k = 2$ follows.

\noindent \textbf{Step 3:} To estimate $I_3(\nt)$, we apply the Azuma--Hoeffding inequality to the tail distribution of the random variable $J_n = \inf \set{2m \in 2\N: \tau_{2m} > \tn}$. Recall that $\tau_{i} - \tau_{i-1}, i = 1,2, \dots,$ are i.i.d.~(see Subsection \ref{firstexittimes}) and that $\frac{n}{T}\theta_n = \nt$ according to (\ref{deftnnt}). Let $\zeta_i := \frac{n}{T}(\tau_i - \tau_{i-1})$, $i \geq 1$. Then, for all $m \in \N$, we have
\begin{align}\label{developtjn}
\p(\jtn \geq 2m) & = \p(\tau_{2m-2} \leq \tn) = \p\Big(\sum_{i=1}^{2m-2} \zeta_i \leq \tfrac{n}{T}\theta_n\Big) \leq \p\Big(\sum_{i=1}^{2m-2} \zeta_i \wedge N \leq \nt\Big) \nonumber\\
& = \p\Big(\sum_{i=1}^{2m-2}(c_N - \zeta_i \wedge N) \geq (2m{-}2)c_N -\nt\Big),
\end{align}
where $N \in \N$ is chosen such that ${3/4 < c_N := \E [\zeta_i \wedge N] < \E[\zeta_i] = 1}$.
Then $\abs{\E \pl{\zeta_i \wedge N } - \zeta_i \wedge N} \leq N$ for all $i \geq 1$, and by (\ref{developtjn}) and the Azuma--Hoeffding inequality (Proposition \ref{PropAH}),
\begin{align*}
\p(\jtn \geq 2m) \leq \exp \pr{-\tfrac{((2m-2)c_N - \nt)^2}{2(2m-2)N^2}}, \quad m \in \N.
\end{align*}
Since $J_n > 0$, we have
\begin{align*}
I_3(\nt) & = \int_{\nt}^{\infty} z^{K-1} \pp{\jtn - {\nt} \geq z} dz \nonumber\\
& = \int_{\nt}^{2\nt+2} z^{K-1} \pp{J_n \geq z + \nt} dz + \sum_{m=3}^{\infty} \int_{(m-1)\nt + 2}^{m \nt + 2} z^{K-1} \p(J_n \geq z + \nt)dz\\
& \leq \int_{\nt}^{2\nt+2} z^{K-1} \pp{J_n \geq 2\nt} dz + \sum_{m=3}^{\infty} \int_{(m-1)\nt + 2}^{m \nt + 2} z^{K-1} \p(J_n \geq m\nt + 2)dz\\
& \leq \int_{\nt}^{2\nt{+}2} z^{K-1} e^{-\tfrac{\nt}{2N^2}\tfrac{((2-2/\nt)c_N - 1)^2}{(2-2/\nt)}} dz + \sum_{m=3}^{\infty} \int_{(m{-}1)\nt{+}2}^{m \nt{+}2} z^{K-1} e^{-\tfrac{\nt}{2N^2}\tfrac{(mc_n - 1)^2}{m}}dz\\
& =: I_{3,1}(\nt) + I_{3,2}(\nt).
\end{align*}
Since $c_N \in (3/4,1)$, there exist constants $c, c' > 0$ such that
\begin{align*}
I_{3,1}(\nt) & = \int_{\nt}^{2\nt{+}2} z^{K-1} e^{-\tfrac{\nt}{2N^2}\tfrac{((2-2/\nt)c_N - 1)^2}{(2-2/\nt)}} dz \leq 2(2\nt+2)^{K-1} e^{-\tfrac{\nt c}{N^2}} \leq C_K^{(3,1)},\\
I_{3,2}(\nt) & = \sum_{m=3}^{\infty} \int_{(m{-}1)\nt{+}2}^{m \nt{+}2} z^{K-1} e^{-\tfrac{\nt}{2N^2}\tfrac{(mc_N - 1)^2}{m}}dz \leq \sum_{m=3}^{\infty} (m\nt + 2)^K e^{-\tfrac{\nt m c'}{N^2}} \leq C_K^{(3,2)},
\end{align*}
where $C_K^{(3,1)}, C_K^{(3,2)} > 0$ depend at most on $K$. This proves (\ref{sufficientHighmoments}) for ${k = 3}$.
\end{proof}
\appendix
\section{Appendix}
\subsection{The class $\gbv$}\label{appendixGBV}
For a function $g : \R \to \R$, let 
$$T_g(x) := \sup  \set{\sum_{i=1}^{N} \abs{g(x_i) - g(x_{i-1})}, \, N \in \N, \, -\infty < x_0 < x_1 < \dots < x_N = x}.$$
If $\lim_{x \to \infty} T_g(x) < \infty$, the function $g$ is said to be of bounded variation. The class of functions with this property will be denoted by $BV$.

A function $g \in BV$ is by definition a bounded function. An error estimation carried out merely for the class $BV$ would rule out e.g.~polynomials, which on the other hand have bounded variation on every compact interval. Therefore, instead of the class $BV$, we consider a class of functions of generalized bounded variation allowing exponential growth. In order to find an applicable representation for a large class of such functions, we will follow the presentation given in \cite{Avikainen}.

Recall the class $\mathcal{M}$ given by Definition \ref{class-of-set-functions}, which consists of set functions $\mu$ (acting on bounded Borel sets on $\R$) that can be written as a difference of two measures $\mu^{1}, \mu^{2} : \B(\R) \to [0,\infty]$ such that $\mu^1(K)$ and $\mu^2(K)$ are finite for all compact sets $K \in \B(\R)$. In \cite[Theorem 3.3]{Avikainen} it is proved that such a decomposition can be chosen to be orthogonal and minimal: There exists a unique pair of measures $\mu^+,\mu^-$ on $\B(\R)$ such that $\mu^+$ and $\mu^{-}$ are mutually singular, and $\mu^{+} \leq \mu^{1}$ and $\mu^{-} \leq \mu^{2}$ hold for all the other decompositions $\mu = \mu^1 - \mu^2$. Even though $\mu \in \mathcal{M}$ is not itself a signed measure (it is undefined on unbounded sets), the aforementioned result, based on the Hahn decomposition theorem, allows us to define the total variation measure
associated to $\mu$ by setting
$$\abs{\mu} : \B(\R) \to [0, \infty], \quad \abs{\mu} := \mu^+ + \mu^-.$$
Consequently, the integral in (\ref{bump-function-condition}) appearing in Definition \ref{DefinitionGBV} of the class $\gbv$ is defined. 

The inclusion $BV \subset \gbv$ follows as a special case of the result \cite[Theorem 4.3]{Avikainen}.
\begin{remark}[Polynomials are contained in $\gbv$]\label{polyonomials_are_included} 
To show that every polynomial $f(x) = \sum_{k=0}^{N} a_k x^k$, $a_k \in \R, N \in \N$ belongs to the class $\gbv$, let
$$c = a_0, \quad d \mu = \sum_{k=1}^{N} ka_k x^{k-1} dx, \quad \text{ and } \quad \mathcal{J} = \emptyset$$
to be the parameters appearing in the representation (\ref{GBVdefrep}) for the function $f$. It remains to notice that this $\mu$ satisfies the condition (\ref{bump-function-condition}), since for every $\beta > 0$, $$\int_{\R} e^{-\beta\abs{x}} d|\mu|(x) = \int_{\R} e^{-\beta\abs{x}} \bigg|\sum_{k=1}^N k a_k x^{k-1}\bigg| dx < \infty.$$
\end{remark}

\subsection{Auxiliary results for Section \ref{local-error-section}}
The following identities are applied in the proofs of Propositions \ref{Propgbvlinearized} and \ref{localerrorrepresentationPROP}.
\begin{lemma}
Let $h > 0$ and recall the operators $\varPi_e$ and $\varPi_o$ given by Definition \ref{Deflinearizationop}.
\begin{itemize}
\item[$(i)$] For all $\xi \in \R$, it holds that
\begin{align}\label{DiracRep}
\varPi_e \I_{\set{\xi}}(x) & = \frac{\I_{\{\xi \in \Z^h_e\}}}{4h}\big(\abs{x-(\xi{-}2h)} + \abs{x-(\xi{+} 2h)} - 2\abs{x-\xi}\big), \quad x \in \R.
\end{align}
\item[$(ii)$] If $y \in [2kh, (2k{+}2)h)$ for $k \in \Z$, then in terms of $d_o$ defined in (\ref{defofdode}),
\begin{align}\label{form}
& \abs{\varPi_o \varPi_e \I_{(y,\infty)}(x) - \varPi_e \I_{(y,\infty)}(x)} = \frac{d_o(x)}{4h} \I_{[(2k{-}1)h, (2k{+}3)h)}(x), \quad x \in \R.
\end{align}
\end{itemize}
\end{lemma}
\begin{proof}
$(i)$: It is obvious by the definition of $\varPi_e$ that ${\varPi_e \IN{\xi} \equiv 0}$ for $\xi \notin \Z^h_e$. If $\xi \in \Z^h_e$, then
\begin{align}\label{piecewiseVarPie}
\varPi_e \I_{\set{\xi}}(x) = \left\{ \begin{array}{ll}
\frac{x - (\xi{-}2h)}{2h}, & (\xi{-}2h) \leq x < \xi,\\
\frac{(\xi + 2h) - x}{2h}, & \xi \leq x < (\xi{+} 2h),\\
\end{array} \right.
\end{align}
and zero elsewhere, so it suffices to verify that (\ref{piecewiseVarPie}) agrees with the representation given in (\ref{DiracRep}).

$(ii)$: Suppose that $y \in [2kh, (2k{+}2)h)$ for some $k \in \Z$. One checks that
\begin{align}\label{alt_rep_varpie_ind}
\varPi_e \I_{(y, \infty)}(x) = \frac{1}{2} + \frac{1}{4h}\abs{x-2kh} - \frac{1}{4h} \abs{x-(2k{+}2)h}, x \in \R.
\end{align}
Then, by the linearity of $\varPi_o$ and by (\ref{alt_rep_varpie_ind}), we have for every $x \in \R$ that
\begin{align}\label{bothsidesofvarpiovarpie}
& \varPi_o \varPi_e \I_{(y,\infty)}(x) - \varPi_e \I_{(y,\infty)}(x) \nonumber\\
& = \frac{1}{4h} \big( \varPi_o \abs{\, \cdot \, - 2kh}(x) - \abs{x - 2kh}\big) - \frac{1}{4h} \big( \varPi_o \abs{\, \cdot \, - (2k{+}2)h}(x) - \abs{x - (2k{+}2)h}\big) \nonumber\\
& = \frac{d_o(x)}{4h} \pr{\I_{[(2k{-}1)h, (2k{+}1)h)}(x) - \I_{[(2k{+}1)h,(2k{+}3)h)}(x)},
\end{align}
since it holds for all $x \in \R$ and $m \in \Z$ that
\begin{align}\label{varpiocoincides}
\varPi_o \abs{\, \cdot \, -2mh}(x) - \abs{x - 2mh} = d_o(x)\I_{[(2m-1)h, (2m+1)h)}(x).
\end{align}
Taking the absolute values of both sides of (\ref{bothsidesofvarpiovarpie}) then completes the proof.
\end{proof}

\subsection{Auxiliary results for Section \ref{Sectionglobalerror}}
Under Assumption \ref{assumption}, recall from (\ref{pmfs}) the notation $P_{\nt + k}(x) = \p(X_{\tau_{\nt + k}} = hx)$ and $P^{J}_{\nt}(x) = \p(J_{n}{-}\nt = x)$, $x \in \Z$. Notice also that for all $k \in 2\N$,
\begin{align*}
P_{k}(x) = {k \choose \tfrac{k+x}{2}} 2^{-k}, \quad \quad x \in 2\Z, \quad \abs{x} \leq k.
\end{align*}
As in \cite{Walsh}, we define the 'effective order' of a monomial $\frac{k^p x^q}{n^r}$ with $p,q,r \in \N_{0}$ to be
$$\breve{O}\pr{\frac{k^p x^q}{n^r}} := \frac{p+q}{2} - r.$$
We will use the following result from \cite{Walsh} in the proof of Lemma \ref{Q3isok}.

\begin{proclaim}[{\cite[Proposition 11.5]{Walsh}}]\label{binomialratio} Let
\begin{align}
R : \mathcal{D}(R) \to \R, \quad &  R(n,k,x) := \frac{P_{n+k}(x)}{P_n(x)}, \nonumber\\
R^{(1)} : 2\N{\times}(2\Z)^n \to \R, \quad & R^{(1)}(n,k,x) := \frac{k}{2n} - \frac{3k^2 + 4kx^2}{8n^2} + \frac{3k^2 x^2}{4 n^3} - \frac{k^2 x^4}{8 n^4}, \label{formulaforR1}
\end{align}
where 
$$\mathcal{D}(R) := \big\{(n,k,x) \in 2\N {\times} (2\Z)^2: \abs{k} \vee \abs{x}  \leq n^{3/5}\big\}.$$ 
Then there exists a constant $C_0 > 0$, an integer $n_0$, and a finite sum $R^{(2)}$ of monomials of effective order at most $-3/2$ such that for all $(n,k,x) \in \mathcal{D}(R)$ with $n > n_0$,
\begin{align}
\abs{R(n,k,x) - [1 - R^{(1)}(n,k,x) + R^{(2)}(n,k,x)]} \leq C_0 n^{-3/2}. \label{charofR}
\end{align}
\end{proclaim}

\begin{lemma}\label{Q3isok}
Suppose that $g \in \EB$ and that $b \geq 0$ is as in (\ref{exponentialboundednesscondition}). Suppose also that $R^{(1)}$ is as in (\ref{formulaforR1}) and that $\Gamma_{\nt}$ is given by (\ref{GammaJoukko}). Then there exists a constant $C > 0$ such that for all $x_0 \in \R$ and $\nt$,
\begin{align*}
& (i) \quad \abs{\E\big[g(x_0{+}X_{\tau_{{\nt}}}) - g(x_0{+}X_{\tau_{J_n}}) ; \Gamma^{\complement}_{\nt}\big]} \leq C \nt^{-3/2} e^{b\abs{x_0} + b^2 \sigma^2 T + b\sigma \sqrt{2 T}},\\
& (ii) \quad \bigg|\E\big[g(x_0{+}X_{\tau_{{\nt}}}) {-} g(x_0{+}X_{\tau_{J}}) ; \Gamma_{\nt} \big]\\
& \quad \quad \quad \quad \, {-}\!\sum_{k=2-\nt}^{\infty} \sum_{x= -\nt}^{\nt} g(x_0{+}xh) P^{J}_{\nt}(k) P_{\nt}(x)R^{(1)}(\nt,k,x)\bigg| \leq C \nt^{-3/2} e^{b\abs{x_0} + b^2 \sigma^2 T}.
\end{align*}
\end{lemma}
\begin{proof}
$(i)$: Since 
$\Gamma^{\complement}_{\nt} \subset \big\{\big|X_{\tau_{\nt}}/h\big| > \nt^{3/5} \big\} \cup \big\{\abs{J_n - \nt} > \nt^{3/5}\big\},$
by H\"older's inequality, (\ref{EEE}), and (\ref{FFF}), there exists a constant $C' > 0$ such that
\begin{align*}
\abs{\E\big[g(x_0{+}X_{\tau_{{\nt}}}) - g(x_0{+}X_{\tau_{J_n}}) ; \Gamma^{\complement}_{\nt}\big]} \leq C' \nt^{-3/2}  \pr{\E\abs{g(x_0{+}X_{\tau_{{\nt}}}) - g(x_0{+}X_{\tau_{J_n}})}^2}^{1/2}.
\end{align*}
The claim follows, since by the triangle inequality, (\ref{AAA}), (\ref{BBB}), and the fact that $g \in \EB$, there exists another constant $\tilde{C} > 0$ such that
\begin{align*}
\pr{\E\abs{g(x_0{+}X_{\tau_{{\nt}}}) - g(x_0{+}X_{\tau_{J_n}})}^2}^{1/2} & \leq \tilde{C} e^{b\abs{x_0} + b^2 \sigma^2 T + b\sigma \sqrt{2 T}}.
\end{align*}

$(ii)$: This item will be proved in several intermediate steps.
\newline
\textbf{Step 1:}
Let us first show that there exists a constant $C > 0$ such that for all $x_0 \in \R$ and $\nt$,
\begin{align}
\bigg|\sum_{k=2-\nt}^{\infty} \sum_{x= -\nt}^{\nt} g(x_0{+}xh) P^{J}_{\nt}(k) P_{\nt}(x) R^{(2)}(\nt,k,x) \IN{\abs{x} \vee \abs{k} \leq \nt^{3/5}}\bigg| \leq C \nt^{-3/2} e^{b\abs{x_0} + b^2 \sigma^2 T} \label{double_sum_small_bdd},
\end{align}
where $R^{(2)}$ is as in Proposition \ref{binomialratio}. Using the relations $h = \sigma \sqrt{\smash[b]{\frac{T}{n}}}$, $\theta_n = \frac{\nt T}{2}$ and (\ref{pmfs}), it can be shown that for given integers $p,q,r \in \N_{0}$ and subsets $\Lambda_1, \Lambda_2 \subset \Z$,
\begin{align}\label{globercompformula}
& \sum_{k = 2-\nt}^{\infty} \sum_{x = -\nt}^{\nt} g(x_0{+}xh)P^{J}_{\nt}(k) P_{{\nt}}(x) \frac{k^p x^q}{{\nt}^r} \IN{x \in \Lambda_1, k \in \Lambda_2} \nonumber\\
& = \nt^{(p+q)/2 - r} \E \pl{\Big(\tfrac{X_{\tau_{{\nt}}}}{\sqrt{\sigma^2 \tn}}\Big)^q g(x_0{+}X_{\tau_{{\nt}}}); X_{\tau_{\nt}}/h \in \Lambda_1}\E \pl{\Big(\tfrac{J_n{-}{\nt}}{\sqrt{{\nt}}}\Big)^p; J_n{-}\nt \in \Lambda_2}.
\end{align}
By the definition of $R^{(2)}$, there exists an integer $N \in \N$, a vector $(a_i)_{i=1}^N \subset \R$, and $(p_i)_{i=1}^N,(q_i)_{i=1}^N,\\ (r_i)_{i = 1}^N \in \N_{0}^{N}$ such that $\frac{p_i+q_i}{2} - r_i \leq -3/2$ for all $1 \leq i \leq N$, and
$$R^{(2)}(\nt,k,x) = \sum_{i=1}^{N} a_i \frac{k^{p_i} x^{q_i}}{\nt^{r_i}} \quad \text{ for } \, (\nt,k,x) \in \mathcal{D}(R).$$
Therefore, by the relation (\ref{globercompformula}), the left-hand side of (\ref{double_sum_small_bdd}) can be rewritten and estimated by
\begin{align*}
& \abs{\sum_{i=1}^{N} a_i \nt^{\frac{p_i+q_i}{2} - r_i} \E \pl{\Big(\tfrac{X_{\tau_{{\nt}}}}{\sqrt{\sigma^2 \tn}}\Big)^{q_i} g(x_0{+}X_{\tau_{{\nt}}}); \big|X_{\tau_{\nt}}/h\big| \leq \nt^{3/5}}\E \pl{\Big(\tfrac{J{-}{\nt}}{\sqrt{{\nt}}}\Big)^{p_i}; \abs{J_n{-}\nt} \leq \nt^{3/5}}}\\
& \leq \nt^{-3/2} \sum_{i=1}^{N} \abs{a_i} \abs{\E \pl{\Big(\tfrac{X_{\tau_{{\nt}}}}{\sqrt{\sigma^2 \tn}}\Big)^{q_i} g(x_0{+}X_{\tau_{{\nt}}})}}\E \Big(\tfrac{\abs{J_n{-}{\nt}}}{\sqrt{{\nt}}}\Big)^{p_i}\\
& \leq \tilde{C} \nt^{-3/2} e^{b\abs{x_0} + b^2 \sigma^2 T},
\end{align*}
where $\tilde{C} > 0$ is a constant implied by (\ref{DDD}) and (\ref{CCC}). This proves (\ref{double_sum_small_bdd}).
\newline
\textbf{Step 2:}
Let us then show that for some constant $C > 0$ and for all $x_0 \in \R$ and $\nt \in 2\N$,
\begin{align}
\bigg|\sum_{k=2-\nt}^{\infty} \sum_{x= -\nt}^{\nt} g(x_0{+}xh) P^{J}_{\nt}(k) P_{\nt}(x)R^{(1)}(\nt,k,x) \IN{\abs{x} \vee \abs{k} > \nt^{3/5}}\bigg| & \leq C \nt^{-3/2} e^{b\abs{x_0} + b^2 \sigma^2 T} \label{gordonpart1}.
\end{align}
By (\ref{formulaforR1}) and (\ref{globercompformula}), it suffices to show that for $p,q,r \in \N_{0}$, there exists a constant $C_{p,q,r} > 0$ such that for all $x_0 \in \R$,
\begin{align*}
& \bigg|\sum_{k = 2-\nt}^{\infty} \sum_{x = -\nt}^{\nt} g(x_0{+}xh)P^{J}_{\nt}(k) P_{{\nt}}(x) \frac{k^p x^q}{{\nt}^r} \IN{\abs{x} \vee \abs{k} > \nt^{3/5}}\bigg| \leq C_{p,q,r} \nt^{-3/2} e^{b\abs{x_0} + b^2 \sigma^2 T}.
\end{align*}
We write $\{\abs{x} \vee \abs{k} > \nt^{3/5}\} = \{\abs{x} > \nt^{3/5}\} \cup {\{\abs{k} > \nt^{3/5}, \abs{x} \leq \nt^{3/5}\}}$ and consider the corresponding sums separately. By (\ref{globercompformula}) and H\"older's inequality,
\begin{align*}
& \abs{\sum_{k = 2-\nt}^{\infty} \sum_{x = -\nt}^{\nt} g(x_0{+}xh)P^{J}_{\nt}(k) P_{{\nt}}(x) \frac{k^p x^q}{{\nt}^r} \IN{\abs{x} > \nt^{3/5}}} \nonumber\\
& = \nt^{\frac{p+q}{2}-r} \abs{\E \pl{\Big(\tfrac{X_{\tau_{{\nt}}}}{\sqrt{\sigma^2 \tn}}\Big)^q g(x_0{+}X_{\tau_{{\nt}}}); \big|X_{\tau_{\nt}}/h\big| > \nt^{3/5}}\E \Big(\tfrac{J - {\nt}}{\sqrt{{\nt}}}\Big)^p} \nonumber\\
& \leq \nt^{\frac{p+q}{2}-r} \pr{\E \pl{g^2(x_0{+}X_{\tau_{\nt}})}}^{1/2}
\pr{\E \pl{\Big(\tfrac{|X_{\tau_{{\nt}}}|}{\sqrt{\sigma^2 \tn}}\Big)^{2q}; \big|X_{\tau_{\nt}}/h\big| > \nt^{3/5}}}^{1/2} \E \Big(\tfrac{\abs{J - {\nt}}}{\sqrt{{\nt}}}\Big)^p\nonumber \\
& \leq \tilde{C}_{p,q} \nt^{\frac{p+q}{2}-r} \pr{e^{2b|x_0|} \E \pl{e^{2b|X_{\tau_{\nt}}|}}}^{1/2} \p \pr{\big|X_{\tau_{\nt}}/h\big| > \nt^{3/5}}^{1/4} \nonumber\\
& \leq C_{p,q} \nt^{\frac{p+q}{2}-r} e^{b\abs{x_0} + b^2 \sigma^2 T} \p \pr{\big|X_{\tau_{\nt}}/h\big| > \nt^{3/5}}^{1/4}
\end{align*}
for some constants $\tilde{C}_{p,q}, C_{p,q} > 0$ implied by (\ref{GGG}), (\ref{CCC}), and (\ref{AAA}).
It remains to observe that by (\ref{EEE}), there exists a constant $C_{p,q,r} > 0$ such that
$$\nt^{\frac{p+q}{2}-r} \p\pr{\abs{X_{\tau_{\nt}}/h} > \nt^{3/5}}^{1/4} \leq C_{p,q,r} \nt^{-3/2}, \quad \nt \in 2\N.$$
The case of ${\{\abs{k} > \nt^{3/5}, \abs{x} \leq \nt^{3/5}\}}$ is similar: By (\ref{globercompformula}), (\ref{DDD}), (\ref{CCC}), and (\ref{FFF}), we find positive constants $\tilde{C}_{p,q}$, $\tilde{C}_{p,q,r}$ such that 
\begin{align*}
& \abs{\sum_{k = 2-\nt}^{\infty} \sum_{x = -\nt}^{\nt} g(x_0{+}xh)P^{J}_{\nt}(k) P_{{\nt}}(x) \frac{k^p x^q}{{\nt}^r} \IN{\abs{k} > \nt^{3/5}, \abs{x} \leq \nt^{3/5}}} \nonumber\\
& = \nt^{\frac{p+q}{2}-r} \abs{\E\pl{\Big(\tfrac{X_{\tau_{{\nt}}}}{\sqrt{\sigma^2 \tn}}\Big)^q g(x_0{+}X_{\tau_{\nt}}); \big|X_{\tau_{\nt}}/h\big| \leq \nt^{3/5}} \E \pl{\Big(\tfrac{J - {\nt}}{\sqrt{{\nt}}}\Big)^p; \abs{J - \nt} > \nt^{3/5}}} \nonumber\\
& \leq \nt^{\frac{p+q}{2}-r} \E\abs{\Big(\tfrac{X_{\tau_{{\nt}}}}{\sqrt{\sigma^2 \tn}}\Big)^q g(x_0{+}X_{\tau_{\nt}})} \pl{\E \pr{\tfrac{\abs{J - \nt}}{\sqrt{\nt}}}^{2p}}^{1/2} \p \pr{\abs{J - \nt} > \nt^{3/5}}^{1/2} \nonumber\\
& \leq \tilde{C}_{p,q,r} \nt^{-3/2} e^{b\abs{x_0} + b^2 \sigma^2 T}.
\end{align*}
\newline
\textbf{Step 3:} Since the processes $(\Delta \tau_k)_{k=1,2, \dots}$ and $(\Delta X_{\tau_{k}})_{k=1,2, \dots}$ are independent (see Subsection \ref{firstexittimes}), the random variable $J_n$ and the process $(X_{\tau_{k}})_{k=0,1, \dots}$ are also independent. Taking also into account that 
\begin{align*}
\text{supp} P_{\nt + k} & = \set{m \in 2\Z : \abs{m} \leq \nt + k} \quad (\text{for each } k \in 2\N),\\
\text{supp} P^{J}_{\nt} & = \set{m {-} \nt: m \in 2\N},
\end{align*}
it can be shown that
\begin{align}\label{GammaNhommeli}
& \E\big[g(x_0{+}X_{\tau_{{\nt}}}) - g(x_0{+}X_{\tau_{J_n}}) ; \Gamma_{\nt} \big] \nonumber\\
& = \sum_{k=2-\nt}^{\infty} \sum_{x= -\nt}^{\nt} g(x_0{+}xh) P^{J}_{\nt}(k) P_{{\nt}}(x)\bigg(1 - \frac{P_{{{\nt}}+k}(x)}{P_{{\nt}}(x)}\bigg) \IN{\abs{x} \vee \abs{k} \leq \nt^{3/5}} \nonumber\\
& = \sum_{k=2-\nt}^{\infty} \sum_{x= -\nt}^{\nt} g(x_0{+}xh) P^{J}_{\nt}(k) P_{{\nt}}(x) \big(1-R(\nt, k, x)\big) \IN{\abs{x} \vee \abs{k} \leq \nt^{3/5}}.
\end{align}
By (\ref{charofR}), (\ref{double_sum_small_bdd}) and (\ref{GammaNhommeli}), there exist constants $C_0, C_1 > 0$ and an integer ${n_0 \in 2\N}$ such that
whenever $\nt > n_0$,
\begin{align}\label{gordonpart2}
& \bigg|\sum_{k=2-\nt}^{\infty} \sum_{x= -\nt}^{\nt} g(x_0{+}xh) P^{J}_{\nt}(k) P_{\nt}(x)R^{(1)}(\nt,k,x)\IN{\abs{x} \vee \abs{k} \leq \nt^{3/5}} \nonumber\\
& \quad \quad \quad \quad - \E\big[g(x_0{+}X_{\tau_{{\nt}}}) - g(x_0{+}X_{\tau_{J_n}}) ; \Gamma_{\nt} \big]\bigg| \nonumber\\
& = \bigg|\sum_{k=2-\nt}^{\infty} \sum_{x= -\nt}^{\nt} g(x_0{+}xh) P^{J}_{\nt}(k) P_{\nt}(x)\big(R^{(1)}(\nt,k,x) - [1-R(\nt,k,x)]\big)\IN{\abs{x} \vee \abs{k} \leq \nt^{3/5}}\bigg|\nonumber\\
& \leq C_0 \nt^{-3/2} \sum_{k=2-\nt}^{\infty} \sum_{x= -\nt}^{\nt} \abs{g(x_0{+}xh)} P^{J}_{\nt}(k) P_{\nt}(x) \IN{\abs{x} \vee \abs{k} \leq \nt^{3/5}}\nonumber\\
& \quad \quad + \bigg| \sum_{k=2-\nt}^{\infty} \sum_{x= -\nt}^{\nt} g(x_0{+}xh) P^{J}_{\nt}(k) P_{\nt}(x)R^{(2)}(\nt,k,x) \IN{\abs{x} \vee \abs{k} \leq \nt^{3/5}}\bigg|\nonumber\\
& \leq C_0\nt^{-3/2} \E \pl{\big|g(x_0{+}X_{\tau_{{\nt}}})\big|; \big|X_{\tau_{\nt}}/h\big| \leq \nt^{3/5}} + C_1 \nt^{-3/2} e^{b\abs{x_0} + b^2 \sigma^2 T} \nonumber\\
& \leq C_2 \nt^{-3/2} e^{b\abs{x_0} + b^2 \sigma^2 T}.
\end{align}
for some constant $C_2 > 0$ implied by \eqref{DDD}. Here we applied (\ref{globercompformula}) and (\ref{double_sum_small_bdd}) for the second last inequality. Consequently, we get the claim for all $\nt > n_0$ by the triangle inequality, (\ref{gordonpart1}), and (\ref{gordonpart2}). By letting
$$M := \sup_{(n,k,x): n \leq n_0} \abs{\big(R^{(1)}(n,k,x) - [1-R(n,k,x)]\big)\IN{\abs{x} \vee \abs{k} \leq \nt^{3/5}}} < \infty,$$
for $\nt \leq n_0$ we find another constant $C_3 = C_3(n_0) > 0$ such that

\begin{align}\label{bigMsmallnt}
& \bigg|\sum_{k=2-\nt}^{\infty} \sum_{x= -\nt}^{\nt} g(x_0{+}xh) P^{J}_{\nt}(k) P_{\nt}(x)\big(R^{(1)}(\nt,k,x) - [1-R(\nt,k,x)]\big)\IN{\abs{x} \vee \abs{k} \leq \nt^{3/5}}\bigg|\nonumber\\
& \leq M \sum_{k=2-\nt}^{\infty} \sum_{x= -\nt}^{\nt} \abs{g(x_0{+}xh)} P^{J}_{\nt}(k) P_{\nt}(x) \IN{\abs{x} \vee \abs{k} \leq \nt^{3/5}}\nonumber\\
& \leq M \E\pl{\big|g(x_0{+}X_{\tau_{\nt}})\big|; \big|X_{\tau_{\nt}}/h\big| \leq \nt^{3/5}} \p\pr{\abs{J_{\nt} - \nt} \leq \nt^{3/5}} \nonumber\\
& \leq C_3 \nt^{-3/2}e^{b\abs{x_0} + b^2 \sigma^2 T}
\end{align}
by (\ref{DDD}). Combine (\ref{gordonpart1}), (\ref{gordonpart2}), and (\ref{bigMsmallnt}) to complete the proof.
\end{proof}

\noindent \textbf{Acknowledgement} The author was financially supported by the Magnus Ehrnrooth foundation during the preparation of this manuscript.


\begin{thebibliography}{99}\small
\bibitem{Avikainen} Avikainen, R.~(2009). On generalized bounded variation and approximation of SDEs. Preprint. Available at: \url{http://www.math.jyu.fi/research/pspdf/383.pdf}.
\bibitem{DasGupta} DasGupta, A.~(2011). Probability for Statistics and Machine Learning: Fundamentals and Advanced Topics. Springer.
\bibitem{Dong} Dong, H., Krylov, N.~V.~(2005). Rate of convergence of finite-difference approximations for degenerate linear parabolic equations with $C^1$ and $C^2$ coefficients. \emph{Electron. J. Differ. Equ.} 2005(102):1--25.
\bibitem{GLS} Geiss, C., Luoto, A., Salminen, P. (2017). On first exit times and their means for
Brownian bridges. Preprint. Available at: \url{https://arxiv.org/abs/1711.06107}.
\bibitem{Heston} Heston, S., Zhou, G. (2000). On the rate of convergence of discrete-time contingent claims.\emph{ Math. Finance} 10(1):53--75.
\bibitem{Juncosa} Juncosa, M.~L., Young, D.~M.~(1953). On the order of convergence of solutions of a difference equation to a solution of the diffusion equation. \emph{SIAM J. Appl. Math.} 1(2):111--135.
\bibitem{LambertonRogers} Lamberton, D., Rogers, L.~C.~G.~(2000). Optimal stopping and embedding. \emph{J. Appl. Probab.}, 37(4):1143--1148.
\bibitem{Lindhagen} Lindhagen, L.~(2007). Finite difference equations and convergence rates in the central limit theorem. \emph{Probab. Math. Statist.}, 27(2):153--166.
\bibitem{RevuzYor} Revuz, D., Yor, M.~(1999). {\it Continuous Martingales and Brownian Motion}. 3rd ed., Springer.
\bibitem{Reynolds} Reynolds, A.~C.~(1972). Convergent finite difference schemes for nonlinear parabolic
equations. \emph{SIAM J. Numer. Anal.} 9(4):523--533.
\bibitem{RogersStapleton} Rogers, L.~C.~G., Stapleton, E. J.~(1997). Fast accurate binomial pricing. \emph{Finance Stochast.} 2(1):3--17.
\bibitem{Walsh} Walsh, J.~B.~(2003). The rate of convergence of the binomial tree scheme. \emph{Finance Stochast.}, 7(3):337--361.
\bibitem{Zhu} Zhu, S., Yuan, G., Sun, W. (2004). Convergence and stability of explicit/implicit schemes for parabolic equations with discontinuous coefficients. \emph{Int. J. Numer. Anal. Model.}, 1(2):131--145.
\end{thebibliography}
\end{document}